\documentclass{article}
\usepackage{amsmath}
\usepackage{amsthm}
\usepackage{amssymb}
\usepackage{hyperref}
\usepackage{verbatim}
\usepackage{graphicx}
\usepackage[usenames,dvipsnames,svgnames,table]{xcolor}
\usepackage{mathrsfs}  
\usepackage{accents} 

\usepackage[hmargin=3cm,vmargin=3.5cm]{geometry}
\def\a{\alpha}
\def\b{\beta}
\def\ga{\gamma}

\def\de{\delta}
\def\De{\Delta}
\def\ep{\epsilon}

\def\la{\lambda}
\def\La{\Lambda}
\def\si{\sigma}

\def\Om{\Omega}

\def\th{\theta}

\def\varep{\varepsilon}
\newcommand{\vphi}[0]{\varphi}
\newcommand{\vp}[0]{\varphi}

\def\BB{{\cal B}}

\def\DD{{\cal D}}

\def\FF{{\cal F}}

\def\II{{\cal I}}

\def\PP{{\cal P}}
\def\DD{{\cal D}}

\newcommand{\N}[0]{\mathbb{N}}
\newcommand{\F}[0]{\mathbb{F}}
\newcommand{\R}[0]{\mathbb{R}}
\newcommand{\Z}[0]{\mathbb{Z}}

\newcommand{\C}[0]{\mathbb{C}}

\newcommand{\T}[0]{\mathbb{T}}

\newcommand{\supp}{\mathrm{supp} \,}
\newcommand{\suppt}{\mathrm{supp}_t \,}

\newcommand{\fr}[2]{\frac{#1}{#2}}

\newcommand{\vect}[1]{\left[ \begin{array}{c} #1 \end{array} \right]}
\newcommand{\mat}[2]{\left[ \begin{array}{ #1} #2 \end{array} \right]}
\newcommand{\ALI}[1]{\begin{align*} #1 \end{align*}}
\newcommand{\tx}[1]{\mbox{#1}}

\newcommand{\lsm}[0]{\lesssim}

\newcommand{\wed}[0]{\wedge}
\newcommand{\pr}[0]{\partial}
\newcommand{\nb}{\nabla}

\newcommand{\co}[1]{\|#1\|_{C^0}}

\newcommand{\Ddt}[0]{\overline{D}_t}

\newcommand{\va}[0]{\vec{a}}

\newcommand{\vcb}[0]{\vec{b}}
\newcommand{\vcc}[0]{\vec{c}}

\DeclareMathAlphabet{\mathpzc}{OT1}{pzc}{m}{it}

\newcommand{\kk}[0]{\kappa}

\newcommand{\hc}[0]{\widehat{C}}

\newcommand{\VR}[0]{\mathring{V}}

\newcommand{\wtld}[1]{\widetilde{#1}}

\newcommand{\ali}[1]{ \begin{align} #1 \end{align} }

\def\XXint#1#2#3{{\setbox0=\hbox{$#1{#2#3}{\int}$}
     \vcenter{\hbox{$#2#3$}}\kern-.5\wd0}}

\newcommand{\calD}{\mathcal D}

\newcommand{\calF}{\mathcal F}

\newcommand{\calT}{\mathcal T}

\newcommand{\Comment}[1]{\begingroup\color{red} #1\endgroup} 

\newcommand{\vsq}[0]{\frac{|v|^2}{2}}  
\newcommand{\vprk}[0]{(v,p,R,\kk, \vp, \mu)}
\newcommand{\ost}[1]{\accentset{\ast}{#1}}

\newcommand{\dmd}[1]{\accentset{\diamond}{#1}}

\newcommand{\sqq}[1]{\frac{|#1|^2}{2}}
\newcommand{\DDR}[0]{\mathring{{\cal D}}}
\newcommand{\osdt}[0]{\ost{D}_t}
\newcommand{\cPI}[0]{{\mathcal P}_I}
\newcommand{\Fsupp}[0]{{{\mathcal F}\mbox{supp }}}
\newcommand{\hxii}[0]{\hat{\xi}}
\newcommand{\ostu}[0]{\ost{\tau}}
\newcommand{\htf}[0]{\hat{f}}
\newcommand{\tlf}[0]{\tilde{f}}

\newcommand{\osn}[0]{N}
\newcommand{\brh}[0]{\bar{H}}
\newcommand{\ostk}[0]{\ost{\kk}}
\newcommand{\ovl}[1]{\overline{#1}}
\newcommand{\undl}[1]{\underline{#1}}
\newcommand{\undvp}[0]{\underline{\vp}}
\newcommand{\mrg}[1]{\mathring{#1}}

\newtheorem{thm}{Theorem}
\newtheorem{lem}{Lemma}[section]

\newtheorem{prop}{Proposition}[section]

\newtheorem{conj}{Conjecture}

\theoremstyle{definition}
\newtheorem{defn}{Definition}[section]

\theoremstyle{remark}
\newtheorem*{rem}{Remark}

\title{ Nonuniqueness and existence of continuous, globally dissipative Euler flows }
\author{  Philip Isett\thanks{Department of Mathematics, University of Texas at Austin, Austin, TX.  The work of P. Isett is supported by the National Science Foundation under Award No. DMS-1402370.}
}
\date{ }

\begin{document}
\maketitle

\begin{abstract}
We show that H\"{o}lder continuous incompressible Euler flows that satisfy the local energy inequality (``globally dissipative'' solutions) exhibit nonuniqueness and contain examples that strictly dissipate kinetic energy.  The collection of such solutions emanating from a fixed initial data may have positive Hausdorff dimension in the energy space even if the local energy equality is imposed, and the set of initial data giving rise to such an infinite family of solutions is $C^0$ dense in the space of continuous, divergence free vector fields on the torus ${\mathbb T}^3$.    

The construction of these solutions involves a new and explicit convex integration approach mirroring Kraichnan's LDIA theory of turbulent energy cascades that overcomes the limitations of previous schemes, which had been restricted to bounded measurable solutions or to continuous solutions that dissipate total kinetic energy.  
\end{abstract}

\section{Introduction}
In this work we consider weak solutions to the incompressible Euler equations, which take the form
\ali{
\begin{split} \label{eq:theEulerEqn}
\pr_t v^\ell + \nb_j(v^j v^\ell) + \nb^\ell p &= 0 \\
\nb_j v^j &= 0,
\end{split}
}
where we use the summation convention to indicate a sum over the repeated spatial index $j$.

We will mainly consider solutions defined on a spatially periodic domain $\tilde{I} \times \T^d$, $\tilde{I}$ an open interval and $\T^d = (\R/\Z)^d$ a $d$-dimensional torus with $d \geq 3$.  A weak solution $(v,p)$ to \eqref{eq:theEulerEqn} may be defined as a locally square-integrable vector field $v : \tilde{I} \times \T^d \to \R^d$ (the velocity field) and a scalar field $p \in \DD'(\tilde{I}\times \T^d)$ (the pressure) that together satisfy \eqref{eq:theEulerEqn} in the sense of distributions.

A weak solution $(v,p)$ to \eqref{eq:theEulerEqn} here will be called {\it globally dissipative} if it is of class $v \in L_{t,x}^3$  and satisfies the {\it local energy inequality}, \eqref{eq:locEnInq} below, which is always interpreted in the sense of distributions\footnote{Note that $v \in L_{t,x}^3$ implies by Calder\'{o}n-Zygmund theory that $p \in L_{t,x}^{3/2}$, making the left hand side of \eqref{eq:locEnInq} a well-defined distribution.  Inequality \eqref{eq:locEnInq} is defined by formally integrating against an arbitrary non-negative, smooth function of compact support, and implies that the left hand side acts on such test functions as a Radon measure.}:
\ali{
\pr_t\left( \vsq \right) + \nb_j\left[ \left(\vsq + p\right) v^j \right] &\leq 0. \label{eq:locEnInq}
}
Globally dissipative weak solutions include all sufficiently regular solutions to \eqref{eq:theEulerEqn} including all solutions of class $(v,p) \in C^1(\tilde{I} \times \T^d)$, which satisfy the stronger condition of the {\it local energy equality}
\ali{
\pr_t\left( \vsq \right) + \nb_j\left[ \left(\vsq + p\right) v^j \right] &= 0.  \label{eq:locEnEq}
}
Solutions obeying the local energy equality also conserve total kinetic energy in the sense that their total kinetic energy $\int_{\T^d} |v|^2/2(t,x) dx$ is equal to a constant as a distribution in $t$.


The physical meaning of the local energy inequality \eqref{eq:locEnInq} is that kinetic energy cannot be created locally, although it may be allowed to dissipate.  This interpretation is most easily seen for solutions $(v,p)$ that are continuous, in which case inequality \eqref{eq:locEnInq} is equivalent to the requirement that the change in kinetic energy contained in a given region $\Om$ during a time interval $[t_1,t_2]$ cannot exceed the flux of energy carried by the incoming and outgoing fluid plus the work done by pressure.  
Namely, for any open $\Om \subseteq \T^d$ with smooth boundary, inward unit normal vector $n$, and $t_1 \leq t_2$, one has
\ali{
\int_\Om \fr{|v|^2}{2}(t_2, x) \tx{d}x - \int_\Om \fr{|v|^2}{2}(t_1,x)\tx{d}x &\leq \int_{t_1}^{t_2} \int_{\pr \Om} \fr{|v|^2}{2}(t,x) (v \cdot n) \tx{d}S \, \tx{d}t + \int_{t_1}^{t_2} \int_{\pr \Om} p(t,x) (v \cdot n) \tx{d}S \, \tx{d}t. \notag
}
In particular, taking $\Om = \T^d$, $\pr \Om = \emptyset$, the total kinetic energy of a globally dissipative weak solution must be nonincreasing in time.  

There has been significant motivation from both physical and mathematical points of view to determine whether continuous, globally dissipative Euler flows are uniquely determined by their initial data and whether there exist examples that strictly dissipate kinetic energy.  Sections~\ref{sec:motivation}-\ref{sec:prevResultsSurvey} below describe these motivations in detail.  However, the known approaches to addressing such existence and uniqueness questions, which all implement the method known as ``convex integration'', have faced limitations either to constructing continuous solutions that dissipate {\it total} kinetic energy but do not satisfy the local energy inequality \eqref{eq:locEnInq} on any open ball \cite{deLSzeCts,deLSzeHoldCts,buckDeLIsettSze,IO-presEn,BDLSVonsag} or to constructing bounded and measurable solutions that satisfy the local energy inequality or equality but are nowhere continuous \cite{deLSzeAdmiss}.

The present work provides the first construction of continuous, globally dissipative Euler flows.  The construction is achieved using a new convex integration approach for constructing solutions satisfying the local energy inequality, which overcomes the limitations of the previous approaches.  Among the new results we obtain are the following:
\begin{thm} \label{thm:locEnIneq} On $(0,\infty) \times \T^d$, for $d \geq 3$, there exist H\"{o}lder-continuous weak solutions $(v,p)$ that satisfy the local energy inequality \eqref{eq:locEnInq} with strict inequality holding everywhere on an open interval of time.
\end{thm}
\begin{thm} \label{thm:locEq}  On $\R \times \T^d$, $d \geq 3$, there exists $\a > 0$ and an uncountable family $(v_\b, p_\b)$ of solutions of class $v \in C^\a_{t,x}(\R \times \T^d)$ having no isolated points in $C_{t,x}^\a$ that all satisfy the local energy equality \eqref{eq:locEnEq} and all achieve the same initial data $v_\b(0,x) = v_0(x)$ at time $0$.
\end{thm}
Stronger forms of Theorems~\ref{thm:locEnIneq} and ~\ref{thm:locEq} are stated in Theorems~\ref{thm:enDispNunq}-~\ref{thm:enEqNunq} of Section~\ref{sec:mainResults} below, which include a stronger nonuniqueness result than stated above and are accompanied by an additional theorem on the density of ``wild'' initial data, Theorem~\ref{thm:wildInitData}.  The nonuniqueness results in these theorems are all proven by incorporating a use of randomness in the construction of the solutions, which represents the first use of a probabilistic method in the context of convex integration.  See Section~\ref{sec:mainResults} below for a more detailed discussion.

We now explain in greater detail the general physical and mathematical motivations surrounding the study of continuous, globally dissipative Euler flows and the local energy inequality.




\subsection{Motivation} \label{sec:motivation}

The importance of the local energy inequality arises from several contexts where the inequality \eqref{eq:locEnInq} is imposed as a criterion towards isolating the most physically relevant weak solutions to the incompressible Euler equations among solutions that may be singular.  Of particular importance is the zero viscosity limit of the incompressible Navier-Stokes equations\footnote{These generalize the incompressible Euler equations \eqref{eq:theEulerEqn} to include a forcing term $\nu \De v^\ell$ on the right hand side of \eqref{eq:theEulerEqn} that describes the internal friction in the fluid, with $\nu > 0$ the viscosity parameter.  For simplicity we will maintain focus in the discussion that follows on the incompressible Euler and Navier-Stokes equations without external forcing terms.} 
which has a substantial interest as it relates to the description of turbulence in fluid motion at high Reynolds numbers.  It is known\footnote{See \cite{scheffHausdorff} and \cite[Appendix]{caffKohnNir}.} that every divergence free initial datum in $L^2$ admits a global in time ``suitable'' weak solution to the 3D Navier-Stokes equations that obeys the corresponding local energy inequality for Navier-Stokes.  This inequality, which generalizes \eqref{eq:locEnInq}, plays a fundamental role in the partial regularity theory of suitable weak solutions to Navier-Stokes \cite{scheffHausdorff, caffKohnNir}.  
Furthermore, one can show that every vector field that arises as a limit in $L_{t,x}^3$ of suitable weak solutions to Navier-Stokes with viscosity tending to zero must be a weak solution to incompressible Euler that satisfies the local energy inequality \eqref{eq:locEnInq}.  
We will discuss such limits further in connection with Conjecture~\ref{conj:navStConj} below. 



A closely related context where conditions such as \eqref{eq:locEnInq} play a primary role is the theory of hyperbolic conservation laws.  
The theory is most successful in the setting of scalar conservation laws, where the class of ``admissible weak solutions'', which are bounded weak solutions that satisfy an entropy condition akin to \eqref{eq:locEnInq} for every convex function of the unknown scalar field, provides the appropriate setting for well-posedness of the initial value problem despite the presence of inevitable singularities in the solutions.  
The simplest example in this theory is the Burgers' equation in one spatial dimension
\ali{
\pr_t u + \pr_x(u^2/2) &= 0. \label{eq:burgers}
}
For scalar conservation laws in one spatial dimension, a weak solution is admissible if and only if it satisfies an entropy condition for at least one strictly convex entropy \cite{panov} (see also \cite{deLOtWdMinEnt,krupVassEntcond}).  For \eqref{eq:burgers} a sufficient entropy condition is given by the local energy inequality
\ali{
\pr_t(u^2/2) + \pr_x(u^3/3) &\leq 0 \label{eq:locEnIneqBurgers}
}
analogous to \eqref{eq:locEnInq}.  These entropy conditions impose an arrow of time on solutions to the conservation laws, which would otherwise be time-reversible, and are viewed in this context as mathematical expressions of the second law of thermodynamics.  For scalar conservation laws the class of admissible solutions also coincides with the class of weak solutions that arise as zero viscosity limits.  However, the theory for general systems of conservation laws faces great challenges towards obtaining results as strong as those described above.  As a general reference for these subjects we refer to \cite{dafermosBook}.

Further motivation for studying the existence and uniqueness of globally dissipative weak solutions to the incompressible Euler equations arises from the study of turbulence in fluid motion and the zero viscosity limit, which in addition motivates the problem of determining the maximal regularity for there can exist a weak solution that satisfies the local energy inequality and strictly dissipates kinetic energy.  
The well known foundational theory of Kolmogorov \cite{K41} postulates that the mean rate of kinetic energy dissipation in fully developed turbulence should be nonvanishing in the limit of zero viscosity, and predicts that mean fluctuations in velocity of nearby fluid elements scale in the bulk of the fluid as distance to the power 1/3 in the inertial range of length scales.  Onsager \cite{onsag} gave an independent derivation of this scaling law applied to the energy spectrum, and he proposed that the dissipation of energy independent of viscosity can be explained as the result of cascades of energy from lower to higher wavenumbers (or coarser to finer length scales) that are modeled by the nonlinear advective term of the incompressible Navier-Stokes and Euler equations.  Onsager observed that this mechanism for dissipation by energy cascades could in principle take place even without viscosity for the incompressible Euler equations \eqref{eq:theEulerEqn}, but that the notion of solution to the inviscid equations could not be interpreted in the classical sense, stating that an Euler solution with spatial H\"{o}lder regularity greater than $1/3$ must satisfy the conservation of kinetic energy.  The idea that the H\"{o}lder exponent $1/3$ marks the threshold regularity for conservation of energy in the incompressible Euler equations has been known in the mathematical literature as the Onsager conjecture and has inspired a surge of mathematical activity in recent years.  These works, to be surveyed further below, have led most recently to a positive resolution of Onsager's conjecture.  We refer to \cite{eyinkDissip,eyinkSreen, deLSzeOnsagSurv,drivasThesis} for more detailed discussions of Onsager's ideas above.

Analogous to the Onsager conjecture for the 3D incompressible Euler equations, we present here the following well known conjecture on the existence of globally dissipative incompressible Euler flows.  This conjecture and the question of uniqueness for the conjectured solutions are the primary focus of the paper.
\begin{conj}[Strong Onsager Conjecture] \label{conj:inviscConj}  There exists an open interval $I$, and a weak solution $(v,p)$ to the incompressible Euler equations on $I \times \T^3$ that is of class $v \in L_t^\infty C_x^{1/3}$ and satisfies the local energy inequality \eqref{eq:locEnInq} with the left hand side not identically zero.
\end{conj}
\noindent It is known from the results on the conservative direction of Onsager's conjecture \cite{eyink,cet,duchonRobert} that the H\"{o}lder exponent $1/3$ cannot be replaced by any greater value, and more generally that every weak solution to Euler of Besov class $L_t^3 B_{3,\infty}^{\a}$ with $\a > 1/3$ satisfies the local energy equality \eqref{eq:locEnEq}.

A notable motivation for the above conjecture arises from the following statement concerning {\it smooth} solutions to the the incompressible {\it Navier-Stokes} equations, which serves as a mathematical interpretation of some conclusions that may be drawn from Kolmogorov and Onsager's seminal papers.
\begin{conj}[K41-Onsager Conjecture for Navier-Stokes] \label{conj:navStConj} There exists a finite open interval $I$, and a sequence of suitable weak solutions to the Navier-Stokes equations $v_{\nu_j}$ on $I \times \T^3$ with viscosity parameters $\nu_j > 0$ tending to $0$ that is uniformly bounded in the $L_t^\infty C_x^{1/3}$ norm, and that dissipates kinetic energy at an average rate that is uniform in viscosity:
\ali{
\limsup_{j \to \infty} \fr{1}{|I|} \int_I \left[ - \fr{d}{dt} \int_{\T^3} \fr{|v_{\nu_j}|^2}{2}(t,x) \tx{d}x \right] \tx{d}t &\geq \varep > 0, \qquad \varep \in \R_+. \label{eq:anomalousDissip}
}
\end{conj}
The existence of dissipative Euler flows stated in the inviscid Conjecture~\ref{conj:inviscConj} is in fact a direct consequence of Conjecture~\ref{conj:navStConj} for the Navier-Stokes equations.  A globally dissipative weak solution to the Euler equations as described in the inviscid conjecture arises as a limit from the viscous conjecture by passing to a subsequence that converges strongly in $C_t L_x^2 \cap L_{t,x}^3$.  The existence of such a convergent subsequence can be easily shown using the Aubin-Lions-Simon lemma and the compactness of $C^{1/3}$ in $L^3(\T^3)$.  By the same argument, the regularity exponent $1/3$ in the conjecture cannot be replaced by any value greater than $1/3$, as the conservation of energy would then have to be satisfied by any limiting Euler flow.

Note that neither of the above conjectures makes assumptions on initial data for the sequence $v_{\nu_j}$ or for the Euler flow $v$, and consequently the conjectures do not imply the existence of blowup for classical solutions to the incompressible Euler equations.  Rather, in the conjectural picture the initial data for the Navier-Stokes solutions may be smooth but developing a Kolmogorov type energy spectrum, and thus converging to a rough vector field in the $0$ viscosity limit.  More generally, if $v_{\nu} \in L_{t,x}^2$ is any weak solution to the Navier-Stokes or Euler equations on $I \times \T^d$, $I$ an open interval, one does have the existence\footnote{This fact is usually discussed in the context of $L_t^\infty L_x^2$ solutions, in which case the initial data are also in $L_x^2$.  One may see this point by defining $v_\nu(t_0, \cdot)$ for each $t_0$ in the closure of $I$ to be the unique element of $\DD'(\T^3)$ for which
\ali{
v_\nu^\ell(t,x) &= v_\nu^\ell (t_0,x) + \int_{t_0}^t \left[ -\PP \nb_j( v_\nu^j(\tau, x) v_\nu^\ell(\tau, x)) + \nu  \De v_\nu^\ell(\tau, x) \right] d\tau \label{eq:defineInitialData}
} 
holds as a distribution in $(t,x)$ on $I \times \T^3$, where $\PP$ is the Leray-projection to divergence free vector fields.  The definition of a weak solution guarantees that the $v_\nu^\ell(t_0, \cdot)$ defined implicitly by \eqref{eq:defineInitialData} is independent of $t$ (i.e., $\pr_t[v_\nu^\ell(t_0,\cdot)] = 0$ weakly) and therefore may be identified with a unique distribution $v_\nu(t_0,\cdot) \in \DD'(\T^3)$ in the $x$ variable for each fixed $t_0$.}  of uniquely defined initial data
$v_{\nu}(t_0,\cdot) \in \DD'(\T^3)$, which converge weakly when the $v_\nu$ converge strongly in $L_{t,x}^2$.  The convergence takes place in stronger norms when the $v_\nu$ satisfy stronger assumptions.  The convergence of initial data in Conjecture~\ref{conj:navStConj} would be uniform as implied for bounded sequences in $C^{1/3}(\T^3)$.    

These conjectures may be generalized to other function space norms that reflect the regularity that is most appropriate, or to other viscous regularizations of Euler.  In the present formulation, the Navier-Stokes solutions in question must be smooth, classical solutions to the equations as they satisfy the Ladyzhenskaya-Prodi-Serrin criterion.  
Bounds in weaker function spaces, including the Besov space $L_t^3 B_{3,\infty}^{\sigma}$ for any $\si > 0$, are sufficient to guarantee the compactness in $L_{t,x}^3$ and, using the same argument, the existence of a subsequence converging to a weak solution to Euler obeying the local energy inequality.  For the existence of an incompressible Euler flow that dissipates total kinetic energy, it suffices to have a uniform, pointwise lower bound on the energy dissipation rate and strong convergence in $L_{t,x}^2$ along a subsequence, which would follow from a uniform in viscosity bound in $L_t^2 B_{2,\infty}^\si$ for some $\si > 0$.  In particular, the $5/3$-law for the energy spectrum of Kolmogorov, which corresponds to a regularity slightly stronger than the Besov space $B_{2,\infty}^{1/3}$, is more than sufficient for this compactness if it were to hold at least as an upper bound in the inviscid limit in the preceding sense.  It is an open question, however, to determine if general a priori estimates such as those above could be proven rigorously for weak solutions to the 3D Navier Stokes equations, and the known mathematical frameworks for establishing compactness in the zero viscosity limit based on the available a priori energy estimates (which include the dissipative solutions of \cite[Chapter 4.4]{lionsTopics} and the measure-valued solutions of \cite{dipMajdaOsc}) are not known to give rise to Euler flows.

\begin{rem}
The direct relationship between the energy spectrum of Kolmogorov and compactness in the inviscid limit discussed above was first observed in Theorem 5.1 of \cite{chenGlimm}, which contains further results that may be obtained under a weak Kolmogorov hypothesis, Assumption (K41W).  The approach via the Aubin-Lions-Simon lemma outlined above provides an alternative approach to their Theorem 5.1, as their Assumption (K41W) implies boundedness in $L_t^2 H_x^{\a}$ for some $\a > 0$ by \cite[Prop 3.1]{chenGlimm}.  

More recent papers addressing the relationship between turbulent statistics and energy dissipation in the inviscid  limit are \cite{constVicolHighReyn,drivasEyinksingLeray}.  The work \cite{constVicolHighReyn} considers important issues related to enstrophy and vorticity production at boundaries, and gives a convergence result under hypotheses imposed only in the inertial range of length scales.  In \cite{drivasEyinksingLeray}, it is shown that  ``quasi-singularities'' must arise in the inviscid limit for Leray-Hopf solutions to Navier-Stokes even if the rate of kinetic energy dissipation does not tend to a positive value as in \eqref{eq:anomalousDissip} but rather tends to zero slowly as $\nu \to 0^+$.  The authors review the existing numerical and experimental evidence for dissipation of energy in the inviscid limit in \cite[Remark 3]{drivasEyinksingLeray}.  They prove also further results on convergence in the inviscid limit including details of the Aubin-Lions-Simon lemma argument noted here.  See also \cite{vassilicos2015dissipation} for further review of the empirical literature on the energy dissipation rate.
\end{rem}



The problem of proving or disproving the existence of kinetic energy dissipation independent of viscosity as in \eqref{eq:anomalousDissip} for 3D Navier-Stokes is a well known open problem.  This situation is in stark contrast to the analogous statements for scalar conservation laws such as the Burgers' equation \eqref{eq:burgers}, where analogues of Conjectures~\ref{conj:inviscConj} and \ref{conj:navStConj} are well understood.  In this context, zero viscosity limits of bounded solutions are admissible solutions that typically exhibit strict dissipation of energy due to shocks, and their viscous approximations exhibit dissipation of energy independent of viscosity.  The energy dissipating, admissible, inviscid solutions have the regularity\footnote{We thank R. Shvydkoy for first pointing this out to us.  The regularity may be seen, as noted in \cite{feirGwSgwWcomp}, by interpolating the maximum principle $\| u \|_{L_{t,x}^\infty} \leq \| u_0 \|_{L^\infty}$ and the BV bound $\| u \|_{L_t^\infty TV_x} \leq \| u_0 \|_{TV}$.  The same may be stated for the viscous solutions (see \cite[Theorems 6.2.3, 6.3.2]{dafermosBook} and the proof of \cite[Theorem 6.2.6]{dafermosBook}).  The dissipation of energy independent of viscosity follows from the compactness of the viscous approximations in $L_{t,x}^2$ whenever the limiting entropy solution to Burgers exhibits strict dissipation of energy.} $L_t^\infty B_{3,\infty}^{1/3}$, which is the maximal possible regularity for energy dissipation in the scale of $L^3$ based Besov spaces (similarly to the Euler equations), and their viscous approximations remain bounded in $L_t^\infty B_{3,\infty}^{1/3}$ in the inviscid limit.  Inviscid limits in this setting are known to both exist and to be unique even for fixed initial data in $L^\infty$.  


\subsection{Previous related results and difficulties of previous approaches} \label{sec:prevResultsSurvey}
Although little progress has been made towards approaching Conjecture~\ref{conj:navStConj} for Navier-Stokes directly, developments in the method of convex integration based on pioneering work of De Lellis and Sz\'{e}kelyhidi \cite{deLSzeIncl,deLSzeAdmiss,deLSzeCts,deLSzeHoldCts} 
have yielded great advances in understanding energy non-conserving solutions to the incompressible Euler equations that are directly relevant to Conjecture~\ref{conj:inviscConj}.  The initial work of \cite{deLSzeIncl} produced $L_{t,x}^\infty$ solutions to the incompressible Euler equations with compact support, which improved and turned out to yield a more systematic approach to earlier breakthrough results of \cite{scheff,shnNonUnq}.  The method was extended in \cite{deLSzeAdmiss} to address admissibility criteria for the Euler equations, which included a proof of nonuniqueness of weak Euler flows of class $L_{t,x}^\infty$ obeying the local energy equality \eqref{eq:locEnEq}, and examples demonstrating that the inequality may be sharp for bounded solutions, before which only the dissipation of total kinetic energy for solutions in $L_t^\infty L_x^2$ was known due to \cite{shnDiss}.  The method of \cite{deLSzeAdmiss} achieved these results by proving not only nonuniqueness of globally dissipative solutions in $L_{t,x}^\infty$, but also a much stronger result\footnote{We have stated here a simplification of their result for the periodic setting.  See the proof of \cite[Theorem 1, (a),(b)]{deLSzeAdmiss} and \cite[Theorem 1, Proposition 2.1]{deLSzeAdmiss} for more general statements. } showing the following.
\begin{thm}[De Lellis, Sz\'{e}kelyhidi \cite{deLSzeAdmiss}] \label{thm:DeLSzeThm} Let $d \geq 2$.  Then there exists an infinite family of weak solutions to the incompressible Euler equations on $[0,T] \times \T^d$ of class $v \in L_{t,x}^\infty \cap C_t L^2_w$, that all share the same initial data $v(0,\cdot)$, and the same kinetic energy density $|v|^2/2(t,x) = e(t)$ a.e. on $(0,T) \times \T^d$, with $|v|^2(0,x) < e(0)$ for $a.e. \, x \in \T^d$, and all have common pressure given by $p(t,x) = -2e(t)/d$.  One may take $e(t) = 1$, or $e(t)$ to be strictly decreasing and smooth on $[0,T]$ with $e(T) = 0$.    
\end{thm}
The key point in this theorem with regard to the local energy inequality and equality is that one may arrange 
$\pr_t(|v|^2/2) + \nb_j( (|v|^2/2 + p) v^j ) = \pr_t e(t)$ to be either strictly negative or $0$.  Consequently, one has nonuniqueness for solutions obeying the local energy equality \eqref{eq:locEnEq}, and the existence of strict dissipation in the local energy inequality \eqref{eq:locEnInq} within the class of $L_{t,x}^\infty$ solutions.  The method of \cite{deLSzeIncl,deLSzeAdmiss} however faces substantial difficulties towards constructing any weak solutions that are continuous.  In the context of solutions that fail to conserve total kinetic energy, these difficulties were later overcome in \cite{deLSzeCts,deLSzeHoldCts} by developing a new version of convex integration for the incompressible Euler equations in 3D that is closely related to the work of Nash \cite{nashC1} on isometric embeddings of class $C^1$.  These works gave the first results towards Onsager's conjecture on the $1/3$ threshold regularity for conservation of energy, and led to a series of improvements and partial results \cite{isettThesis, deLSzeBuck,Buckmaster,buckDeLIsettSze,buckDeLSzeOnsCrit} including most recently a full proof of the conjecture by the author \cite{isettOnsag}; see also \cite{BDLSVonsag,isettEndpt}.  

Several works following \cite{deLSzeCts,deLSzeHoldCts} have focused on proving the existence of strict dissipation of total kinetic energy for solutions in this range of regularity.  The works \cite{deLSzeCts,deLSzeHoldCts} prove that the total kinetic energy of a $(1/10{-}\ep)$-H\"{o}lder Euler flow on $[0,T]\times \T^3$ may be equal to any prescribed, smooth, strictly positive function of time $e(t) > 0$.  This result is strengthened and generalized to higher regularity solutions in \cite{deLSzeBuck,IOnonpd,IO-presEn,BDLSVonsag}, where the last reference obtains this result for solutions having the regularity $(1/3{-}\ep)$ of the Onsager conjecture.  However, these works face substantial difficulties towards obtaining local dissipation of energy (see Section~\ref{sec:difficulties} below).

Another branch of works building on \cite{deLSzeCts,deLSzeHoldCts} have extended the method to obtain nonuniqueness results for weak solutions.  Nonuniqueness of continuous solutions was first addressed in \cite{isettThesis}, where it is shown that any smooth initial datum admits weak solutions of class $C_{t,x}^{1/5-\ep}$ that are identically constant outside of a finite interval that may be chosen arbitrarily small.  
The nonuniqueness of H\"{o}lder continuous Euler flows dissipating total kinetic energy or having a prescribed energy profile was proven in \cite{Daneri}, while in \cite{danSze} it is shown that there is an $L^2$ dense subset of ``wild'' divergence free initial data that admit infinitely ``admissible'' solutions of class $C_{t,x}^{1/5-\ep}$, thus extending a previous result for $L_{t,x}^\infty$ solutions of \cite{szeWiedYoung}.  In this work a solution is called ``admissible'' if its energy remains bounded by that of the initial data.  This definition of admissible is motivated by the weak-strong uniqueness theory for the Euler equations, which shows that weak (even measure-valued) solutions to incompressible Euler with energy bounded by their initial data must coincide with the classical solution obtaining the same initial data provided the latter exists (see \cite{brenDeLSzweakstrong} and \cite{wiedWeakStrongSurv}).

We note as well the recent work of \cite{buckVicweakNS} on weak solutions to the 3D Navier Stokes equations of class $L_t^2 H_x^\b$ for some $\b > 0$, which shows that such solutions may be nonunique, may have any prescribed smooth energy profile with compact support, and may generate any H\"{o}lder continuous weak solution to Euler as a zero viscosity limit in $C_tL_x^2$.  At present these solutions appear to be separate from the previous discussion as they do not satisfy the local energy inequality for Navier-Stokes that defines a suitable weak solution assumed in Conjecture~\ref{conj:navStConj}.

Finally, we remark that our work has led to a further improved H\"{o}lder exponent in \cite{deL2020non} and a generalization to the compressible Euler equations in \cite{giri2021non}.  Compared to these works, which draw on our main ideas, our work maintains some advantages.  Namely, we obtain solutions with compact forward in time support, we establish a stronger nonuniqueness result, and we prove a result on the density of wild initial data as a consequence of a general approximation theorem.  We also anticipate that our approach may be more possible to extend to the two-dimensional setting where Mikado flows are not applicable.

\subsubsection{Main difficulties} \label{sec:difficulties}

It must be emphasized that dissipation of {\it total} kinetic energy, though useful as a criterion for data that admit a classical or Lipschitz solution, would not be expected to provide a uniqueness criterion for weak solutions that exhibit strict dissipation of energy, whereas in contrast the local energy inequality may be considered for this purpose.  
This point may be seen most easily in the example of the Burgers equation \eqref{eq:burgers}, which admits the infinite family of compactly supported solutions on $[0,2]\times \R$ equal to
\ali{
u_\a(t,x) &= \a \cdot 1_{0 < x < \a t} + (x / t) \cdot 1_{\a t \leq x < t} + 1 \cdot 1_{t \leq x \leq 1 + t/2} + 0 \cdot 1_{x > 1 + t/2}, \quad 0 < t \leq 2,
}
for $0 < x < \infty$, and extended to be odd $u_\a(t,-x) = -u_\a(t,x)$ for $x < 0$.  For any $0 \leq \a \leq 1$, $u_\a(t,x)$ is a weak solution to \eqref{eq:burgers} with initial data $u_\a(0,x) = 1_{0 < x \leq 1} - 1_{-1 \leq x < 0}$, and for small values of $\a \leq (1/6)^{1/3}$, $u_\a$ satisfies the dissipation of total kinetic energy $\fr{d}{dt} \int_{\R} u_\a^2(t,x) dx \leq 0$.  In this case, the unique entropy solution  is the solution $u_0(t,x)$, which exhibits strict dissipation of energy $-\fr{d}{dt} \int_\R u_0^2(t,x) dx = 2/3$ and is the only bounded weak solution to Burgers for the given initial data that satisfies the local energy inequality \eqref{eq:locEnIneqBurgers}.  Here the failure of total kinetic energy dissipation to provide a uniqueness criterion may be explained by the fact that the entropy solution has strictly positive total energy dissipation, giving room for nearby solutions to exist such as the family $u_\a$ that exhibit a small amount of local energy creation but not strong enough to disturb the strict dissipation of total kinetic energy.  We are thus motivated to consider the local energy inequality \eqref{eq:locEnInq} for the Euler equations as a more stringent admissibility criterion for energy dissipating solutions with lower regularity.   The problem of constructing continuous solutions satisfying this inequality, which is motivated by the discussion of Conjectures~\ref{conj:inviscConj} and \ref{conj:navStConj}, would appear even more challenging from this vantage point in view of the fact that uniqueness holds for continuous weak solutions to Burgers (see, e.g. \cite{dafermosBook}).

Due to intrinsic differences in the setting of continuous solutions, the proof of Theorem~\ref{thm:DeLSzeThm} on $L_{t,x}^\infty$ solutions faces substantial difficulties towards being extended to constructing continuous, globally dissipative solutions or more generally towards a proof of Conjecture~\ref{conj:inviscConj}.  
Specifically the proof of Theorem~\ref{thm:DeLSzeThm} yields a series $v = \sum_q V_q$ where each $V_q$ is of size $1$ in $L^\infty_{t,x}$, the partial sums $v_k = \sum_{q \leq k} V_q$ are uniformly bounded in $L^\infty_{t,x}$ with energy densities $\fr{|v_k|^2}{2}(t,x)$ converging in $C_t L_x^p$ to the prescribed kinetic energy density for $p < \infty$, and the series converges only in $L_{t,x}^p$ for $p$ strictly less than $\infty$.  In contrast the construction of continuous and H\"{o}lder continuous solutions relies on convergence of the approximate solutions to be achieved {\it uniformly} at a rate compatible with the desired regularity.  
The known methods for producing H\"{o}lder continuous solutions lead to at best $C_t L_x^p$ approximation of the energy density (which follows from the estimates of \cite{IOnonpd,IO-presEn}) and not to uniform convergence.

\subsection{Main results and new ideas}\label{sec:mainResults}

The main achievement of this work is to develop an entirely different strategy for constructing solutions satisfying the local energy inequality that completely avoids the difficulties faced in generalizing the approach of Theorem~\ref{thm:DeLSzeThm}.  The strategy we develop in this work provides a significantly more explicit picture of how local energy dissipation may be achieved in the construction of solutions compared to the proof of Theorem~\ref{thm:DeLSzeThm}.  This construction has properties that are interesting from a physical point of view.  In particular, the arrow of time plays a fundamental role in the scheme, and the construction itself turns out to closely mirror a picture of turbulence theorized by Kraichnan, called the Lagrangian Direct Interaction Approximation \cite{kraichnan1965lagrangian}, which postulates that turbulent energy cascades are governed by {\it trilinear} interactions of waves that are carried by the mean velocity field.  

With this new strategy, we obtain the following results: 
\begin{thm}\label{thm:enDispNunq} For any $\a \in [0,1/15)$, there exists an infinite family of weak solutions $(v_\b, p_\b)$ to \eqref{eq:theEulerEqn} of class $v_\beta \in C_{t,x}^\a ( [0,\infty) \times \T^3)$ that satisfy the local energy inequality~\ref{eq:locEnInq} and share the same initial data while having uniformly bounded, compact support.  Such a family may be chosen to have a common dissipation measure $\mu = -[ \pr_t (|v_\b|^2/2) + \nb_j((|v_\b|^2/2 + p_\b) v_\b^j)]$ independent of $\b$ while being homeomorphic to a Cantor set as a subspace of $C_{t,x}^\a$, and also having positive Hausdorff dimension as a subspace of $C_t L_x^2$.
\end{thm}
\noindent This theorem gives in particular the first examples of continuous Euler flows that satisfy the local energy inequality with nonzero dissipation, and the first partial result towards Conjecture~\ref{conj:inviscConj}.  

Our motivation for showing nonuniqueness for a fixed energy dissipation measure is related to the possible physical interpretation of general nonunique solutions.  One may imagine that nonuniqueness of globally dissipative Euler flows may be due to information that is lost in the zero viscosity limit if such solutions were to arise in this limit.  A related idea conveyed in \cite{shnNonUnq} is that the energy nonconserving Euler flows may be viewed as Euler flows with an ``invisible'' force that is equal to zero as a distribution, but may be captured by a different mathematical description.  One may ask in this regard if the dissipation measure itself can encapsulate the effects of a weakly vanishing friction and restore uniqueness.  The result of Theorem~\ref{thm:enDispNunq}, reinforcing Theorem~\ref{thm:DeLSzeThm}, would suggest that such a result may not be expected.




Our next theorem shows that the same type of nonuniqueness result holds for global solutions to incompressible Euler that satisfy local conservation of energy as in the local energy equality \eqref{eq:locEnEq}.  
\begin{thm}\label{thm:enEqNunq}  The statement of Theorem~\ref{thm:enDispNunq} holds also with the local energy inequality \eqref{eq:locEnInq} replaced by the local energy equality \eqref{eq:locEnEq}, and the condition that the family $(v_\a)_{\a \in 2^\N}$ have compact support replaced by the condition that the family $(v_\a)_{\a \in 2^\N}$ are defined globally on the domain $\R \times \T^3$.
\end{thm}
The fact that the family of solutions is homeomorphic to a Cantor set has some physical interest.  Since there are no isolated members of the family, no finite-precision measurement (that is continuous in the $C^\alpha$ topology) can isolate a solution to restore uniqueness.

We note here that the nonuniqueness statement we obtain is stronger than the previous results on nonunique energy-conserving weak solutions, including those proven in the $L_{t,x}^\infty$ context.   
More precisely, in \cite{deLSzeAdmiss} what is shown is that there is a complete, separable metric space of ``subsolutions'' such that there is a Baire residual (i.e. generic) set of subsolutions consisting of solutions to Euler that obey the local energy equality and have the same initial data (in the sense of $C_t L^2_w$).  An exercise in real analysis shows that every separable metric space has a dense $G_\de$ subset with zero Hausdorff dimension; thus, one cannot deduce that a Baire generic subset of a complete, separable metric space will have positive Hausdorff dimension.  Furthermore, our theorem yields a metric space of solutions that is already complete (since it is compact) rather than being only Baire generic, and the initial data is obtained in a strong topology ($C^\alpha$) rather than a weak topology, so energy conservation holds for all $t \in [0,\infty)$ including $t = 0$.


The nonuniqueness statements in Theorems~\ref{thm:enDispNunq} and ~\ref{thm:enEqNunq} are proven by implementing a new strategy for proving nonuniqueness of solutions 
that represents the first use of a probabilistic method in the context of convex integration.  The proof proceeds by inserting a random coin toss into each step of the iteration at a point where none of the estimates of the iteration will be disturbed.  In this way our work gives a first answer to the question of whether there exists a probabilistic approach to convex integration raised in \cite[Problem 9]{deLSzeHFluid}.  This question has a natural motivation in our context given the role that randomness is believed to play in turbulent motion.  The method leads naturally to a Cantor family of solutions, as the outcome of every coin toss changes the final solution in a quantitative way.  
Our proof that the resulting family of solutions has positive Hausdorff dimension in the energy space $C_tL_x^2$ uses the exponential (rather than double-exponential) growth of frequencies in the iteration obtained via techniques of \cite{isettVicol}.  It is plausible to expect that our method may also produce an infinite Hausdorff family of solutions when a greater amount of randomness is employed. 



The final application we consider here concerns the $C^0$ density of ``wild'' initial data that admit an infinite family of nonunique weak solutions to incompressible Euler satisfying the local energy equality \eqref{eq:locEnEq}.  The result 
may be compared to the analogous density results in \cite{szeWiedYoung,danSze}; however, our proof is by a quite different approach.  We refer to Section~\ref{sec:constructDissipSolns} for a more precise discussion.  
\begin{thm} \label{thm:wildInitData}  For any $\a < 1/15$, the set of initial data on $\T^3$ that admit infinitely many solutions of class $C_{t,x}^\a$ that obey the local energy equality \eqref{eq:locEnEq} on a common interval of time containing $0$ is $C^0$ dense in the space of continuous, divergence free initial data.
\end{thm}

Our proof of Theorems~\ref{thm:enDispNunq}-~\ref{thm:wildInitData} above proceeds by developing a new type of convex integration argument based on the concept of a ``dissipative Euler-Reynolds flow''.  This concept, which extends the notion of an Euler-Reynolds flow of \cite{deLSzeCts}, keeps track of errors in solving the Euler equation while allowing for a relaxation of the local energy inequality \eqref{eq:locEnInq} to hold.  The goal of the argument is then to design an elaborate convex integration scheme that can not only reduce the error in solving the incompressible Euler equations, but also reduce the error in satisfying the local energy inequality simultaneously.  



Obtaining a scheme that satisfies all of the conditions we require for this argument turns out to be more restrictive than previous constructions of incompressible Euler flows that have been used in the results towards Onsager's conjecture.  We have found in particular that both Beltrami flows (as applied in \cite{deLSzeCts,deLSzeHoldCts,isett, buckDeLIsettSze} et. al.) and Mikado flows (as applied in \cite{danSze,isettOnsag,BDLSVonsag}) appear to be incompatible with our goals for separate reasons.  For Beltrami flows it is apparently the case that no significant progress towards improving the error in satisfying the local energy inequality is obtained when they are implemented in the construction.  Mikado flows seem to face a different difficulty, namely that their implementation in \cite{isettOnsag, BDLSVonsag} introduces both creation and dissipation of energy in the disjoint regions where the waves are supported and cut off.  It is unclear what kind of construction may be necessary to advance beyond such difficulties and hopefully achieve a future proof of Conjecture~\ref{conj:inviscConj}.  

In view of the above, our proof relies on a different construction based on ``one-dimensional'' waves that was first outlined by \cite{isettVicol} and turns out to be suitable for improving the error in the local energy inequality.  In the present paper we improve the ansatz that had been outlined there to make the waves ``multi-dimensional'' in our proofs of Theorems~\ref{thm:enDispNunq}-\ref{thm:wildInitData}.  Our proof generalizes readily to extend Theorems~\ref{thm:enDispNunq}-\ref{thm:wildInitData} to all dimensions $d \geq 3$ while obtaining the same regularity $1/15$ independent of the dimension.  We believe that our argument may also be extended in part to dimension two\footnote{It is possible that both Conjectures~\ref{conj:inviscConj} and \ref{conj:navStConj} could be true in both dimensions two and three.  However, the two settings are highly different regarding the circumstances under which one may possibly encounter solutions with vorticity becoming arbitrarily large as would be true for the sequence in Conjecture \ref{conj:navStConj}.  }; however, in this case there are additional complications related to the error in the local energy inequality that have prevented us so far from obtaining a scheme of comparable regularity.  We discuss these points further in Section~\ref{sec:othDimRemarks}.

In the course of our proof of Theorems~\ref{thm:enDispNunq}-\ref{thm:wildInitData}, our work develops several new techniques in the method of convex integration.  As noted, our method of proving nonuniqueness in Theorems~\ref{thm:enDispNunq} and \ref{thm:enEqNunq} is new and yields a stronger nonuniqueness result than the methods previously used in the literature on H\"{o}lder continuous solutions.  
Our method of establishing density of wild initial data is also new, as we obtain the result as a consequence of a general Approximation Theorem (Theorem~\ref{thm:approxThm} below), which implies also both Theorems~\ref{thm:enDispNunq} and \ref{thm:enEqNunq}.  We have also maintained and proven estimates throughout the proof that are as sharp as possible so that our work may be applicable to higher regularity regimes of parameters and potentially to a broader range of applications.   


$ $


\noindent {\bf Acknowledgements.}  The author is grateful to many people for discussions that have motivated his pursuit of this problem.  We especially thank T. Buckmaster, T. Drivas, G. Eyink, D. Levermore, S. Friedlander, R. Shvydkoy, V. \v{S}verak and V. Vicol  for enriching conversations and encouragement, which have been influential to the writing of this introduction.  We also thank T. Drivas and G. Eyink for communicating their work \cite{drivasEyinksingLeray}.  The work of the author has been supported by the National Science Foundation under Award Nos. DMS-1402370 and DMS-2055019, and by a Sloan Fellowship.  

\subsection{Outline}
We will describe the notation of the paper in Section~\ref{sec:notation}.  The basic concepts of a dissipative Euler-Reynolds flow and the frequency energy levels that will be used to estimate such flows are described in Section~\ref{sec:dissipEuRnFlow}.  The Main Lemma is stated in Section~\ref{sec:mainLemma}.  The bulk of the paper then consists of the proof of the Main Lemma, which is contained in Section~\ref{sec:mainLemProof}.  Section~\ref{sec:mainStep1} provides a general introduction to the proof of the Main Lemma, which is carried out in detail in Sections~\ref{sec:constrShape}-\ref{sec:commutadvec}.  Section~\ref{sec:verifying} provides a summary and final details of the proof, and summarizes where in \ref{sec:constrShape}-\ref{sec:commutadvec} the various statements of the Main Lemma have been proven.  Section~\ref{sec:othDimRemarks} outlines the extension of the proof to dimensions $d \geq 3$ and contains discussion of the two dimensional case.

Theorems~\ref{thm:enDispNunq}-\ref{thm:wildInitData} are deduced from a general Approximation Theorem stated in Theorem~\ref{thm:approxThm}.  The proof of Theorem~\ref{thm:approxThm}, which immediately implies Theorems~\ref{thm:enDispNunq} and \ref{thm:enEqNunq}, is carried out in Sections~\ref{sec:admissConditions}-\ref{sec:approximThm} using the Main Lemma of Section~\ref{sec:mainLemma}.  The final Section~\ref{sec:constructDissipSolns} provides the proof of Theorem~\ref{thm:wildInitData}, which is formulated in more detail as Theorem~\ref{thm:wildDenseThm}.

\subsection{Notation} \label{sec:notation}
We will always view tensor fields (including scalar and vector fields) on $\T^3 = (\R/\Z)^3$ as distributions defined on all of $\R^3$ that are $\Z^3$-periodic.  We will use the symbol $p \in \widehat{\R^3}$ to denote the frequency variable, and $\Fsupp F := \supp \hat{F} \subseteq \widehat{\R^3}$ to denote the Fourier support of a tensor field $F$.  We use $\suppt F = \{ t \in \R ~:~ \{ t \} \cap \supp F \neq \emptyset \}$ to denote the time-support of a tensor field $F$.  We will follow the summation convention of summing upper and lower spatial indices that are repeated (e.g., $\nb_j(v^jv^\ell) = \sum_{j=1}^3 \nb_j(v^j v^\ell)$), maintaining the indices as raised or lowered depending on whether they are viewed as covariant or contravariant indices.  

We will always use vector notation to denote a multi-index $\va = (a_1, \ldots, a_n)$ of order $n = |\va| \geq 0$, where each $a_i \in \{1, 2, 3 \}$.  The expression $\nb_{\va}$ refers to the corresponding partial derivative operator $\nb_{\va} = \nb_{a_1} \nb_{a_2} \cdots \nb_{a_n}$, while for $h \in \R^3$, $h^{\va}$ refers to the corresponding multinomial $h^{a_1} h^{a_2} \cdots h^{a_n}$ of degree $n = |\va| \geq 0$.  We will write $\| \nb^L F \|_{C^0} := \max_{|\va| = L} \co{ \nb_{\va} F}$ to denote the $C^0$ norm of the $L$'th derivative of a tensor field $F$, where all $C^0$ norms are taken as space-time norms.  The $C_tL_x^2$ norm of a tensor field taking values in $L^2(\T^3)$ continuously in $t$ is given by $\| F \|_{C_t L_x^2} = \sup_t \| F(t,\cdot) \|_{L^2(\T^3)} = \sup_t (\int_{\T^3} |F(t,x)|^2 dx )^{1/2}$.  

We record the following elementary counting inequality, which may be shown by induction on $m$:
\ali{
(x_1 + \ldots + x_m - y)_+ \leq \sum_i (x_i - y)_+, \qquad \mbox{ if } y, x_i \geq 0 \mbox{ for all } 1 \leq i \leq m. \label{eq:countingIneq}
}
At certain points we will use the notation $A \unlhd B$ to refer to an inequality $A \leq B$ that has not been established, but will be assumed from that point in order to ensure that certain terms are well-defined, and will later be verified to hold.

\section{Dissipative Euler-Reynolds Flows} \label{sec:dissipEuRnFlow}
The starting point for our proof will be to introduce the following notion of a {\bf dissipative Euler-Reynolds flow}, which provides us with a natural space of vector and tensor fields that can be used to approximate weak solutions to the Euler equations that satisfy the local energy inequality \eqref{eq:locEnInq}.  We define this notion as follows by augmenting the well-known Euler-Reynolds system of \cite{deLSzeCts} with an additional inequality (Inequality \eqref{eq:relaxedLocEnIneq} below) that represents a relaxation of the local energy inequality.  The inequality contains two new fields (the unresolved flux density and current, $\kk, \vp$), which, together with the Reynolds stress $R$, measure the error in the local energy inequality.  

\begin{defn} A {\bf dissipative Euler-Reynolds flow} on $\tilde{I} \times \T^n$, $\tilde{I}$ an interval, is a tuple of tensor fields $(v, p, R, \kk, \vphi, \mu)$ on $\tilde{I} \times \T^n$ consisting of: a vector field $v^\ell$ (the {\bf velocity field}), a scalar function $p$ (the {\bf pressure}), a symmetric tensor field $R^{j\ell}$ (the {\bf Reynolds stress}), a scalar function $\kk$ (the {\bf unresolved flux density}), a vector field $\vphi^j$ (the {\bf unresolved flux current}), all smooth in the spatial variables and satisfying the following system of equations
\ali{
\pr_t v^\ell + \nb_j(v^j v^\ell) + \nb^\ell p &= \nb_j R^{j\ell}, \qquad \nb_j v^j = 0 \label{eq:euReyn} \\
\pr_t \left( \vsq \right) + \nb_j\left[ \left( \vsq + p\right) v^j \right] &\leq D_t \kk + \nb_j[v_\ell R^{j\ell}] + \nb_j \vphi^j \label{eq:relaxedLocEnIneq} \\
D_t F &:= \pr_t F + \nb_i(v^i F), \notag
}
together with the scalar function $\mu \geq 0$ (the {\bf dissipation measure}) defined by
\ali{
\mu = - \left( \pr_t \left( \vsq \right) + \nb_j\left[ \left( \vsq + p\right) v^j \right] \right) + D_t \kk + \nb_j[v_\ell R^{j\ell}] + \nb_j \vphi^j. \label{eq:dissipMeasdef}
}
\end{defn}
\begin{rem}  The definition of a dissipative Euler-Reynolds flow, which generalizes the notion of an Euler-Reynolds flow, can be naturally motivated by observing that any ensemble average $(\overline{v}, \overline{p})$ of globally dissipative Euler flows must be a dissipative Euler-Reynolds flow.  From the point of view of ensemble averages, the unresolved flux density $\kk$ arises as half the trace of $R^{j\ell} := \overline{v^j v^\ell} - \overline{v}^j\overline{v}^\ell$ while the unresolved flux current $\phi^j$ arises from trilinear variations in velocity and bilinear variations in velocity and pressure in a cumulant expansion of the averaged cubic nonlinearity $\overline{\left(\fr{|v|^2}{2} + p \right) v^j }$ on the left hand side of \eqref{eq:locEnInq}.  See \cite[Section 1]{isett} for more on the significance of this remark for performing convex integration.
\end{rem}

Our strategy at this point will be to construct a convex integration scheme that generates a sequence of smooth, dissipative Euler Reynolds flows $(v,p,R,\kk, \vphi, \mu)_{(k)}$ indexed by $k$ with $(R, \kk, \vphi)_{(k)}$ all tending to $0$ uniformly as $k \to \infty$ and with $(v,p)_{(k)}$ converging uniformly to a H\"{o}lder continuous, weak solution to incompressible Euler.  As the error terms in \eqref{eq:relaxedLocEnIneq} that involve $\kappa$, $R$ and $\vphi$ will all converge weakly to zero, the limiting solution $(v,p)$ will satisfy the local energy inequality \eqref{eq:relaxedLocEnIneq} with dissipation measure $\mu = - [ \pr_t(|v|^2/2) + \nb_j( (|v|^2/2 + p) v^j) ]$ equal to the weak limit $\mu = \mbox{w-}\lim_{k \to \infty} \mu_{(k)} \geq 0$.  



We measure the size of these errors using the following notion of {\bf compound frequency-energy levels}.  The notion is adapted from the use of analogous estimates in \cite{isettVicol}, which were introduced for the purpose of executing a convex integration scheme in which only a single ``component'' of the error could be reduced in size during the construction of solutions.  In our construction, we will be able to eliminate essentially a part of the error that can be built from restricting to a single subspace in each stage.  In what follows we will view $dx^1 = [1,0,0], dx^2 = [0,1,0], dx^3 = [0, 0, 1]$ as elements in the dual of $\R^3$, so that $\ker dx^1 = \langle e_2, e_3 \rangle, \ker dx^2 = \langle e_1, e_3 \rangle, \ldots$, and $\ker dx^1 \otimes \ker dx^1 = \oplus_{i, j \neq 1} \langle e_i \otimes e_j \rangle$, etc.
\begin{defn} \label{eq:compFrEnlvls} For $\Xi \geq 2$, $e_v \geq e_\vphi \geq e_R > e_G> 0$ and integers $L > 0$, $I, J \in \{ 1, 2, 3\}$, $I \neq J$, we say that a dissipative Euler-Reynolds flow $(v, p, R, \kk, \vphi, \mu)$ on $\tilde{I} \times \T^3$ has {\bf compound frequency-energy levels} (of type $[I,J]$ and order $L$) bounded by $(\Xi, e_v, e_\vphi, e_R, e_G)$ if one can express the stress tensor and unresolved flux density and current in the form
\ali{
\label{eq:decompErrTerms}
\begin{split}
R^{j\ell} &= R_{[I]}^{j\ell} + R_{[J]}^{j\ell} + R_{[G]}^{j\ell}, \\
\kk &= \kk_{[I]} + \kk_{[G]}, \quad \kk_{[I]} = \de_{j\ell}R_{[I]}^{j\ell}/2 \\
\vp^j &= \vp_{[1]}^j + \vp_{[G]}^j,
\end{split}
} 
where $R_{[I]}^{j\ell}$ and $R_{[J]}^{j\ell}$ are symmetric tensor fields that take values respectively in the vector subspaces  $\ker dx^I \otimes \ker dx^I$ and $\ker dx^J \otimes \ker dx^J$ of $\R^3 \otimes \R^3$, $\vp_{[I]}^j$ takes values in $\ker dx^I$, and the following estimates hold uniformly on $\tilde{I} \times \T^3$:
\ali{
\label{eq:vpBd}
\co{\nb_{\va} v} \leq \Xi^{|\va|} e_v^{1/2}, \quad 
\co{\nb_{\va} p} &\leq \Xi^{|\va|} e_v,  \qquad \qquad 1 \leq |\va| \leq L  \\
\co{\nb_{\va} D_t p} &\leq \Xi^{|\va| + 1} e_v^{3/2}, \qquad 0 \leq |\va| \leq L-1 \label{eq:pAdvecBd} \\
\co{\nb_{\va} F} &\leq \Xi^{|\va|} h_F, \qquad \qquad 0 \leq |\va| \leq L \label{eq:quadc0bdlvl} \\
\co{\nb_{\va} D_t F } &\leq \Xi^{|\va|+1} e_v^{1/2} h_F, \qquad 0 \leq |\va| \leq L - 1. \label{eq:quadDdtbdlvl}
}
Here $(F,h_F)^t$ is any column of the following formal matrix (in which we define $e_{\undvp} := e_\vp^{1/3} e_R^{2/3}$)
\ali{
\mat{c|ccccccc}{ F & R_{[I]} & R_{[J]} & R_{[G]} & \kk_{[I]} & \kk_{[G]} & \vp_{[I]} & \vp_{[G]} \\
		h_F &	e_\vp & e_R & e_G & e_\vp & e_R & e_\vp^{3/2} & e_{\undvp}^{3/2}  
		},
		\label{eq:lastFrqenBd}
}
and $D_t F := \pr_t F + v \cdot \nb F = \pr_t F + \nb_i(v^i F)$.  If $\tilde{I}$ is finite we also require $|\tilde{I}| \geq 8 (\Xi e_v^{1/2})^{-1}$.
\end{defn}

\section{The Main Lemma} \label{sec:mainLemma}
The following technical Main Lemma summarizes the result of one stage of the iteration.  The Lemma is stated in terms of the notion of a $(\bar{\de},M)$-well-prepared dissipative Euler-Reynolds flow (to be defined in Definition~\ref{defn:wellPrepared} below), which places conditions on the form of the principal part of the error term $R_{[1]}^{j\ell}$ that are consistent with the outcome of the previous stage of the iteration, and that in particular imply positive definiteness and that the trace part is dominant on the appropriate intervals.  In particular, even though the error terms are supported in a certain interval $\tilde{I}_{[G]}$, it is only on a smaller interval $\tilde{I}_{[l]}$ that the trace part dominates the largest incoming error terms.  This nuance of the iteration is used to obtain compactly supported solutions, but not to obtain solutions satisfying local energy conservation.

In the following, we recall that $e_{\undvp} = e_\vp^{1/3} e_R^{2/3}$ is as in Definition~\ref{eq:compFrEnlvls}.  

\begin{lem}[Main Lemma] \label{lem:mainLem} For any $L \geq 2$, there exist constants $\bar{\de} > 0$ and $\hc, C_L > 1$ such that the following holds: Let $(\Xi, e_v, e_\vphi, e_R, e_G)$ be compound frequency energy levels.  Let $(v,p,R,\kk, \vp, \mu)$ be a dissipative Euler-Reynolds flow on $\tilde{I} \times \T^3$ with compound frequency energy levels (of type $[l,l{+}1]$ mod 3 and order $L$) bounded by $(\Xi, e_v, e_\vphi, e_R, e_G)$.  Let $\tilde{I}_{[l]} \subseteq \tilde{I}_{[G]}$ be nonempty subintervals of $\tilde{I}$, let $\bar{e} \colon \tilde{I} \to \R_+$ be a smooth, non-negative function 
and let $N \geq 1$ be such that
\ali{
N &\geq \max \left\{ (e_v/e_G)(e_\vp^{1/2}/e_{\undvp}^{1/2}), (e_v^{1/2}/e_\vp^{1/2})(e_\vp/e_G)^2 \cdot (e_\vp/e_{\undvp}), (e_v/e_\vp)^{3/2} \right\}, \label{eq:Nastdef} \\
\suppt &(R_{[1]}, R_{[2]}, R_{[G]}, \vp_{[1]}, \vp_{[G]}, \kk_{[G]} ) \subseteq \tilde{I}_{[G]}. \label{eq:intervalContainingSupp}
}
Suppose also that $(v,p,R,\kk, \vp, \mu)$ is {\bf $(\bar{\de},M)$-well-prepared} (for stage $[l]$) for some $M \geq 1$ in the sense of Definition~\eqref{defn:wellPrepared} below with respect to these frequency energy levels and the trio $(\tilde{I}_{[l]}, \tilde{I}_{[G]}, \bar{e})$.

Then there exists a dissipative Euler-Reynolds flow $(\ost{v}, \ost{p}, \ost{R}, \ost{\kk}, \ost{\vp}, \ost{\mu})$ that has compound frequency energy levels (of type $[l{+}1,l{+}2]$ mod $3$ and order $L$) bounded by
\ali{
(\ost{\Xi}, \ost{e}_v, \ost{e}_\vphi, \ost{e}_R, \ost{e}_G) &= (N \hc \Xi, \hc e_\vp, \hc e_{\undvp}, \hc e_G, (e_v^{1/2}/(e_\vp^{1/2} N) )^{1/2} e_\vp), \quad e_{\undvp} = e_\vp^{1/3} e_R^{2/3}, \label{eq:frEnLvlsNew}
}
and that is $(\bar{\de}, \hc)$-well prepared with respect to the frequency energy levels \eqref{eq:frEnLvlsNew}, some pair of intervals $\tilde{I}_{[l+1]} \subseteq \tilde{I}_{[G\ast]}$ contained in $\tilde{I}$, and some smooth, non-negative $e \colon \tilde{I} \to \R_+$.  The trio $(\tilde{I}_{[l+1]}, \tilde{I}_{[G\ast]}, e)$ may be chosen such that $e = e(t)$ depends only on $(\Xi, e_v, e_\vp, e_R, e_G)$, $\tilde{I}_{[G]}$ and $L$.

One may arrange that $\ost{v}$ has the form $\ost{v} = v + V$ such that 
\ali{
\co{V} &\leq C_L e_\vp^{1/2}, \label{eq:c0CorrectBd} \\
\sup_t \left| \int_{\T^3} \phi_\ell(t,x) V^\ell(t,x) dx \right| &\leq \Xi^{-1} e_\vp^{1/2} \sup_t \| \nb \phi(t,\cdot) \|_{L^1}, \qquad \tx{ for all } \phi_\ell \in C^\infty(\tilde{I} \times \T^3), \label{eq:testFunction}
}
and also such that the restrictions of $V$ and the components of $(\ost{v}, \ost{p}, \ost{R}, \ost{\kk}, \ost{\vp})$ to any interval $\tilde{I}' \subseteq \tilde{I}$ depend only on the data $(\Xi, e_v, e_\vp, e_R, e_G)$, $(\tilde{I}_{[l]}, \tilde{I}_{[G]}, \bar{e})$ and on the restriction of $(v,p,R,\kk, \vp)$ and their components to the $(\Xi e_v^{1/2})^{-1}$ neighborhood of $\tilde{I}'$ in $\tilde{I}$.

It is possible to arrange that $\ost{\mu} - \mu$ depends only on $\tilde{I}_{[G]}$, $(\Xi, e_v, e_\vp, e_R, e_G)$ and $L$, and is everywhere non-negative 
with $\ost{\mu} = \mu$ on $\tilde{I}_{[G]}$.  One may also arrange that 
\ali{
\tilde{I}_{[l+1]} &= \{ t + \bar{t} ~:~ t \in \tilde{I}_{[G]},\,|\bar{t}| \leq 10^{-3}(\Xi e_v^{1/2})^{-1} \} \cap \tilde{I} \label{eq:tildeIlp1} \\
\suppt V \cup \suppt (\ost{R}, \ost{\kk}, \ost{\vp})) &\subseteq  \tilde{I}_{[G\ast]} := \{ t \leq \sup \tilde{I}_{[G]} + (\Xi e_v^{1/2})^{-1}/40 \} \cap \tilde{I}, \label{eq:suppContained}
}
where the containment \eqref{eq:suppContained} applies also to the all the components of $(\ost{R}, \ostk, \ost{\vp})$ in \eqref{eq:decompErrTerms}.


Furthermore, for any interval $\tilde{J}$ with length $|\tilde{J}| \geq \Xi^{-1} e_v^{-1/2}$ that is contained in $\tilde{I}_{[l]}$, 
there exists a (possibly different) dissipative Euler-Reynolds flow $(v_2, p_2, R_2, \kk_2, \vp_2, \mu_2)$, $v_2 = v + V_2$ obeying all of the above conclusions stated for $(\ost{v}, \ost{p}, \ldots)$ and for $V$ with the same functions $e_2(t) = e(t)$ and $\mu_2 = \ost{\mu}$, but possibly differing from $\ost{v}$ somewhere in the interval $\tilde{J}$ and satisfying
\ali{
&\supp_t \left(v_2 - \ost{v}, p_2 - \ost{p}, R_2 - \ost{R}, \kk_2 - \ost{\kk}, \vp_2 - \ost{\vp} \right) \subseteq \tilde{J} \label{eq:diffSupp} \\
&\sup_t \int_{\T^3} | v_2(t,x) - \ost{v}(t,x)|^2 dx \geq ((C_L  M)^{-1} - N^{-1}) e_{\vp}, \label{eq:lowerBdNew}
}
with \eqref{eq:diffSupp} applying also to the components of $(R_2 - \ost{R}, \kk_2 - \ost{\kk}, \vp_2 - \ost{\vp})$.
\end{lem}

In the proof of Lemma~\ref{lem:mainLem} that follows we will consider only the case of type $[l,l{+}1] = [1,2]$ frequency energy levels to simplify notation.  The general case, which will be used in Section~\ref{sec:mainLImpMThm} to prove the main theorems, follows by symmetry by permuting the coordinate axes.  We note that the estimate \eqref{eq:testFunction} will not be required for the applications we consider in this paper.  However, we have included the bound due its usefulness for further applications such as results of $h$-principle type (e.g., \cite{isettVicol, IOnonpd}).

\section{Proof of the Main Lemma} \label{sec:mainLemProof}
\subsection{Proof of Main Lemma: Introduction} \label{sec:mainStep1}
Let $\vprk$ be a dissipative Euler flow given as in the assumptions of Lemma~\ref{lem:mainLem}.  We will construct the new velocity field $\ost{v}^\ell = v^\ell + V^\ell$ and pressure $\ost{p} = p + P$ by adding carefully designed corrections $V$ and $P$ to the velocity field and pressure.  

The correction $V^\ell$ is a ``high-frequency'', divergence free vector field constructed as a sum $V = \sum_{I \in \II} V_I$ of individual waves $V_I^\ell$ that are divergence free and that are together indexed by a set $\II$.  The $V_I^\ell$ are complex-valued vector fields, and each $V_I$ has a conjugate wave $V_{\bar{I}} := \overline{V_I}$ indexed by $\bar{I} \in \II$ so that the correction $V^\ell$ is real valued.  A new point in this construction is that there will be several types of waves that partition the index set into $\II = \II_R \amalg \II_\vp$.  The waves indexed by $\II_\vp$ are those that will be used to cancel out (part of) the unresolved flux current $\vp$, whereas the waves indexed by $\II_R$ will be used to cancel out (part of) the Reynolds stress $R$, and these two families are disjoint.  These families will furthermore be partitioned into $\II_R := \ovl{\II}_R \amalg \dmd{\II}_R, \II_\vp = \ovl{\II}_\vp \amalg \dmd{\II}_\vp$, where the waves in $\ovl{\II} = \ovl{\II}_R \amalg \ovl{\II}_\vp$ are the largest and they will be designed to cancel out the main error terms $R_{[1]}^{j\ell}$ and $\vp_{[1]}^j$, while those in $\dmd{\II} := \dmd{\II}_R \amalg \dmd{\II}_\vp$ will be designed to cancel out (part of) the remaining, smaller error terms and may have slightly larger support in time.  This decomposition yields also that $\II = \ovl{\II} \amalg \dmd{\II}$.




The new Reynolds stress $\ost{R}$ will be constructed as in previous incompressible Euler schemes as a sum of several terms that will be required to satisfy the following equations
\ali{
\ost{R}^{j\ell} &= R_T^{j\ell} + R_S^{j\ell} + R_H^{j\ell} + R_M^{j\ell} \label{eq:stressDecomp} \\
\nb_j R_T^{j\ell} &= \pr_t V^\ell + \nb_j[v_\ep^j V^\ell + V^j v_\ep^\ell] \\
R_S^{j\ell} &= \sum_I V_I^j \overline{V_I^\ell} + P \de^{j\ell} + R_\ep^{j\ell}  \label{eq:stressTerm1} \\
\nb_j R_H^{j\ell} &= \sum_{J \neq \bar{I}} \nb_j( V_I^j V_J^\ell) \\
R_M^{j\ell} &= [ (v^j - v_\ep^j) V^\ell + V^j (v^\ell - v_\ep^\ell) ] + (R^{j\ell} - R_\ep^{j\ell}) \label{eq:Rmolldef} \\
&:= R_{Mv}^{j\ell} + (R^{j\ell} - R_\ep^{j\ell}). \notag 
}
The vector field $v_\ep^\ell$ and symmetric tensor field $R_\ep^{j\ell}$ are regularizations of the given $v^\ell$ and $R^{j\ell}$ that will be defined in Section~\ref{sec:regPrelimBds} below.  As in Definition~\ref{eq:compFrEnlvls}, we will also have a decomposition $R_\ep^{j\ell} = R_{[1\ep]}^{j\ell} + R_{[2\ep]}^{j\ell} + R_{[G\ep]}^{j\ell}$, where the $R_{[l\ep]}^{j\ell}$ take values in $\ker dx^l \otimes \ker dx^l$ for $l = 1, 2$, and we require that $\co{R_{[l\ep]}} \unlhd A_L \co{R_{[l]}}$ for $l = 1, 2$ and $\co{R_{[G\ep]}} \unlhd A_L \co{R_{[G]}}$ with $A_L \geq 1$ depending on $L$. 

Towards considering the local energy inequality, let us introduce the {\bf resolved energy flux density} of $(v,p)$, which we define to be $\DD[v, p] := \pr_t(|v|^2/2) + \nb_j((|v|^2/2 + p)v^j)$.  
Then the definition of the dissipation measure $\mu$ in \eqref{eq:dissipMeasdef} gives
\ali{
\DD[v,p] &= D_t \kk + \nb_j(v_\ell R^{j\ell}) + \nb_j \vp^j - \mu \label{eq:dissipationDef}
}

A major goal of this construction will be to design the new $(\ost{v}, \ost{p})$ so that its resolved energy flux $\ost{\DD} := \DD[\ost{v}, \ost{p}]$ has the form \eqref{eq:dissipationDef} with $(\ost{\kk}, \ost{R}, \ost{\vp})$ much smaller than the original $(\kk, R, \vp)$.  Using $\ost{v} = v + V$, $\ost{p} = p+P$, $p v^j = v_\ell(p\de^{j\ell})$, and \eqref{eq:dissipationDef} the new $\ost{\DD} = \DD[\ost{v}, \ost{p}]$ can be decomposed in the form
\ali{
\ost{\DD} &= \mathring{\DD}_T + \mathring{\DD}_S + \mathring{\DD}_H + \DDR_{\kk} + \DDR_\vp - \mu, \label{eq:newDensity}
}
where the first three terms are related to the decomposition of the stress in \eqref{eq:stressDecomp}
\ali{
\mathring{\DD}_T&= \pr_t(v_\ell V^\ell) + \nb_j(v_\ell V^\ell v^j) + \nb_j\left( \left(\vsq + p\right) V^j \right) \label{eq:transFluxTerm} \\
\mathring{\DD}_S &= \nb_j\Big[ v_\ell[ \sum_I V_I^j \overline{V_I^\ell} + P \de^{j\ell} + R^{j\ell}] \Big] \label{eq:DDS} \\
\mathring{\DD}_H &= \nb_j\Big[ v_\ell[ \sum_{J \neq \bar{I}} V_I^j V_J^\ell ] \Big], 
}
and where the other terms involve the unresolved flux density and current, $\kk$ and $\vp$
\ali{
\DDR_{\kk} &= \pr_t\left( \sqq{V} + \kk\right) + \nb_j\left( \left( \sqq{V} + \kk \right) v^j \right) \label{eq:fluxDissTerm} \\
\DDR_\vp &= \nb_j\Big[ \sum_{I,J,K} (V_I)_\ell V_J^j V_K^\ell + P V^j + \vp^j \Big]. \label{eq:vpTerms}
}

Since one of our main objectives is to ensure the new unresolved flux density $\ost{\vp}$ will be substantially smaller than the given $\vp$, we aim to design our corrections so that the frequency cascades in the product of equation \eqref{eq:vpTerms} will cancel out as much as possible of the ``low-frequency'' part of $\vp$.  At the same time, we will also need to reduce the sizes of both $\kk$ and $R$.

The goal of canceling out $\vp$ will be accomplished by the waves $V_I$, $I \in \II_\vp$, as follows.  We will specify a subset $\calT \subseteq \II_\vp \times \II_\vp \times \II_\vp$ that is symmetric under conjugation $(I,J,K) \in \calT \Leftrightarrow (\bar{I}, \bar{J}, \bar{K}) \in \calT$ with the property that all trilinear interactions of the form $(V_I)_\ell V_J^j  V_K^\ell$ with $I, J, K \in \calT$ produce low frequency terms that together will cancel out the term $\vp$ in \eqref{eq:vpTerms}.  We perform this task by first splitting $\DDR_\vp = \DDR_{\vp L} + \DDR_{\vp H}$ into a low frequency cascade term plus purely high frequency terms of the form
\ali{
\DD_{\vp L} = \nb_j \Big[ \sum_{I, J,K \in \calT} (V_I)_\ell V_J^j V_K^\ell + \vp_\ep^j \Big], \qquad
\DDR_{\vp H} = \nb_j \Big[ \sum_{I,J,K \notin \calT} (V_I)_\ell V_J^j V_K^\ell + P V^j + (\vp^j - \vp_\ep^j) \Big], \label{eq:phiTermLeft}
}
where $\vp_\ep$ is a regularization of $\vp$ to be defined with the other regularizations in Section~\ref{sec:regPrelimBds} below.

It is of interest to note\footnote{We thank T. Drivas for this comment.} that the use of tri-linear interactions as above is reminiscent of the trilinear interactions that appear in the predictions of the DIA and LDIA theories by Kraichnan for modeling the energy spectrum in turbulent flows (see e.g. \cite[Equation 3]{kraichnan1991turbulent}).  Precisely, Kraichnan's theory predicts that energy cascades in turbulence are governed by trilinear interactions of waves that are carried along the mean flow, much like what will follow.  

In addition to canceling out the term $\vp_\ep$ in \eqref{eq:phiTermLeft}, we must also cancel out the unresolved flux density $\kk$ in \eqref{eq:fluxDissTerm}.  We will treat $\kk$ with a similar approach by first isolating the low frequency part of $\kk$ inside the error term $\DDR_{\kk}$.  Thus we split $\DDR_{\kk} = \DD_{\kk L} + \DDR_{\kk H}$ into low and high frequency terms as follows
\ali{
\DD_{\kk L} &= \Ddt\left[ \sum_I \fr{V_I \cdot \overline{V_I}}{2} + \kk_\ep \right] \label{eq:DDflmL} \\
\DDR_{\kk H} &= \Ddt\Big[\sum_{J \neq \bar{I}} \fr{V_I \cdot V_J}{2} \Big] + \Ddt[ \kk - \kk_\ep ] + \nb_j\left( \left(\sqq{V} + \kk\right) (v^j - v_\ep^j) \right), \label{eq:highkkTerm}
}
where here and in what follows the operator $\Ddt$ denotes the {\bf coarse scale advective derivative} $\Ddt F := (\pr_t + v_\ep \cdot \nb) F = \pr_t F+ \nb_i(v_\ep^i F)$.

Observe that there is an important coupling between our goals of canceling out the term $\kk_\ep$ in \eqref{eq:DDflmL} and the low frequency stress term $R_\ep$ in \eqref{eq:stressTerm1}, namely that term $\sum_I V_I \cdot \overline{V_I} / 2$ that appears in \eqref{eq:DDflmL} is exactly half the trace of the term $\sum_I V_I^j \overline{V_I^\ell}$ that appears in \eqref{eq:stressTerm1}.  We address this tension using our control of the pressure term $P \de^{j\ell}$ and the following idea.

The term $\sum_I V_I \cdot \overline{V_I} / 2$ is essentially the contribution of $V = \sum_I V_I$ to the coarse scale part of the energy density, since the waves $V_I$ will be of high frequency and will be essentially orthogonal.  For a solution that exhibits local energy dissipation, this contribution to the local energy flux should be {\it decreasing} in time.  With this motivation, we will choose a function $e(t)$ with the following properties:
\ali{
\label{eq:eDef}
\begin{split}
e(t) \geq 0,& \qquad  \Ddt e(t) = \pr_t e(t) \leq 0 \\
e(t) &\geq K_0 e_{\undvp},  \qquad \mbox{ if } t \leq \sup \tilde{I}_{[G]} + (\Xi e_v^{1/2})^{-1} \\
\co{\Ddt^r e^{1/2}(t)} &\leq K_1 (\Xi e_v^{1/2})^r e_{\undvp}^{1/2}, \qquad r = 0, 1, 2,
\end{split}
}
where $K_0$ is a constant that will be specified later and $K_1$ is another constant that is allowed to depend on $K_0$.  Our choice of $K_0$ will depend on $L$, but not on any other parameters in the construction.  

The choice of $e(t)$ will determine the contribution of $V = \sum_I V_I$ to the new dissipation measure $\ost{\mu}$, as we will rewrite \eqref{eq:DDflmL} as
\ali{
\DD_{\kk L} &= \Ddt\left[\sum_I \fr{V_I \cdot \overline{V_I}}{2} + \kk_\ep - e(t) \right] + \Ddt e(t) \label{eq:fltermLow}
}
with the key observation that $\Ddt e(t) \leq 0$ is {\bf non-positive} since $e(t)$ is nonincreasing.  This fact implies that this term can be absorbed into the new dissipation measure $\ost{\mu} := \mu - \Ddt e(t)$.  In particular, the dissipation measure $\ost{\mu} = \mu$ will remain unchanged if we take $e(t)$ to be constant in time.

The correction $P$ to the pressure is now determined by our requirement that it is possible to choose $V_I$ to cancel out the main terms in both of equations \eqref{eq:stressTerm1} and \eqref{eq:fltermLow} simultaneously.  Namely, taking the trace of \eqref{eq:stressTerm1} and comparing to \eqref{eq:fltermLow} we require $P = P(t,x)$ to satisfy the following equality pointwise 
\ali{
-e(t) + \kk_\ep &=  (3 P(t,x) + \de_{j\ell} R_\ep^{j\ell})/2. \label{eq:PdefnPre}
}
Using that $\kk = \kk_{[1]} + \kk_{[G]}$, $\kk_{[1]} = -(1/2) \de_{j\ell} R_{[1]}^{j\ell}$, the equation \eqref{eq:Pdefn} becomes
\ali{
- 3 P(t,x) &= 2e(t) - 2\kk_{[G\ep]} + \de_{j\ell} (R_{[2\ep]}^{j\ell} + R_{[G\ep]}^{j\ell}), \label{eq:Pdefn}
}
The main term in this equation will be the term coming from $e(t)$ due to the lower bound in \eqref{eq:eDef}.

\subsection{General shape of the construction} \label{sec:constrShape}
Our correction $V = \sum_I V_I$ will be a sum of divergence free ``plane waves'' that oscillate at a large frequency $\la := B_\la N^* \Xi$.  The parameter $B_\la \geq 2$ will be taken such that $\la$ is an integer multiple of $2 \pi$, and will be bounded in size by a constant $\overline{B_\la}$, which will be the last constant chosen in the argument and depends on all of the other parameters and constants in the construction.  (Informally one should simply think of $B_\la$ as a large constant.)  

More explicitly, the leading term in each individual wave $V_I^\ell$ will be equipped with a real-valued phase function $\xi_I(t,x)$ and an amplitude $v_I^\ell(t,x)$, and there will be a lower order term $\de V_I^\ell$ that is present to ensure that $V_I^\ell$ is exactly divergence free:
\ali{
\begin{split}
V_I^\ell &:= e^{i \la \xi_I} (v_I^\ell + \de v_I^\ell) = \VR_I^\ell + \de V_I^\ell \\
\VR_I^\ell &= e^{i \la \xi_I} v_I^\ell, \qquad \de V_I^\ell := e^{i \la \xi_I} \de v_I^\ell. \label{eq:VIform}
\end{split}
}
The amplitude $v_I^\ell$ will be required to take values orthogonal to the phase gradient $\langle \nb \xi_I \rangle$ in order to make the leading order term in \eqref{eq:VIform} divergence free to leading order in $\la$.  Meanwhile, the phase function $\xi_I$ will be advected by the coarse scale flow of $v_\ep$ initiating from a time $t(I)$ with data $\hat{\xi}_I(x)$:
\ali{
(\pr_t + v_\ep^j \nb_j) \xi_I = 0 \mbox{ on } \R \times \T^3, \qquad &\xi_I(t(I), x) = \hat{\xi}_I(x) \mbox{ on } \T^3, \label{eq:phaseFunc}\\
v_I \cdot \nb \xi_I &= 0 \mbox{ on } \R \times \T^3. \label{eq:orthog}
}
The index $I \in \II$ will have components $I = (k,f) \in \Z \times \F$, where $k \in \Z$ will specify the time interval on which $V_I$ will be supported, $t(I) = k \ostu$ will be specified in line \eqref{eq:timeScale} below, $\F$ is a finite set that we will specify, and $f \in \F$ describes the direction of oscillation $V_I$.  
The index set has a conjugation symmetry $\bar{I} = (k, \bar{f})$, $\bar{f} \in \F$, which corresponds to the conjugate wave defined by $V_{\bar{I}} = \overline{V_I}, v_{\bar{I}} = \overline{v_I}, \xi_{\bar{I}} = - \xi_I$.  The set $\F$ will be partitioned into disjoint subsets $\F = \F_R \amalg \F_\vp$, and we write $I \in \II_R$ if $f(I) \in \F_R$, while $I \in \II_\vp$ if $f(I) \in \F_\vp$.  In general, $k(I)$ and $f(I)$ refer to the $(k,f)$ components of $I$.  We will similarly decompose $\F = \ovl{\F} \amalg \dmd{\F}$, $\F_R = \ovl{\F}_R \amalg \dmd{\F}_R$, $\F_\vp = \ovl{\F}_\vp \amalg \dmd{\F}_\vp$ into larger vs. smaller amplitude waves, each of these groups having the conjugation symmetry $f \in \dmd{\F}_R \Rightarrow \bar{f} \in \dmd{\F}_R$, etc.

Waves having the form \eqref{eq:VIform}-\eqref{eq:orthog} are also used in \cite{isett}, where an additional condition is imposed to emulate the Beltrami flows of \cite{deLSzeCts}.  Our construction will be different in that we use the Ansatz based on ``one-dimensional, plane-waves'' introduced in \cite{isettVicol}, and we take advantage of the frequency localization techniques introduced there.  Our construction also improves on this one-dimensional Ansatz by noting that, in the case of incompressible Euler, one is free to use ``multi-dimensional'' waves in the construction provided one maintains the orthogonality condition \eqref{eq:orthog}.  

We describe these waves as follows.  The oscillations will occur essentially in the direction of the $x^1$ coordinate, and the vector fields will take values $v_I \in \langle \nb \xi_I \rangle^\perp$ approximately tangent to the level hyperplanes of the $x^1$ coordinate, $\langle \nb \xi_I \rangle^\perp \approx \ker dx^1$.  The waves will have operators $\cPI$ that localize to frequencies of size $\approx \la$ near the line $\{ (p_1, 0, 0) \} \subseteq \widehat{\R^3}$.  

More precisely let $\psi \colon \R^3 \to \R$ be a Schwartz function with compact Fourier support in $\supp \widehat{\psi} \subseteq B_{2/3}(1,0,0) \subseteq \widehat{\R^3}$ and $\widehat{\psi} = 1$ in $B_{1/2}(1,0,0)$.  We will choose a finite set of nonzero integers $\F_{\Z} \subseteq \Z_{\neq 0}$, symmetric under negation $n \in \F_{\Z} \Leftrightarrow -n \in \F_{\Z}$, with the property that 
\ali{ 
B_4(n + n', 0, 0) \cap B_4(0, 0, 0) = \emptyset \mbox{ for all } n \neq -n' \mbox{ in } \F_{\Z}.  \label{eq:noWrongCascades}
}
We will also assign to each $[k] \in \Z / 2 \Z$ and $f \in \F$ a natural number $n_{[k], f} \in \F_{\Z}$ by an injective map $n_\cdot \colon (\Z / 2 \Z) \times \F \to \F_{\Z}$ with the conjugation symmetry $n_{[k],\bar{f}} = - n_{[k],f}$.  For $I = (k, f) \in \Z \times \F$, we let $n_{[I]} = n_{[k],f}$, where $[k]$ is $k$ evaluated mod $2$.


We now specify that each wave $V_I$, $I = (k,f)$, has frequency support near $\la n_{[I]}(1,0,0)$ by setting
\ali{
V_I^\ell = \cPI[e^{i \la \xi_I} v_I^\ell], \quad
\widehat{\cPI[U^\ell]}(m) = \widehat{\psi}\left(\fr{m}{n_{[I]} \la}\right) \left[ \hat{U}^\ell(m) -  \fr{(\hat{U}(m) \cdot m)}{|m|^2} m^\ell \right], \quad m^\ell \in \Z^3 \cong \widehat{\T^3}. \label{eq:HfreqProjectiondef}
}
In this way, each $V_I$ is divergence free, since the operator $\cPI$ in \eqref{eq:HfreqProjectiondef} is simply the usual Leray projection to mean-zero divergence-free vector fields but cut off to frequencies near $\la n_{[I]}(1,0,0)$.  Note that \eqref{eq:HfreqProjectiondef} may be expressed in physical space as a convolution operator
\ali{
V_I^\ell &= \int_{\R^3} \VR_I^b(x - h) K_{Ib}^\ell(h) dh, \qquad K_{Ib}^\ell(h) = (n_{[I]}\la)^{3} \mathring{K}_b^\ell(n_{[I]} \la h) \label{eq:convolution}
}
where $\mathring{K}_b^\ell(h)$ the Kernel associated to $\PP[ \psi \ast \cdot ]$, which is a Schwartz, matrix-valued function on $\R^3$.  Note we have suppressed the dependence of $\cPI$ and $K_{Ib}^\ell$ on $\la = B_\la N^* \Xi$ in our notation.


We choose the initial data $\hat{\xi}_I(x)$ for the phase functions in \eqref{eq:phaseFunc} to be $\hat{\xi}_I(x) := n_{[I]} (1, 0, 0) \cdot x$ so that $\la \hat{\xi}_I(x)$ is well-defined on $x \in \T^3 = \R^3 / \Z^3$ modulo integer multiples of $2 \pi$.  We will impose the following properties to ensure $\nb \xi_I$ remains sufficiently close to $\nb \hat{\xi}_I = n_{[I]}(1,0,0)$ on the support of $v_I$:
\ali{
\co{\nb \xi_I - \nb \hxii_I} \unlhd (1/4)\co{~| \nb \hxii_I|~}, \qquad \co{~|\nb \xi_I|^{-1}} \unlhd 2 \co{~|\nb \hxii_I|^{-1}}. \label{eq:importantnbxiprops}
}

The amplitudes $v_I^\ell$ for $I = (k,f)$ have the form
\ali{
v_I^\ell &= e_I^{1/2}(t) \eta_I(t) \ga_I  \tilde{f}_I^\ell. \label{eq:amplitudeForm}
}
Here $e_I(t)$ is either the function introduced in \eqref{eq:eDef} above if $I \in \dmd{\II}$ or a different function $\bar{e}(t)$ to be introduced later if $I \in \ovl{\II}$, the $\ga_I = \ga_I(t,x)$ are real-valued scalar functions (the ``coefficients''), the $\eta_I(t)$ are time cutoffs, and the $\tilde{f}_I^\ell$ are vector fields of size $\sim 1$ that take values in $\langle \nb \xi_I \rangle^\perp$ to ensure that $v_I^\ell$ does also.  The time cutoffs have a time-scale $\ostu$ of the form
\ali{
\ostu := \ost{b} \Xi^{-1} e_v^{-1/2}, \qquad \ost{b} = b_0 (e_v^{1/2}/(e_\vp^{1/2} B_\la N))^{1/2}, \label{eq:timeScale}
}
Here $b_0 \leq 1$ is a positive constant that will be chosen in Proposition~\ref{prop:transportEstimates} below  to ensure \eqref{eq:importantnbxiprops} and depends only on the $A_0$ in \eqref{eq:vepBasicBd}.  (The above choice of $\ost{b}$ will optimize the main error terms.)  

We construct $\eta_I$ differently for $I \in \II_R$ than for $I \in \II_\vp$.  Namely, let $\bar{\eta} \colon \R \to \R$ be\footnote{For example, starting with any smooth $\tilde{\eta}(t)$ with $\suppt \tilde{\eta} \subseteq [-1,1]$ with $0 \leq \tilde{\eta}(t) \leq 1$ for all $t$ and $\tilde{\eta}(t) = 1$ on $[-2/3, 2/3]$ one can set $\bar{\eta}(\bar{t}) = \tilde{\eta}(\bar{t}) / (\sum_{k \in \Z} \tilde{\eta}^6(\bar{t} - k))^{1/6}$, noting that the denominator is bounded below and smooth.} a smooth cutoff supported in $\suppt \bar{\eta} \subseteq [-1,1]$ such that $0 \leq \bar{\eta}(\bar{t}) \leq 1$ and $\sum_{k \in \Z} \bar{\eta}^6(\bar{t} - k) = 1$ for all $\bar{t} \in \R$.  For $I \in \II_R$, $I = (k,f)$, we set $\eta_I(t) = \bar{\eta}^3(\ostu^{-1}(t - k \ostu))$, whereas for $I \in \II_\vp$ we set $\eta_I(t) = \bar{\eta}^2(\ostu^{-1}(t - k \ostu))$.  In this way the $\eta_I$ form part of a quadratic partition of unity for $I \in \II_R$ as opposed to a cubic partition of unity for $I \in I_\vp$, and the following bounds hold:
\ali{
\mbox{For all $I = (k,f)$, } \suppt \eta_I \subseteq \{ |t - k\ostu| \leq \ostu \}, \qquad \co{\pr_t^r \eta_I} \lsm_r \ost{b}^{-r} (\Xi e_v^{1/2})^r, \quad 0 \leq r \leq 2. \label{eq:etakprops}
}

We make $\tilde{f}_I$ orthogonal to $\nb \xi_I$ by orthogonally projecting some constant vector $\htf_I \in \langle \nb \hxii_I \rangle^\perp$:
\ali{
\tilde{f}_I^\ell(t,x) &= \htf_I^\ell - |\nb \xi_I|^{-2}(\nb \xi_I \cdot \htf_I) \nb^\ell \xi_I . \label{eq:tildeIdef}
}
Note that \eqref{eq:tildeIdef} is well-defined and bounded once \eqref{eq:importantnbxiprops} is verified, and is equal to $\htf_I^\ell$ at the initial time $t(I) = k \ostu$.  The error in approximating $\tlf_I$ by its starting value $\htf_I$ will be written using the hat notation
\ali{
v_I^\ell = \hat{v}_I^\ell + \hat{\de} v_I^\ell, \qquad \hat{\de} v_I^\ell = e^{1/2}(t) \ga_I ( \tlf_I^\ell - \htf_I^\ell ) = - e^{1/2}(t) \ga_I |\nb \xi_I|^{-2} ( (\nb \xi_I - \nb \hxii_I) \cdot \htf_I ) \nb^\ell \xi_I. \label{eq:expressInitForm}
} 

The initial direction $\htf_I$ will depend only on the component $f = f(I) \in \F$ of $I = (k, f)$ in a manner we now describe.  For $I \in  \dmd{\II}_\vp$, the initial direction $\htf_I$ will be chosen from the finite set $\dmd{\BB}_\vp \subseteq \langle (1,0,0) \rangle^\perp$ defined by $\dmd{\BB}_\vp := \{ (0,1,0), (0,0,1) \}$, while for $I \in \ovl{\II}_\vp$ we use a different orthonormal basis for $\ker dx^1$, $\ovl{\BB}_\vp := \{ (0, 1/\sqrt{2}, 1/\sqrt{2}), (0, -1/\sqrt{2}, 1/\sqrt{2}) \}$.  For $I \in \dmd{\II}_R$, the initial direction $\htf_I$ will be chosen from a finite set $\dmd{\BB}_{R} \subseteq \langle (1,0,0) \rangle^\perp = \ker dx^1$ of cardinality $\# \dmd{\BB}_{R} = 3$ with the following properties:
\ali{
 \mbox{ The tensors } (\htf^j \htf^\ell)_{\htf \in \dmd{\BB}_{R}} &\mbox{form a basis for } \ker dx^1 \otimes \ker dx^1, \mbox{ and } \label{eq:prop1basis}\\
 \mbox{ there exists $\ost{\BB}_R \subseteq \ker dx^2$ such that } (\hat{g}^j \hat{g}^\ell)_{\hat{g} \in \ost{\BB}_{R}} &\mbox{form a basis for } \ker dx^2 \otimes \ker dx^2, \mbox{ and } \label{eq:prop2basis}\\
\sum_{\htf \in \dmd{\BB}_{R}} \htf^j \htf^\ell + &\sum_{\hat{g} \in \ost{\BB}_{R}} \hat{g}^j \hat{g}^\ell = \de^{j\ell}. \label{eq:prop3basis}
}
We take for $\dmd{\BB}_R$ a subset of $\langle (1,0,0) \rangle^\perp$ for which $\sum_{\htf \in \BB_{R}} f \otimes f = e_2 \otimes e_2 + (1/2)e_3 \otimes e_3$, and then for $\ost{\BB}_R$ we choose a subset of $\langle (0,1,0) \rangle^\perp$ that satisfies $\sum_{\hat{g} \in \ost{\BB}_{R}} g \otimes g = e_1 \otimes e_1 + (1/2)e_3 \otimes e_3$.  To find these subsets, observe that the quadratic forms on the right hand sides of these equations become the standard inner product on $\R^2$ after choosing appropriate bases for $\langle (1,0,0) \rangle^\perp$ and $\langle (0,1,0) \rangle^\perp$, which itself be written as $f_1 \otimes f_1 + f_2 \otimes f_2 + f_3 \otimes f_3$ for vectors $f_i$ that form the vertices of an equilateral triangle in the plane centered at the origin.  An exercise in linear algebra shows that the properties \eqref{eq:prop1basis}-\eqref{eq:prop3basis} are all satisfied by this construction.  The exact dependence of $\htf_I$ on $I$ and the description of $\htf_I$ for $I \in \ovl{\II}_R$ will now be specified in Section~\ref{sec:solvingCoeffs} below.


\subsection{Solving for the coefficients in the smaller amplitudes} \label{sec:solvingCoeffs}

In this Section and in Section~\ref{sec:solvingMainCoeffs}, we finish describing the shape of the construction including the coefficient functions $\ga_I$, the index sets $\II, \F$ and the assignment of integers $n_{[I]}$.  


Our goal for the $V_I$ with $I \in \II_\vp$ is to obtain an array of frequency cascades such that the low frequency terms arising in \eqref{eq:phiTermLeft} for $(I,J,K) \in \calT$ may cancel out $\vp_\ep$.  Low frequency contributions from the product $(V_I)_\ell V_J^j V_K^\ell$ should arise only in cases where the phases cancel out in the sense that $e^{i \la (\xi_I + \xi_J + \xi_K)} = 1$, which will occur in cases where the waves share the same initial time $t(I) = t(J) = t(K)$ and satisfy $n_{[I]} + n_{[J]} + n_{[K]} = 0$.  We therefore desire that $\F_\Z$ has several trios of numbers $(n_1, n_2, n_3) \in \F_\Z^3$ for which $n_1 + n_2 + n_3 = 0$.  We arrange these features as follows.

We define the set $\F_\vp$ to have the form $\F_\vp := \BB_\vp \times \{ \pm 1 \} \times \{ \tx{``a''}, \tx{``p''} \}$, so each element $f \in \F_\vp$ has the form $f = (\htf, \si, \tx{``r''})$.  The component $\htf \in \BB_\vp = \ovl{\BB}_\vp \amalg \dmd{\BB}_\vp$ specifies the direction of the amplitude in the sense of \eqref{eq:tildeIdef}, 
the ``sign'' $\si \in \{ \pm 1 \}$ distinguishes a wave from its conjugate $\bar{f} = (\htf, - \si, \tx{``r''})$, and the string ``r'' characterizes the ``role'' of the wave: if it is ``active'' or ``passive''.  We achieve our desired frequency cascades by requiring the following properties from our map $n_\cdot$ restricted to $(\Z/2\Z) \times \F_\vp$:
$ $

\noindent $\bullet$  [Conjugate Symmetry] \,  We require $n_{[k],\bar{f}} = - n_{[k],f}$ for all $([k],f) \in (\Z/2\Z) \times \F_\vp$.

\noindent $\bullet$ [High-High-Low Cascades] We require that
\ali{
n_{[k],(\htf, \si, \tx{``p''})} + n_{[k],(\htf, \si, \tx{``p''})} + n_{[k],(\htf, \si, \tx{``a''})} &= 0, \label{eq:cascadeCondition}
}
and we define the set of ``cascade trios'' $\calT_\F \subseteq \F_\vp^3$ to be those trios $(f_1, f_2, f_3) \in \F_\vp^3$ that satisfy $f_1 = (\htf, \si, \tx{``a''})$ and $f_2 = f_3 = (\htf, \si, \tx{``p''})$ or are a permutation of such trios.  
Note that $\# \calT_\F = 4 \cdot 3!$, due to two choices of $e_2, e_3 \in \BB_\vp$, the two possible signs and the 3! possible permutations.

\noindent $\bullet$ [No Undesired Cascades] We require that
\ali{
| n_{[k],f} + n_{[k'],f'} + n_{[k''], f''} | &\geq 4 \label{eq:goodHighCascades}
}
except for the previously specified High-High-Low cascades required in \eqref{eq:cascadeCondition}.  Any injective map $n_\cdot \colon (\Z/2\Z) \times \F_\vp \to \N$ that satisfies these properties and whose image satisfies \eqref{eq:noWrongCascades} suffices for defining the restriction of $n_\cdot$ to $(\Z/2\Z) \times \F_\vp$.

We now explain how to choose coefficients $\ga_I, I \in \II_\vp$ to eliminate part of the unresolved flux current $\vp_\ep$ using the frequency cascades in \eqref{eq:phiTermLeft}.  We say that a trio $(I,J,K) \in \II^3$ belongs to the set $\calT$ in \eqref{eq:phiTermLeft} if $I, J, K \in \II_\vp$ have the same initial time $t(I)$ and $(f(I), f(J), f(K)) \in \calT_\F$ as defined in \eqref{eq:cascadeCondition}.  
We will require that the leading order part of \eqref{eq:phiTermLeft} in $\DDR_{\vp L}$ cancels except for a term taking values in $\ker dx^2$; namely we require that
\ali{
\sum_{I, J, K \in \calT} (\VR_I)_\ell \VR_J^j \VR_K^\ell + \vp_\ep^j = \sum_{I, J, K \in \calT} v_{I\ell} v_J^j v_K^\ell + \vp_\ep^j &= \vp_{(2)}^j + \mbox{Lower order terms} \label{eq:cubicEqn}
}
pointwise for all $(t,x)$, where $\vp_{(2)}^j$ takes values in $\ker dx^2$.  Here we use that, for $I, J, K \in \calT$, the term $(\VR_I)_\ell \VR_J^j \VR_K^\ell = (v_I \cdot v_K) v_J^j$ is purely low frequency, since $e^{i \la (\xi_I + \xi_J + \xi_K)} = 1$ holds for all $(t,x)$.  Namely, the identity holds at the initial time $t = t(I)$ and persists for all time by uniqueness of solutions to regular transport equations and by \eqref{eq:phaseFunc},\eqref{eq:cascadeCondition} and the definition of $\hxii_I$ following \eqref{eq:convolution}.


The problem \eqref{eq:cubicEqn} will be solved separately by the two different types of waves, $\calT = \ovl{\calT} \amalg \dmd{\calT}$ where $\ovl{\calT} = \calT \cap (\ovl{\II}_\vp)^3, \ovl{\calT} = \calT \cap (\dmd{\II}_\vp)^3$.  The larger waves $I \in \ovl{\II}_\vp$ will be assigned to the principal term $\vp_{[1\ep]}^j$, while the waves $I \in \dmd{\II}_\vp$ will be assigned to the general term $\vp_{[G\ep]}^j$, which though smaller may have slightly larger support in time than $\vp_{[1\ep]}^j$.  We choose linear maps $\pi_{(1)}$ and $\pi_{(2)}$ on $\R^3$ such that $\mbox{Rng } \pi_{(1)} \subseteq \langle (1, 0, 0) \rangle^\perp$, $\mbox{Rng } \pi_{(2)} \subseteq \langle (0,1, 0) \rangle^\perp$, and $\pi_{(1)} + \pi_{(2)} = \mbox{Id}$ on $\R^3$ to isolate the part $\vp_{[G1]}^j := \pi_{(1)} \vp_{[G]}^j$ taking values in $\ker dx^1$.  We will then achieve \eqref{eq:cubicEqn} by solving separately the equations 
\ali{
\sum_{I, J, K \in \ovl{\calT}} v_{I\ell} v_J^j v_K^\ell &= - \vp_{[1]}^j + \mbox{Lower order terms,} \label{eq:ovlCubic} \\
\sum_{I, J, K \in \dmd{\calT}} v_{I\ell} v_J^j v_K^\ell &= - \vp_{[G1]}^j + \mbox{Lower order terms.} \label{eq:dmdCubic}
}
We focus first on \eqref{eq:dmdCubic}.  We use the partition of unity property for the $\bar{\eta}(\cdot)$ to decompose $\vp_{[G1]}^j = \sum_{k \in \Z} \bar{\eta}^6(\ostu^{-1}(t - k \ostu)) \vp_{[G1]}^j$ and recall the expression $v_I^j = e_I^{1/2}(t) \bar{\eta}^2(\ostu^{-1}(t - t(I))) \ga_I \htf_I^j + \hat{\de} v_I^j$.  Factoring out the cutoffs and $e(t)^{3/2}$ from \eqref{eq:dmdCubic}, we will achieve \eqref{eq:dmdCubic} if for all $k \in \Z$ we solve
\ali{
\sum_{(f_1, f_2, f_3) \in \dmd{\calT}_\F} \ga_{(k,f_1)} \ga_{(k,f_2)} \ga_{(k,f_3)} \htf_2^j (\htf_1 \cdot \htf_3) &= - e^{-3/2}(t) \vp_{[G1]}^j  \label{eq:localEqn}
}
pointwise on $[k \ostu - \ostu, k \ostu + \ostu] \cap \tilde{I} \times \T^3$.
Note that, since each $\htf_i$ is either $e_2$ or $e_3$, the sum in \eqref{eq:localEqn} has only two distinct terms, both repeated $2 \cdot 3!$ times.  One takes values in the $e_2$ direction and involves some $f_i = (e_2, \pm 1, \tx{``a''})$, and the other term points in the $e_3$ direction and involves $(e_3, \pm 1, \tx{``a''})$.  With this observation, we can solve \eqref{eq:localEqn} if we make $\ga_I = K_0^{-1/2}$ constant for every ``passive wave'' (those for which $f(I) = (\htf_I, \si, \tx{``p''})$), then for ``active waves'' (those for which $f(I) = (\htf_I, \si, \tx{``a''})$) we set
\ali{
\ga_I(t,x) &= - K_0 e^{-3/2}(t) (\htf_I \cdot \vp_{[G1]}) / (2 \cdot 3!). \label{eq:solvedForvpCoeffs}
}
We will also solve \eqref{eq:ovlCubic} using the same approximation $v_{I\ell} v_J^k v_K^j = \hat{v}_{I\ell} \hat{v}_J^k \hat{v}_K^j + (\hat{\de} v_I)_\ell \hat{v}_J^k \hat{v}_K^j + \ldots$ for $I \in \ovl{\II}_\vp$, which leaves the low frequency part of \eqref{eq:phiTermLeft} in the form
\ali{
\begin{split} \label{eq:vpLdef}
\DDR_{\vp L} = \nb_j \vp_L^j, \quad \vp_L^j :=&\, 
\vp_{(2)}^j + \sum_{I, J, K \in \calT} (\hat{\de} v_I)_\ell v_J^j v_K^\ell + \mbox{similar } O(\hat{\de} v_I \cdot v_J \cdot v_K ) \mbox{ terms}  \\
&+ \sum_{I, J, K \in \calT} (\de V_I)_\ell \VR_J^j \VR_K^\ell + \mbox{similar } O(\de V \cdot V \cdot V) \mbox{ terms},
\end{split} \\
\vp_{(2)}^j &:= \pi_{(2)} \vp_{[G\ep]}^j. \label{eq:secondComponentVp}
}
We also observe that the above choices of $\ga_I$ for $I \in \dmd{\II}_\vp$ lead to the estimate
\ali{
\max_{I \in \II_\vp} \co{ \ga_I } &\leq A K_0^{-1/2} \label{eq:vpcoeffBd}
}
for some absolute constant $A$, where $K_0$ is the constant in the lower bound in \eqref{eq:eDef}.

The construction of $\ga_I$ for $I \in \dmd{\II}_R$ is analogous to that for $\dmd{\II}_\vp$ except that the goal is to cancel out a part of the stress $R_{\ep}$ appearing in \eqref{eq:stressTerm1} that takes values in $\ker dx^1$.  To isolate this part, choose linear maps $\pi_{[1]}$, $\pi_{[2]}$, $\pi_{[3]}$ on the vector space $\R^3 \otimes \R^3$ such that $\tx{Rng } \pi_{[l]} \subseteq \ker dx^l \otimes \ker dx^l$ for $l = 1, 2, 3$ and such that $\pi_{[1]} + \pi_{[2]} + \pi_{[3]} = \mbox{Id}$ on $\R^3 \otimes \R^3$, then decompose $R_\ep = R_{[1*]}^{j\ell} + R_{[2*]}^{j\ell} + R_{[3*]}^{j\ell}$ where
\ali{
R_{[l*]}^{j\ell} = R_{[l\ep]}^{j\ell} + \pi_{[l]} R_{[G\ep]}^{j\ell}, \quad \mbox{ for } l = 1, 2, \qquad 
R_{[3*]}^{j\ell} = \pi_{[3]} R_{[G\ep]}^{j\ell}. \label{eq:threeStressTerms}
}
Here $R_{[1\ep]}, R_{[2\ep]}$ and $R_{[G\ep]}$ refer to the decomposition of $R_\ep$ referenced after line \eqref{eq:Rmolldef}, which will be defined in Section~\ref{sec:regPrelimBds}.  In particular, we have that $\mbox{Rng } R_{[l*]} \subseteq \ker dx^l \otimes \ker dx^l$ for all $l = 1, 2, 3$.  From \eqref{eq:prop3basis}, the pressure term similarly decomposes into $\de^{j\ell} = \de_{[1]}^{j\ell} + \de_{[2\ast]}^{j\ell}$, where
\ali{
\de_{[1]}^{j\ell} := \sum_{\hat{f} \in \dmd{\BB}_R} \hat{f}^j \hat{f}^\ell \in \ker dx^1 \otimes \ker dx^1, \qquad \de_{[2\ast]}^{j\ell} := \sum_{\hat{g} \in \ost{\BB}_R} \hat{g}^j \hat{g}^\ell \in \ker dx^2 \otimes \ker dx^2. \label{eq:de12def}
}
We will now arrange that the remaining part of the term in \eqref{eq:stressTerm1} will have the form
\ali{
\sum_I V_I^j \overline{V_I^\ell} + P \de^{j\ell} + R_\ep^{j\ell} &= P \de_{[2]}^{j\ell} + R_{[2*]}^{j\ell} + R_{[3*]}^{j\ell} + \mbox{Small Error Terms} \label{eq:wantCanceling} 
}
by ensuring that
\ali{
\sum_{I \in \II_R} V_I^j \overline{V_I^\ell} &= - P \de_{[1]}^{j\ell} - R_{[1\ast]}^{j\ell} - \sum_{I \in \II_\vp} V_I^j \overline{V_I^\ell}  + \mbox{Small Error Terms}. \label{eq:oneComponent}
}
Yet again we distinguish $\II_R = \ovl{\II}_R \amalg \dmd{\II}_R$ and obtain \eqref{eq:oneComponent} by solving separately
\ali{
\sum_{I \in \ovl{\II}_R} v_I^j \overline{v_I^\ell} &= - R_{[1\ep]}^{j\ell} - \sum_{I \in \ovl{\II}_\vp} v_I^j \overline{v_I^\ell}  + \mbox{Small Error Terms}, \label{eq:stressEqnovl} \\
\sum_{I \in \dmd{\II}_R} v_I^j \overline{v_I^\ell} &= - P \de_{[1]}^{j\ell} - R_{[G1]}^{j\ell} - \sum_{I \in \dmd{\II}_\vp} v_I^j \overline{v_I^\ell}  + \mbox{Small Error Terms}, \label{eq:stressEqndmd}
}
where $R_{[G1]}^{j\ell} := \pi_{[1]} R_{[G\ep]}^{j\ell}$.  Here we used that $V_I^\ell = e^{i \la \xi_I} v_I^\ell + \de V_I^\ell$ and discarded the $\de V_I^\ell$ terms.  

We focus now on determining $v_I$, $I \in \dmd{\II}_R$ to solve \eqref{eq:stressEqndmd}.  We define $\dmd{\F}_R$ to be $\dmd{\F}_R := \dmd{\BB}_R \times \{ \pm 1 \}$ so that each $f \in \dmd{\F}_R$ has the form $f = (\htf, \si)$, $\htf \in \dmd{\BB}_R, \si \in \{ \pm 1 \}$ and we let $\bar{f} = (\htf, -\si)$ be the conjugate index.  We specify that $n_\cdot$ restricted to $(\Z/2\Z)\times \dmd{\F}_R$ to be any injective map $n_\cdot \colon (\Z/2\Z)\times \dmd{\F}_R \to \N$ such that the completed function $n_\cdot \colon (\Z/2\Z)\times \dmd{\F} \to \N$, $\dmd{\F} = \dmd{\F}_R \amalg \dmd{\F}_\vp$, is injective and maintains the property \eqref{eq:noWrongCascades}.  We multiply \eqref{eq:stressEqndmd} by the identity $1 = \sum_{k \in \Z} \bar{\eta}^6(\ostu^{-1}(t - k \ostu)) $ and recall that for $I \in \dmd{\II}_R$, $v_I^\ell = e^{1/2}(t) \eta_I \ga_I \htf_I^\ell + \hat{\de} v_I^\ell$, $\eta_I = \bar{\eta}^3(\ostu^{-1}(t - k \ostu))$.  Dropping the $\hat{\de} v_I$ terms and factoring out the cutoffs and $e(t)$ from \eqref{eq:stressEqndmd}, we obtain \eqref{eq:stressEqndmd} if, for all $k \in \Z$ we solve the following equation pointwise on $\tilde{I} \times \T^3$:
\ali{
\label{eq:solvingForgaIrs}
\sum_{f \in \dmd{\F}_R} 2 \ga_{(k, f)}^2 \htf^j \htf^\ell &= (2/3)\de_{[1]}^{j\ell} + \varep^{j\ell}, \\
\begin{split} \label{eq:varepDef}
\varep^{j\ell} \stackrel{\eqref{eq:Pdefn}}{=} e(t)^{-1}\Big[(1/3)(-2\kk_{[G\ep]} + \de_{j\ell}(R_{[2\ep]}^{j\ell} + R_{[G\ep]}^{j\ell})) &- R_{[G1]}^{j\ell} - \sum_{I \in \dmd{\II}_\vp} \hat{v}_I^j \overline{\hat{v}_I^\ell} \Big], \\
 e(t)^{-1}\sum_{I \in \dmd{\II}_\vp} \hat{v}_I^j \overline{\hat{v}_I^\ell} = &\sum_{I \in \dmd{\II}_\vp} \eta_I^2 \ga_I^2 \htf_I^j \htf_I^\ell.
\end{split}
}
Using~\eqref{eq:prop1basis} and \eqref{eq:de12def}, provided $\varep$ is sufficiently small, the quadratic equation admits unique, positive solutions that depend smoothly as functions $\ga_f(\varep)$ of $\varep$, $f \in \dmd{\F}_R$, and that obey the estimate\footnote{Here we use the bounds $\max_l \co{R_{[l\ep]}} + \co{R_{[G\ep]}} \unlhd A_L e_\vp$, $\co{\kk_\ep} \unlhd A_L e_\vp$ and $\co{\vp_\ep} \unlhd A_L e_\vp^{3/2}$, which will be checked in Section~\ref{sec:regPrelimBds}.  }
\ali{
\ga_{(k,f)}(t,x) = \ga_f(\varep) &= 1/\sqrt{3} + O(\co{\varep}), \qquad \co{\varep} \stackrel{\eqref{eq:eDef},\eqref{eq:vpcoeffBd}}{=} O((K_0)^{-1}), \label{eq:gabounds}
}
where $K_0$ is the parameter assumed in the lower bound for $e(t)$ in \eqref{eq:eDef}.  The implicit constants in the $O(\cdot)$ notation are geometric and do not depend on any parameters of the construction, which allows us to guarantee a bound of $1/3 \leq \ga_{(k,f)} \leq 2/3$ for all $f \in \dmd{\F}_R$ by taking $K_0$ to be a sufficiently large constant depending on $L$ (in particular ensuring that the square root is well-defined in solving \eqref{eq:solvingForgaIrs}).

To solve the analogous equation \eqref{eq:stressEqnovl} for the larger waves $I \in \ovl{\II}_R$, we will make a similar approximation of $V_I = e^{i \la \xi_I} v_I + \de V_I$, $v_I = \hat{v}_I + \hat{\de} v_I$, and arrange the main terms to satisfy
\ali{
\sum_{I \in \ovl{\II}_R} \hat{v}_I^j \overline{\hat{v}_I^\ell} &= - R_{[1\ep]}^{j\ell} - \sum_{I \in \ovl{\II}_\vp} \hat{v}_I^j \overline{\hat{v}_I^\ell} \label{eq:principOvlReq}
}
Defining the $v_I$ in this way, and noting $\de V_I \overline{V_I} = \de v_I \overline{v_I}$ leaves an error in  \eqref{eq:stressTerm1} of the form 
\ali{
\label{eq:stressTerm}
\begin{split}
R_S^{j\ell} &= P \de_{[2]}^{j\ell} + R_{[2\ast]}^{j\ell} + R_{[3\ast]}^{j\ell} + \sum_{I \in \II} \left(\de v_I^j \overline{v_I^\ell} + \hat{\de} v_I^j \overline{v_I^\ell}\right) + \mbox{similar } O(\de v_I v_I) \mbox{ and } O(\hat{\de} v_I v_I) \mbox{ terms.}
\end{split}
}
At the same time, the error that remains for the unresolved flux density in \eqref{eq:fltermLow} reduces to
\ali{
\label{eq:leftOverDensity}
\begin{split}
\sum_I \fr{V_I \cdot \overline{V_I}}{2} + \kk_\ep - e(t)  \stackrel{\eqref{eq:Pdefn}}{=}& \fr{1}{2} \Big(\sum_I V_I \cdot \overline{V_I} + \de_{j\ell} ( P \de_{[2\ast]}^{j\ell} + R_\ep^{j\ell} ) \Big) \\
= \kk_L :=& \fr{1}{2} \Big( \de_{j\ell}(P \de_{[2\ast]}^{j\ell} + R_{[2\ast]}^{j\ell} + R_{[3\ast]}^{j\ell}) \Big) 
+ \fr{1}{2} \sum_I \Big( \de v_I \cdot \overline{v_I} + \hat{\de} v_I \cdot \overline{v_I} \Big) \\
+& \mbox{ similar } O(\de v_I v_I ) + O(\hat{\de} v_I v_I) \mbox{ terms.} 
\end{split}
}
Note that the main term is $(1/2)\de_{j\ell}(P\de_{[2\ast]}^{j\ell} + R_{[2\ast]}^{j\ell})$, which is the only term of size $e_R$.  This term will be canceled out in the subsequent stage of the iteration where we use waves with values in $\ker dx^2$.  For now we isolate the main terms in the new stress for the next stage of the iteration by defining
\ali{
\ost{R}_{[2]}^{j\ell} := P\de_{[2\ast]}^{j\ell} + R_{[2\ast]}^{j\ell}, \quad \ost{\kk}_{[2]} := \de_{j\ell}\ost{R}_{[2]}^{j\ell}/2, \quad \kk_L = \ost{\kk}_{[2]} + \kk_{LG}. \label{eq:choseMainNewterm}
}
The terms contained in $\kk_{[LG]}$ from \eqref{eq:leftOverDensity} will be shown to have a smaller order of magnitude $\lsm e_G$.

\subsection{Solving for the coefficients in the larger amplitudes} \label{sec:solvingMainCoeffs}
The task that remains to fully specify the shape of the construction is to explain how to define $v_I$ for $I \in \ovl{\II}$ to tackle the principal terms of equations \eqref{eq:ovlCubic},\eqref{eq:principOvlReq}.  To solve these equations, we make an assumption on the structure of our given $R_{[1]}^{j\ell}$, which is necessary for defining our construction and also consistent with the outcome of designing the smaller amplitudes $v_I$ for $I \in \dmd{\II}$ as above.
  
\begin{defn} \label{defn:wellPrepared} Let $\bar{\de} > 0$ and $M \geq 1$.  We say that a dissipative Euler-Reynolds flow $(v,p,R,\kk,\vp)$ is {\bf $(\bar{\de}, M)$-well-prepared} (for stage $[1]$) with respect to a set of frequency energy levels $(\Xi, e_v, e_\vp, e_R, e_G)$, a pair of intervals $\tilde{I}_{[1]} \subseteq \tilde{I}_{[G]}$ and a smooth non-negative $\bar{e} \colon \tilde{I} \to \R_{\geq 0}$ if 
there is a positive number $\undl{e}_\vp \geq M^{-1} e_\vp$ 
such that the following properties hold for $\de_{[1\ast]} := (e_3 \otimes e_3) + (e_2 \otimes e_2)/2$ and all $0 \leq s \leq 2$:
\ali{
R_{[1]}^{j\ell} = -\bar{e}(t) \de_{[1\ast]}^{j\ell} &+ R_{[1\circ]}^{j\ell} \label{eq:formOfR1de} \\
\suppt (R_{[1\circ]}, \vp_{[1]}) \subseteq \tilde{I}_{[1]},& \quad \suppt \bar{e} \subseteq \tilde{I}_{[G]} \\
\co{D_t^s \bar{e}^{1/2}(t)} &\leq (\Xi e_v^{1/2})^s e_\vp^{1/2}, \qquad\quad \label{eq:givenAmpBds} \\
\bar{e}(t)^{-1} \co{R_{[1\circ]}} + \bar{e}(t)^{-3/2} &\co{\vp_{[1]}} \leq \bar{\de}, \qquad \qquad \mbox{ if } t \leq \sup \tilde{I}_{[1]} + (\Xi e_v^{1/2})^{-1}. \label{eq:barDeAssump} \\
\co{D_t^s [\bar{e}(t)^{-1/2}] } &\leq 10 (\Xi e_v^{1/2})^s \undl{e}_\vp^{-1/2} \qquad \mbox{ if } t \leq \sup \tilde{I}_{[1]} + (\Xi e_v^{1/2})^{-1} \label{eq:underBdsEninc} \\
\undl{e}_\vp^{1/2} \co{\nb_{\va} D_t^r R_{[1\circ]} } + \co{\nb_{\va} D_t^r \vp_{[1]} } &\leq \Xi^{|\va|} (\Xi e_v^{1/2})^r \undl{e}_\vp^{3/2}, \qquad 0 \leq r \leq 1, \quad 0 \leq r + |\va| \leq L. \label{eq:cicVp1bds}
}
\end{defn}
\noindent Alternatively, we may allow for \eqref{eq:formOfR1de} to hold with $\de_{[1]}^{j\ell}$ in place of $\de_{[1\ast]}^{j\ell}$ in the definition of $(\bar{\de},M)$-well-preparedness for stage 1 (which will be convenient for the proof of Lemma~\ref{lem:prepareInitLevels} below).  We similarly define $(\bar{\de},M)$-well-prepared for stage $[l]$ by replacing $(1,2,3)$ with $(l, l{+}1, l{+}2)$ mod $3$ where appropriate.  Note that we do not make any monotonicity assumption on $\bar{e}(t)$ in contrast to \eqref{eq:eDef}.  

Having an assumption of the type \eqref{eq:barDeAssump}, which causes the term $\bar{e}(t) \de_{[1\ast]}^{j\ell}$ to dominate the error, will be necessary for achieving any gain in the size of the error, since otherwise we would require adding an additional term of size $e_\vp$ into the pressure that would generate additional error terms of size $e_\vp$ that take values outside of $\ker dx^1 \otimes \ker dx^1$.  

We will use the function $\bar{e}(t)$ of Definition~\ref{defn:wellPrepared} as our choice of $e_f(t)$ appearing in the amplitude $v_I$ in \eqref{eq:amplitudeForm} for $f = f(I) \in \ovl{\F} = \ovl{\F}_R \amalg \ovl{\F}_\vp$.  With this choice, equation \eqref{eq:ovlCubic} will be satisfied (with error terms involving $\hat{\de} v_I$) by solving the equation analogous to \eqref{eq:localEqn}
\ali{
\sum_{(f_1, f_2, f_3) \in \ovl{\calT}_\F} \ga_{(k,f_1)} \ga_{(k,f_2)} \ga_{(k,f_3)} \htf_2^j (\htf_1 \cdot \htf_3) &= - \bar{e}^{-3/2}(t) \vp_{[1\ep]}^j. \label{eq:coeffOvl}
}
Similarly to the argument after \eqref{eq:solvedForvpCoeffs} and using $\bar{\de}$-well-preparedness, we may solve \eqref{eq:coeffOvl} by taking coefficients for passive waves that are constant, $\ga_I = \bar{\de}^{1/3}$ if $f(I) \in \ovl{\BB}_\vp \times \{ \pm 1 \} \times \{ \tx{``p''}\}$, while we choose active waves of the form 
\ali{
\ga_I &= - \bar{\de}^{-2/3} \bar{e}(t)^{-3/2} (\htf_I \cdot \vp_{[1\ep]}) / (2\cdot 3!). \label{eq:othGaIform}
}
Using \eqref{eq:barDeAssump}, we then have a bound of $\co{ \ga_I } \leq A \bar{\de}^{1/3}$ for all $I \in \ovl{\II}_\vp$ for some absolute constant $A$.  

Turning now to $I \in \ovl{\II}_R$, we will solve \eqref{eq:principOvlReq} by ensuring that for all $k \in \Z$ the following equation analogous to \eqref{eq:solvingForgaIrs} holds pointwise on $\tilde{I} \times \T^3$:
\ali{
\label{eq:solvingForgaIrsOvl}
\sum_{f \in \ovl{\F}_R} 2 \ga_{(k, f)}^2 \htf^j \htf^\ell &= \de_{[1\ast]}^{j\ell} + \ovl{\varep}^{j\ell}, \\
\begin{split} \label{eq:varepDefOvl}
\ovl{\varep}^{j\ell} \stackrel{\eqref{eq:Pdefn}}{=} \bar{e}(t)^{-1}\Big[- R_{[\circ\ep]}^{j\ell} & - \sum_{I \in \ovl{\II}_\vp} \hat{v}_I^j \overline{\hat{v}_I^\ell} \Big], \\
 \bar{e}(t)^{-1}\sum_{I \in \ovl{\II}_\vp} \hat{v}_I^j \overline{\hat{v}_I^\ell} = &\sum_{I \in \ovl{\II}_\vp} \eta_I^2 \ga_I^2 \htf_I^j \htf_I^\ell.
\end{split}
}
Here $R_{[\circ \ep]}^{j\ell}$ is a regularization of $R_{[1\circ]}$ defined in Section~\ref{sec:regPrelimBds} below, which is assumed to satisfy $\co{R_{[\circ \ep]}^{j\ell}} \leq A_L \co{R_{[1\circ]}}$ with $A_L$ depending only on $L$.  Using the bound \eqref{eq:barDeAssump} in the definition of $\bar{\de}$-well-preparedness, we may then estimate \eqref{eq:varepDefOvl} by
$\co{\ovl{\varep}} \leq A_L \bar{\de}^{2/3}$ 
where $A_L$ depends only on $L$.

To solve \eqref{eq:solvingForgaIrsOvl} we first choose a subset $\ovl{\BB}_R \subseteq \ker dx^1$ of cardinality $\# \ovl{\BB}_R = 3$ such that the tensors $(\hat{f}^j \hat{f}^\ell)_{\hat{f} \in \ovl{\BB}_R}$ form a basis for $\ker dx^1 \otimes \ker dx^1$ and such that $\sum_{\hat{f} \in \ovl{\BB}_R} \hat{f}^j \hat{f}^\ell = \de_{[1\ast]}^{j\ell}$.  This subset may be constructed as in the argument following line \eqref{eq:prop3basis}.  For $\bar{\de}$ sufficiently small depending on $L$, assuming $\ga_{(k,f)} = \ga_{(k,\bar{f})}$ are equal, we obtain a solution to \eqref{eq:solvingForgaIrsOvl} by first solving the linear equation for the $\ga_I^2$ and then taking square roots to obtain the $\ga_I$, which will obey the bounds $\ga_{(k,f)} = 1/2 + O(\co{\ovl{\varep}}) = 1/2 + O(\bar{\de}^{2/3})$.  For $\bar{\de}$ sufficiently small depending on $L$, the solutions obtained in this way satisfy $1/4 \leq \ga_{(k,f)} \leq 1$ pointwise, and take the form $\ga_{(k,f)}(t,x) = \ga_f(\ovl{\varep})$ for $\ga_f(\cdot)$ one of finitely many functions indexed by $f \in \ovl{\FF}_R$ that are smooth in $\ovl{\varep}$ and defined on a ball of radius $A_L \bar{\de}^{2/3} \geq 2 \co{\ovl{\varep}}$.  

At this point, we are able to fix the final choice of $\bar{\de}$ to be sufficiently small depending on $L$ so that the properties of the preceding paragraph all remain satisfied.  This choice completes our description of the shape of the construction.

\subsection{Regularizations, preliminary bounds and the energy increment} \label{sec:regPrelimBds}
In this section we define regularizations $(v_\ep, p_\ep, R_\ep, \kk_\ep, \vp_\ep)$ of the given dissipative Euler-Reynolds flow $\vprk$ and obtain some preliminary estimates for the construction.

We define the coarse scale velocity field $v_\ep^\ell$ by mollifying $v^\ell \mapsto v_\ep^\ell = \chi_\ep \ast \chi_\ep \ast v^\ell$ in the spatial variables using a Schwartz kernel with length scale $\ep_v$: $\chi_\ep(h) = \ep_v^{-3} \chi_1(\ep_v^{-1} h),$ $\chi_1 \colon \R^3 \to \C$.  We assume that $\chi_\ep$ has compact Fourier support, $\supp \hat{\chi}_1 \subseteq \{ |\xi| \leq 2 \} \subseteq \widehat{\R^3}$, $\hat{\chi}(\xi) = 1$ for $|\xi| \leq 1$, which implies that $\Fsupp v_\ep \subseteq \{ |\xi| \leq 2 \ep_v^{-1} \}$ and $\Fsupp ( v - v_\ep) \subseteq \{ |\xi| \geq \ep_v^{-1} \}$ in $\widehat{\R^3}$.  

This choice leads to the following estimates (see \cite[Lemma 14.1]{isett} for a proof of \eqref{eq:vminvepBd}) 
\ali{
\co{\nb_{\va} v_\ep} &\leq A_0 \co{\nb_{\va} v }, \qquad 1 \leq |\va| \leq L \label{eq:vepBasicBd} \\
\co{(v - v_\ep)} &\lsm (\ep_v)^L \co{\nb^L v }. \label{eq:vminvepBd}
}
Inequality \eqref{eq:vepBasicBd} follows with $A_0 = \| \chi_1 \|_{L^1(\R^3)}^2$ by commuting the spatial derivative with the mollifier.


We choose the length scale $\ep_v$ of the form $c_{1} [ (N)^{1/L} \Xi ]^{-1}$ to guarantee the following bound on the main term appearing in \eqref{eq:Rmolldef}:
\ali{
\co{v - v_\ep} &\lsm c_1^L (e_v^{1/2}/N) \label{eq:vMinVep} \\
\Big\|\sum_I [ (v^j -v_\ep^j) e^{i \la \xi_I} v_I^\ell &+ e^{i \la \xi_I} v_I^j (v^\ell - v_\ep^\ell) ] \Big\|_{C^0} \lsm  \co{ v - v_\ep } \max_I \co{|v_I|} \notag \\
&\lsm A_1 \co{v - v_\ep} e_\vp^{1/2} \leq  \fr{e_v^{1/2} e_\vp^{1/2}}{1000 N}. \label{eq:vMinusVepbd}
}
Examining Sections~\ref{sec:constrShape}-\ref{sec:solvingCoeffs}, the constant $A_1$ (and hence the choice of $c_1$) depends only on the constant appearing in the upper bound $\co{e^{1/2}(t)} \leq K_1 e_{\undvp}^{1/2} \leq K_1 e_\vp^{1/2}$ of \eqref{eq:eDef}, since all the terms $\htf_I, \ga_I$ and $\eta_I$ appearing in \eqref{eq:amplitudeForm} are uniformly bounded by geometric constants while $\co{\bar{e}^{1/2}} \leq e_\vp^{1/2}$.  The constant $A_1$ will thus be determined in line \eqref{eq:enIncdef} below, where we specify the function $e^{1/2}(t)$.


For the remaining fields $F \in (p, R, \kk, \vp)$, we first define a spatial mollification $F_{\ep_x}$ by setting $F_{\ep_x} := \chi_{\ep_x} \ast \chi_{\ep_x} \ast F$, with $\chi_{\ep_x}(h) = \ep_x^{-3} \chi_1(\ep_x^{-1} h)$ a Schwartz kernel on $\R^3$ with integral $1$ that obeys the vanishing moment condition $\int_{\R^3} h^{\va} \chi_{\ep_x}(h) dh = 0$ for all multi-indices $1 \leq |\va| \leq L$.  We then define $F_\ep$ using the coarse scale flow of $v_\ep$ to mollify in time:
\ali{
F_\ep = \eta_{\ep_t} \ast_{\Phi} F_{\ep_x} &:= \int_\R F_{\ep_x}(\Phi_s(t,x)) \eta_{\ep_t}(s) ds. 
\label{eq:mollifyAlongFlow}
}
Here $\eta_{\ep_t}(s) = \ep_t^{-1} \eta_1(\ep_t^{-1} s)$ is a smooth, positive mollifying kernel with compact support in the interval $|s| \leq \ep_t \leq \Xi^{-1} e_v^{-1/2}/2$ and $\int_\R \eta_{\ep_t}(s) ds = 1$, while $\Phi_s(t,x)$ is the flow map of $\pr_t + v_\ep \cdot \nb$ (the {\bf coarse scale flow} of $v_\ep$), which takes values in $\R \times \T^3$ and is defined as the unique solution to the ODE
\ali{
\Phi_s(t,x) = (t + s, \Phi_s^i(t,x)), \quad \Phi_0(t,x) = (t,x), \quad \fr{d}{ds}\Phi_s^i(t,x) = v_\ep^i(\Phi_s(t,x)). \label{eq:coarseScaleFlow}
}
There is a small issue that if $F$ is supported near the boundary of $\tilde{I} \times \T^3$ in $\R \times \T^3$, the expression \eqref{eq:mollifyAlongFlow} is not well-defined if $t+s$ exists the interval $\tilde{I}$.  We can get around this issue by choosing two different mollifiers $\eta^+_{\ep_t}(s)$ and $\eta^-_{\ep_t}(s)$ supported respectively in $0 < s < \ep_t$ and $-\ep_t < s < 0$, then setting
\ali{
F_\ep(t,x) &= \bar{\chi}(t) \eta^+_{\ep_t} \ast_\Phi F_{\ep_x} + (1 - \bar{\chi}(t)) \eta^-_{\ep_t} \ast_\Phi F_{\ep_x} \label{eq:timeCutoffFep}
}
with $\bar{\chi} \in C^\infty(\tilde{I})$ equal to $\bar{\chi}(t) = 1$ for $t \leq \inf \tilde{I} + \Xi^{-1} e_v^{-1/2}$, equal to $\bar{\chi}(t) = 0$ for $t \geq \sup \tilde{I} - \Xi^{-1} e_v^{-1/2}$, taking values between $0 \leq \bar{\chi}(t) \leq 1$ on $\tilde{I}$ and obeying the bounds $\co{(\pr_t)^r \bar{\chi}} \lsm_r (\Xi e_v^{1/2})^r$.  Such a cutoff $\bar{\chi}$ exists since we have assumed $|\tilde{I}| \geq 8 \Xi^{-1} e_v^{-1/2}$.  We note that if the interval $I$ for the domain of the Euler-Reynolds flow is unbounded, then the use of the cutoff $\bar{\chi}(t)$ is not needed.

For each field $F$ to be regularized, we assume bounds of the form
\ali{
\begin{split} \label{eq:assumedFbounds}
\co{\nb_{\va} F} &\leq \Xi^{|\va|} h_F, \qquad 0 \leq |\va| \leq L \\
\co{\nb_{\va} D_t F} &\leq \Xi^{|\va|+1} e_v^{1/2}h_F, \qquad 0 \leq |\va| \leq L - 1,
\end{split}
}
and we choose $\ep_x$ and $\ep_t$ of the form $\ep_x = c_0 (N^{-1/L} \Xi^{-1})$, $\ep_t = c_0 (N \Xi e_\vp^{1/2})^{-1}$.  
Note that the above construction of $F_\ep$ guarantees the bound $\co{F_\ep} \leq A_L \co{F}$ that was required in Section~\ref{sec:solvingCoeffs}, with the constant $A_L = \| \eta_1 \|_{L^1(\R)} \cdot \| \chi_1 \|_{L^1(\R^3)}^2$ depending only on the choice of smooth functions $\eta_1$ and $\chi_1$ (in particular independent of $c_0$ or $c_1$).  

These choices of parameters lead to the following estimates as in \cite[Section 18.3]{isett}, which determines our choice of $c_0$ depending only on $\eta_1$ and $\chi_1$.
\begin{prop}[Preliminary Estimates] \label{prop:prelimBds} We have the following bounds 
\ali{
\label{eq:velocBd1}
\co{\nb_{\va} v_\ep } &\lsm_{\va} \osn^{(|\va| - L)_+/L} \Xi^{|\va|} e_v^{1/2}, \qquad \mbox{ for all }  |\va| \geq 1, \\
\co{\nb_{\va} \Ddt v_\ep} &\lsm_{\va} \osn^{(|\va| + 1 - L)_+/L} \Xi^{|\va|} \Xi e_v^{1/2}, \qquad \mbox{ for all }  |\va| \geq 0, \label{eq:velocBd2} \\
\co{\nb_{\va} \Ddt ( \nb v_\ep) } &\lsm_{\va} \osn^{(|\va| + 2 - L)_+/L} \Xi^{|\va|} \Xi e_v^{1/2}, \qquad \mbox{ for all }  |\va| \geq 0. \label{eq:velocBd3} 
}
If \eqref{eq:assumedFbounds} holds for $(F, h_F)$, then for all multi-indices $\va$ we have
\ali{
\co{\nb_{\va} F_\ep } + \co{\nb_{\va} F_{\ep_x} } &\lsm_{\va} \osn^{(|\va| - L)_+/L} \Xi^{|\va|} h_F \label{eq:coBdFieldsmoll} \\
\co{\nb_{\va} \Ddt F_\ep} + \co{\nb_{\va} \Ddt F_{\ep_x} } &\lsm_{\va} \osn^{(|\va| + 1 - L)_+/L} \Xi^{|\va| + 1} e_v^{1/2} h_F \label{eq:c0DtbdFmoll} \\
\co{\nb_{\va} \Ddt^2 F_\ep } &\lsm_{\va} \ep_t^{-1} \osn^{(|\va| + 1 - L)_+/L} \Xi^{|\va| + 1} e_v^{1/2} h_F \\
&\lsm_{\va} N^{1 + (|\va| + 1 - L)_+/L} \Xi^{|\va| + 2} e_v^{1/2} e_\vp^{1/2} h_F. \label{eq:c0DddtbdFmoll}
}
Examples include $(F,h_F)^t$ any column of the following formal matrix
\ali{
\mat{c|ccccccc}{ F & R_{[1]} & R_{[2]} & R_{[G]}  & \kk_{[G]} & R_{[1\circ]} & \vp_{[1]} & \vp_{[G]}\\
							 h_F &	e_\vp & e_R &  e_G         & e_R       & \undl{e}_\vp & \undl{e}_\vp^{3/2} & e_{\undvp}^{3/2}  
		}. \label{eq:mollifiedTensors}
}
Furthermore, the constants in \eqref{eq:velocBd1}-\eqref{eq:c0DtbdFmoll} do not depend on the choice of $c_1$ in the definition of $\ep_v = c_1 \osn^{-1/L} \Xi^{-1}$ when the total number of derivatives (counting $\Ddt$) does not exceed $L$.  We also have the following bounds on the mollification errors for $c_0$ sufficiently small depending on $\eta_1$, $\chi_1$
\ali{
\begin{split}
\co{F - F_{\ep_x}} &\leq \fr{h_F}{400 N} \\
\co{F - F_{\ep}} &\leq (e_v^{1/2}/e_\vp^{1/2})\fr{h_F}{200 N}. \label{eq:errorMollBounds}
\end{split}
}
\end{prop}
\begin{proof}  The bounds \eqref{eq:velocBd1}-\eqref{eq:velocBd2} can be found in \cite[Sections 15-16]{isett}.  These bounds imply \eqref{eq:velocBd3} as follows.  Note that $\Ddt (\nb_c v_\ep^\ell) = \nb_c \Ddt v_\ep^\ell - \nb_c v_\ep^b \nb_b v_\ep^\ell$.  The first term obeys \eqref{eq:velocBd3} by \eqref{eq:velocBd2} while the second may be estimated using the product rule and \eqref{eq:velocBd1}
\ALI{
\co{\nb_{\va}[ \nb_c v_\ep^b \nb_b v_\ep^\ell ] } &\lsm \sum_{|\va_1| + |\va_2| = |\va|} \co{\nb_{\va_1}\nb_c v_\ep^b} \co{ \nb_{\va_2} \nb_b v_\ep^\ell ] } \\
&\lsm \sum_{|\va_1| + |\va_2| = |\va|} [N^{(|\va_1| + 1 - L)_+/L} \Xi^{|\va_1| + 1} e_v^{1/2}] [N^{(|\va_2| + 1 - L)_+/L} \Xi^{|\va_2| + 1} e_v^{1/2}] \\
\co{\nb_{\va}[ \nb_c v_\ep^b \nb_b v_\ep^\ell ] } &\lsm N^{(|\va| + 2 - L)_+/L} \Xi^{|\va| + 2} e_v^{1/2}.
}
Here we used the counting inequality \eqref{eq:countingIneq} with $y = L-1 \geq 0$.

The proofs of \eqref{eq:coBdFieldsmoll}-\eqref{eq:c0DddtbdFmoll} with $(F, h_F)$ replaced by $(R,e_R)$ may be found in \cite[Section 18]{isett}.  We note that the second advective derivative estimate \eqref{eq:c0DddtbdFmoll} uses the commutation property
\ali{
\Ddt \eta \ast_\Phi G(t,x) &= \int_{\R} \Ddt G(\Phi_s(t,x)) \eta_{\ep_t}(s) ds = -\int_{\R} G(\Phi_s(t,x)) \eta_{\ep_t}'(s) ds
}
for mollification along the flow proven in \cite[Lemma 18.2]{isett} (see also \cite[Proposition 11.1]{IOnonpd} for an alternative proof).  This calculation explains the appearance of $\ep_t^{-1}$ in \eqref{eq:c0DddtbdFmoll}.  We also note that the bounds \eqref{eq:c0DtbdFmoll}-\eqref{eq:c0DddtbdFmoll} remain undisturbed after multiplying by a time cutoff function $\bar{\chi}(t)$ that obeys $\co{\pr_t \bar{\chi}} \lsm (\Xi e_v^{1/2})^r$ for $0 \leq r \leq 2$, as can be seen using the product rule and $\Xi e_v^{1/2} \leq \ep_t^{-1}$.

To prove \eqref{eq:errorMollBounds} we use the estimate $\co{F - F_{\ep_x}} \lsm \ep_x^L \co{\nb^L F} \leq A_{\chi_1} c_0^L N^{-1} h_F$, the bound $\co{F - F_\ep} \leq \co{F - F_{\ep_x} } + \co{F_{\ep_x} - \eta_{\ep_t} \ast_\Phi F_{\ep_x} }$ and the inequality 
\ALI{
\co{F_{\ep_x} - \eta_{\ep_t} \ast_\Phi F_{\ep_x} } &\leq A_{\eta_1} \ep_t \co{\Ddt F_{\ep_x}} \leq A_{\eta_1} c_0 (e_v^{1/2}/e_\vp^{1/2})(h_F/N), 
}
taking $c_0$ small depending on $\chi_1$ and $\eta_1$.  (See \cite[Section 18.3]{isett} for details.)  Note that the presence of time cutoffs as described in \eqref{eq:timeCutoffFep} does not disturb the estimate.  
\end{proof}
We will also use the following bounds, which result from applying the proofs of \eqref{eq:coBdFieldsmoll}-\eqref{eq:c0DtbdFmoll},\eqref{eq:errorMollBounds} with $(R,e_R)$ replaced by $(p, e_v)$: 
\ali{
\co{p - p_{\ep_x} } &\leq \fr{e_v}{200 \osn} \label{eq:errorMollPress} \\
\co{\nb_{\va} p_{\ep_x}} &\lsm_{\va}  \osn^{(|\va| - L)_+/L} \Xi^{|\va|} e_v \qquad \qquad \mbox{ if } |\va| \geq 1, \mbox{ and} \label{eq:mollxPressbd} \\
\co{\nb_{\va} \Ddt p_{\ep_x}} &\lsm_{\va}  \osn^{(|\va| + 1 - L)_+/L} \Xi^{|\va|+1} e_v^{3/2} \qquad \mbox{ for all } |\va| \geq 0. \label{eq:mollxPressDdt}
}
The observation applied here is that the proofs of these estimates use only the bounds on $\nb_{\va} p$ for $|\va| \geq 1$ and on $\nb_{\va} D_t p$, but do not use any assumption on the $C^0$ norm of $p$.

With the above choice of $\ep_t$, we can also choose the energy increment $e(t)$ in \eqref{eq:eDef}.  We assume that the parameter $K_0$ has been fixed at this point, although so far we have only imposed after line \eqref{eq:gabounds} a lower bound on $K_0$ (depending on $\pi_{[1]}$ and the set $\BB_R$).  The final choice of $K_0$, to be stated in line \eqref{eq:atLastK0} of Section~\ref{sec:mainTermProblemCurrent}, will depend also on the operator $\pi_{[2]}$ and the set $\ovl{\BB}_R$ of Section~\ref{sec:solvingCoeffs}, but will not depend of the choice of $c_1$ in the definition of $\ep_v$.  This choice of $K_0$ will guarantee in particular that
\ali{
K_0^{-1} e_{\undvp}^{-1} \co{\ost{R}_{[2]}} + K_0^{-3/2} e_{\undvp}^{-3/2} \co{\ost{\vp}_{[2]}} &\unlhd \bar{\de}, \label{eq:goalForK0}
}
where $\ost{R}_{[2]}$ was defined in \eqref{eq:choseMainNewterm}, while $\ost{\vp}_{[2]}$ has yet to be defined.

Now let $\tilde{I}_{[G]}$ be the interval in the assumptions of Lemma~\ref{lem:mainLem}.  Let $\tau = (\Xi e_v^{1/2})^{-1}$ and define $e^{1/2}(t)$ as follows by mollifying the characteristic function of a time interval
\ali{
e^{1/2}(t) &= 2(K_0 e_{\undvp})^{1/2} \eta_\tau \ast_t 1_{\leq \sup \tilde{I}_{[G]} + \tau/100}(t). \label{eq:enIncdef}
}
Here $\eta_\tau$ is a smooth, non-negative mollifier in the $t$ variable with compact support in $0 \leq t \leq \tau/200$.  The properties in \eqref{eq:eDef} now follow from this definition.   (To ensure the lower bound holds on the support of $(\kk_\ep, R_\ep, \vp_\ep)$, we use here that $\ep_t \leq \tau \cdot 10^{-3}$ for $c_0$ small enough, since $N \geq (e_v^{1/2}/e_\vp^{1/2})$ from \eqref{eq:Nastdef}.) 

We also record the following immediate consequence of \eqref{eq:coBdFieldsmoll}-\eqref{eq:c0DddtbdFmoll}.
\begin{prop} For $(F, h_F)$ as in Proposition~\ref{prop:prelimBds}, $0 \leq r \leq 2$ and $b_\vp^{-2} := (\osn e_\vp^{1/2}/e_v^{1/2})$, we have
\ali{ \label{eq:compactBdsFep}
\co{\nb_{\va} \Ddt^r F_\ep } &\lsm_{\va} \left[\osn^{(|\va| + 1 -L)_+/L} \Xi^{|\va|}\right]\left[ ( \Xi e_v^{1/2})^r b_\vp^{-2(r - 1)_+}\right] h_F.
}
\end{prop}

With the regularizations defined, we are ready to estimate the components of the construction.

\subsection{Bounding the components of the construction} \label{sec:constructBounds}
In this section we estimate the components of the construction.  All implicit constants in the $\lsm$ notation will be allowed to depend on the constant parameters such as $K_0, c_0, b_0, \ldots$ that were discussed in the preceding Sections~\ref{sec:constrShape}-\ref{sec:regPrelimBds}.

We start by stating bounds for the phase gradients $\nb \xi_I$, which follow from analysis of the equation 
\ali{
(\pr_t + v_\ep^i \nb_i)\nb_a \xi_I = - \nb_a v_\ep^i \nb_i \xi_I, \qquad \nb_a \xi_I(t(I),x) = \nb_a \hxii_I.
}
\begin{prop}[Transport Estimates] \label{prop:transportEstimates} There exists a positive number $b_0 \leq 1$ depending on $A_0$ in \eqref{eq:vepBasicBd} such that \eqref{eq:importantnbxiprops} holds for times $|t - t(I)| \leq \ostu$.  For such times $t$, we have the following estimates
\ali{
\co{\nb_{\va} \nb \xi_I} &\lsm_{\va} \osn^{(|\va| + 1 - L)_+/L} \Xi^{|\va|} \label{eq:spaceDrivPhase} \\
\co{\nb_{\va} \Ddt \nb \xi_I} &\lsm_{\va} \osn^{(|\va| + 1 - L)_+/L} \Xi^{|\va|} (\Xi e_v^{1/2}) \label{eq:firstDerivPhase} \\
\co{ \nb_{\va} \Ddt^2 \nb \xi_I} &\lsm_{\va} \osn^{( |\va| + 2 - L)_+/L} \Xi^{|\va|} (\Xi e_v^{1/2})^2. \label{eq:scndDerivPhase}
}
\noindent  We also have the following bounds for $\hat{\de} \nb \xi_I := \nb \xi_I - \nb \hxii_I$:
\ali{
\co{\nb_{\va} ( \nb \xi_I - \nb \hxii_I ) } &\lsm_{\va} (B_\la \osn)^{|\va|/2} \Xi^{|\va|} \ost{b} \label{eq:c0bdPhasediff} \\
\co{\nb_{\va} \Ddt^r ( \nb \xi_I - \nb \hxii_I ) } &\lsm_{\va} (B_\la \osn)^{|\va|/2} \Xi^{|\va|}(\ost{b}^{-1} \Xi e_v^{1/2})^r \ost{b}, \qquad 0 \leq r \leq 2. \label{eq:Ddtphasediff}
}
\end{prop}
\begin{proof}  Recalling the coarse scale flow $\Phi_s$ from \eqref{eq:coarseScaleFlow}, we have $\fr{d}{ds} |\nb \xi_I|^2(\Phi_s(t,x)) = - \nb_\ell v_\ep^i \nb_i \xi_I \nb^\ell \xi_I$ for $|s| \leq \tau$ (each term being evaluated at $\Phi_s(t,x))$, which implies by a Gronwall inequality argument that
\ali{
e^{-|s| \co{\nb v_\ep} }  \left|~ |\nb \hxii_I|^2 \right| \leq \left|~ |\nb \xi_I|^2(\Phi_s(t,x)) \right| \leq e^{-|s| \co{\nb v_\ep} }  \left|~ |\nb \hxii_I|^2 \right|.
}
Using $|s| \co{\nb v_\ep} \leq b_0 (\Xi e_v^{1/2})^{-1} A_0 \Xi e_v^{1/2} = b_0 A_0$, we have $e^{|s| \co{\nb v_\ep} } \leq 2$ if $b_0 > 0$ is small depending on $A_0$.  From $\fr{d}{ds} \nb_\ell \xi_I(\Phi_s(t(I),x)) = - \nb_\ell v_\ep^i(\Phi_s(t(I),x)) \nb_i \xi_I(\Phi_s(t(I),x))$  we obtain
\ALI{
|\nb \xi_I(\Phi_s(t(I),x)) - \nb \hxii_I | &\leq 2 \ostu \co{\nb v_\ep} |\nb \hxii_I| \leq \ost{b} A_0 |\nb \hxii_I|.
}
Taking $b_0$ small depending on $A_0$, we obtain \eqref{eq:importantnbxiprops}, and also the $|\va| = 0$ case of \eqref{eq:c0bdPhasediff}.

The estimates \eqref{eq:spaceDrivPhase}-\eqref{eq:scndDerivPhase} are shown in \cite[Section 17]{isett}.  (Note that the proof of \eqref{eq:scndDerivPhase} goes through the Euler-Reynolds system and the main term comes from the pressure gradient.)  
The bounds \eqref{eq:spaceDrivPhase}-\eqref{eq:Ddtphasediff} imply the remaining estimates  \eqref{eq:Ddtphasediff} and the $|\va| \geq 1$ case of \eqref{eq:c0bdPhasediff} by comparison using $L \geq 2$, $(B_\la \osn)^{1/2} \geq \ost{b}^{-1}$ and the fact that $\nb \hxii_I$ is constant.
\end{proof}
We next estimate the components of the construction.  We start with the following Lemma:
\begin{lem} \label{lem:accountLem1}  For integers $0 \leq R \leq 2$ and $M \geq 0$, define the weighted norm
\ali{
\bar{H}^{(M,R)}[F] &:= \max_{0 \leq r \leq R} \left\{ \max_{0 \leq |\va| + r \leq M + R}  \fr{\co{\nb_{\va} \Ddt^r F}}{N^{(|\va| + 1 - L)_+/L} \Xi^{|\va|} \ostu^{-r}} \right\}. \label{eq:brhDefn}
}
Then one has the triangle inequality $\bar{H}^{(M,R)}[F + G] \leq \bar{H}^{(M,R)}[F] + \bar{H}^{(M,R)}[G]$ and the product rule
\ali{
\bar{H}^{(M,R)}[F G] &\lsm_M \bar{H}^{(M,R)}[F] \cdot \bar{H}^{(M,R)}[G]. \label{eq:productRule}
}
We also have the following for $\ostu = \ost{b} (\Xi e_v^{1/2})^{-1}$ and $0 \leq s \leq 2$:
\ali{
\begin{split} \label{eq:RcaseForGenbd}
\brh^{(M,R)}[F] &= \max_{0 \leq r \leq R} \ostu^{r} \brh^{(M+R-r,0)}[\Ddt^r F], \qquad 
\brh^{(M,0)}[\Ddt^s F] \leq \ostu^{-s} \brh^{(M,s)}[F]. 
\end{split}
}
\end{lem}
\begin{proof} The triangle inequality follows quickly from the definition of the norm.  The product rule in the case $R = 0$ follows by using inequality \eqref{eq:countingIneq} to obtain for all $0 \leq |\va| \leq M$ that
\ALI{
\co{\nb_{\va}(FG)} &\lsm \sum_{|\vcb| +|\vcc| = |\va|} \co{\nb_{\vcb} F} \co{\nb_{\vcc} G} \\
&\lsm \sum_{|\vcb| +|\vcc| = |\va|} \left(N^{(|\vcb| + 1 - L)_+/L} \Xi^{|\vcb|} \brh^{(M,0)}[F]\right) \left(N^{(|\vcc| + 1 - L)_+/L} \Xi^{|\vcc|} \brh^{(M,0)}[G]\right) \\
&\lsm N^{(|\va| + 1 - L)_+/L} \brh^{(M,0)}[F]\brh^{(M,0)}[G].
}
We note that the properties in \eqref{eq:RcaseForGenbd} follow immediately from the definition of $\brh$ in \eqref{eq:brhDefn}.  We can obtain the product rule in the case $0 < R \leq 2$ from the case $R = 0$ above and \eqref{eq:RcaseForGenbd} by writing 
\ALI{
\brh^{(M,R)}[FG] &\lsm_M \sum_{0 \leq r \leq R}\ostu^{r}\sum_{r_1 + r_2 = r} \brh^{(M+R-r,0)}[\Ddt^{r_1} F \Ddt^{r_2} G] \\
&\lsm_M \sum_{0 \leq r \leq R}\ostu^{r}\sum_{r_1 + r_2 = r} \brh^{(M+R-r,0)}[\Ddt^{r_1} F] \brh^{(M+R-r,0)}[\Ddt^{r_2} G] \\
&\lsm_M \sum_{0 \leq r \leq R}\sum_{r_1 + r_2 = r} (\ostu^{r_1}\brh^{(M+R-r_1,0)}[\Ddt^{r_1} F])(\ostu^{r_2} \brh^{(M+R-r_2,0)}[\Ddt^{r_2} G]) \\
\brh^{(M,R)}[FG] &\lsm_M \brh^{(M,R)}[F] \brh^{(M,R)}[G].
}
This bound concludes the proof of Lemma~\ref{lem:accountLem1}.
\end{proof}

In terms of these weighted norms, we have the following bounds
\begin{prop} \label{prop:prelimSummaries} We have that $\brh^{(M,1)}[\nb v_\ep] \lsm_M \Xi e_v^{1/2}$, and the estimate $\brh^{(M,2)}[F_\ep] \lsm_M h_F$ holds for $(F, h_F)^t$ any column of the formal matrix in line \eqref{eq:mollifiedTensors}.  We also have that $\brh^{(M,2)}[F] \lsm h_F$ for $(F,h_F)^t$ any column of the following formal matrix
\ali{
\mat{c|ccccc}{ F &  \nb \xi_I &  \eta_I & e^{1/2}(t) & \bar{e}^{1/2}(t) & \bar{e}(t)^{-1/2} \\
							 h_F & 1    &   1 &  e_{\undvp}^{1/2} & e_\vp^{1/2} &\undl{e}_\vp^{-1/2} 
		}, \label{eq:otherBoundsStart}
}
where the last estimate for $\bar{e}(t)^{-1/2}$ holds on $\suppt (R_{[\ep \circ]}, \vp_{[1\ep]})$.  
\end{prop}
\begin{proof}  We obtain $\brh^{(M,1)}[\nb v_\ep] \lsm_M \Xi e_v^{1/2}$ from \eqref{eq:velocBd1},\eqref{eq:velocBd3} by noting that $\ostu^{-1} \geq B_\la^{1/2}N^{1/2} (\Xi e_v^{1/2})$, which implies $N^{(|\va| + 2 - L)_+/L} (\Xi e_v^{1/2}) \leq N^{(|\va| + 1 - L)_+/L} \ostu^{-1}$.  

The bound $\brh^{(M,2)}[F_\ep] \lsm_M h_F$ follows from \eqref{eq:compactBdsFep} by noting that $(\Xi e_v^{1/2}) \leq \ostu^{-1} = \ost{b}^{-1} (\Xi e_v^{1/2})$ and $b_\vp^{-2} (\Xi e_v^{1/2})^2 \lsm \ost{b}^{-2} (\Xi e_v^{1/2})^2 = \ostu^{-2}$, where $b_\vp^{-2} = (\osn e_\vp^{1/2}/e_v^{1/2})$ and $\ost{b}^{-2} = b_0^{-2}(B_\la N e_\vp^{1/2}/e_v^{1/2})$.  

The bound $\brh^{(M,2)}[\nb \xi_I] \lsm_M 1$ follows from \eqref{eq:spaceDrivPhase}-\eqref{eq:scndDerivPhase} by noting that $(\Xi e_v^{1/2}) \leq \ostu^{-1}$ and that
\ALI{
N^{(|\va| + 2 - L)_+/L} (\Xi e_v^{1/2})^2 \leq N^{1/2} N^{(|\va| + 1 - L)_+/L} (\Xi e_v^{1/2})^2 \leq N^{(|\va| + 1 - L)_+/L} \ostu^{-2}.
}
for $L \geq 2$.  The bounds for the cutoffs $\eta_I$ and the functions $e^{1/2}$, $\bar{e}^{1/2}$ and $\bar{e}^{-1/2}$ follow from the estimates \eqref{eq:etakprops},\eqref{eq:eDef},\eqref{eq:givenAmpBds},\eqref{eq:underBdsEninc} and $(\Xi e_v^{1/2})\leq \ostu^{-1}$.
\end{proof}
We record also the following useful property of the weighted norm \eqref{eq:brhDefn}.
\begin{prop} \label{prop:controlCombos}  For any operator of the form 
\ali{
D^{(\va, \vec{r})} &= \nb_{\va_1} (\Ddt)^{r_1} \nb_{\va_2} (\Ddt)^{r_2} \nb_{\va_3} \label{eq:mixedSpaceAdvec}
}
with $\va = (\va_1, \va_2, \va_3)$, $|\va| = |\va_1| + |\va_2| + |\va_3|$, $\vec{r} = (r_1, r_2)$ and $0 \leq r := r_1 + r_2 \leq 2$, we have the bound
\ali{
\co{D^{(\va, \vec{r})} F} \lsm_{\va} N^{(|\va|+1-L)_+/L} \Xi^{|\va|} \ostu^{-r} \brh^{(|\va|,r)}[F]. \label{eq:comboBound}
}
\end{prop}
\begin{proof}  The result is immediate if $r_1 + r_2 = 0$.  In general, for $1 \leq r_1 + r_2 \leq 2$, repeatedly applying the commutator identity $\Ddt \nb_a F = \nb_a(\Ddt F) - \nb_a v_\ep^i \nb_i F$ leads to an expansion of the form
\ALI{
D^{(\va, \vec{r})} F &= \sum_{0 \leq m \leq r} \sum_{\vcb, \vcc, s} c_{m,\vcb,\vcc, s} \prod_{i = 1}^{m} \left(\nb_{\vcc_i}\Ddt^{s_i} (\nb v_\ep)\right) \nb_{\vcb} \Ddt^{s_{m+1}} F,
}
where the summation runs over $s_1 + \cdots + s_m + m + s_{m+1} = r$ (the total number of advective derivatives), the empty product is equal to $1$ in the case $m = 0$, and the multi-indices satisfy $|\vcb| + \sum_i |\vcc_i| = |\va|$ (the total number of spatial derivatives).  (We have omitted the summation over the indices of $\nb v_\ep$.)  In particular, at most one advective derivative falls on $\nb v_\ep$, and we obtain the bounds
\ALI{
\co{D^{(\va, r)} F} &\lsm \sum_{0 \leq m \leq 3} \sum_{\vcb, \vcc} \prod_{i = 1}^{m} \left(N^{(|\vcc_i| + 1 - L)_+/L}\ostu^{-s_i} \Xi^{|\vcc_i|} (\Xi e_v^{1/2})\right) \left[ N^{(|\vcb| + 1 - L)_+/L} \ostu^{-s_{m+1}} \Xi^{|\vcb|} \brh_F\right] \\
&\lsm N^{(|\va| + 1 - L)_+/L} \Xi^{|\va|} \ostu^{-r} \brh_F,
}
where $\brh_F := \brh^{(|\va|,r)}[F]$.  Here we used Proposition~\ref{prop:prelimSummaries} for $\brh^{(M,1)}[\nb v_\ep]$, $(\Xi e_v^{1/2})^m \leq \ostu^{-m}$, and the counting inequality \eqref{eq:countingIneq} with $y = L-1 \geq 0$.  \end{proof}

We now proceed to estimate the components of the correction.

\begin{prop} \label{prop:compBds} The following bounds hold uniformly for $I \in \II$ and for both $\mathring{\varep}^{j\ell} \in \{ \varep^{j\ell}, \ovl{\varep}^{j\ell} \}$.
\ali{
\co{\nb_{\va} \mathring{\varep} } + \co{\nb_{\va} \ga_I } + \co{\nb_{\va}\tilde{f}_I} &\lsm_{\va} \osn^{(|\va| + 1 - L)_+/L} \Xi^{|\va|} \label{eq:conbcompbds} \\
\co{\nb_{\va} \Ddt \mathring{\varep} } + \co{\nb_{\va} \Ddt \ga_I } + \co{\nb_{\va} \Ddt \tilde{f}_I} &\lsm_{\va} \osn^{(|\va| + 1 - L)_+/L} \Xi^{|\va|} \ost{b}^{-1} (\Xi e_v^{1/2}) \label{eq:Ddtcompbds} \\
\co{\nb_{\va} \Ddt^2 \mathring{\varep} } + \co{\nb_{\va} \Ddt^2 \ga_I } + \co{\nb_{\va} \Ddt^2 \tilde{f}_I} &\lsm_{\va} \osn^{(|\va| + 1 - L)_+/L} \Xi^{|\va|} \ost{b}^{-2} (\Xi e_v^{1/2})^2 \label{eq:Ddt2compBds}
}
These estimates generalize to operators of the form $D^{(\va, \vec{r})}$ defined as in \eqref{eq:mixedSpaceAdvec}.
\end{prop}

\begin{proof}[Proof of \eqref{eq:conbcompbds}-\eqref{eq:Ddt2compBds} for $\ga_I$, $I \in \II_\vp$]  If $I \in \II_\vp$, then either $I$ is passive (in which case $\ga_I = K_0^{-1/2}$ or $\ga_I = \bar{\de}^{1/3}$ is a constant and the bounds are clear) or $I$ is active and is given as in \eqref{eq:solvedForvpCoeffs} or \eqref{eq:othGaIform}.  On $\supp \vp_1$, $e(t) = 4 K_0 e_{\undvp}$ is a constant so we easily have
$\co{\Ddt^r e^{-3/2}(t)} \lsm (\Xi e_v^{1/2})^r e_{\undvp}^{-3/2} \mbox{ for } 0 \leq r \leq 2$.  Now using Lemma~\ref{lem:accountLem1} and Proposition~\ref{prop:prelimSummaries} we find that for $I \in \dmd{\II}_\vp$,
\ALI{
\brh^{(M,2)}[\ga_I] \lsm_M \brh^{(M,2)}[e^{-3/2}(t) ] \brh^{(M,2)}[\vp_{[G\ep]}] \lsm e_{\undvp}^{-3/2} e_{\undvp}^{3/2} = 1.
}
Similarly, for $I \in \ovl{\II}_\vp$, we have using \eqref{eq:underBdsEninc} the bound
\ALI{
\brh^{(M,2)}[\ga_I] &\lsm_M \brh^{(M,2)}[\bar{e}^{-3/2}(t)] \brh^{(M,2)}[ \vp_{[1]} ] \lsm_M \brh^{(M,2)}[\bar{e}^{-1/2}(t)]^3 \undl{e}_\vp^{3/2} \lsm \undl{e}_{\vp}^{-3/2} \undl{e}_\vp^{3/2} = 1.
}
\end{proof}
\begin{proof}[Proof of \eqref{eq:conbcompbds}-\eqref{eq:Ddt2compBds} for $\varep$]
Let $\tilde{R}_{[1]}^{j\ell} := (1/3)(-2\kk_{[G\ep]} + \de_{j\ell}(R_{[2\ep]}^{j\ell} + R_{[G\ep]}^{j\ell})) - R_{[G1]}^{j\ell}$.  From \eqref{eq:varepDef} we can write $\varep^{j\ell} = e(t)^{-1} \tilde{R}_{[1]}^{j\ell} + \sum_{I \in \dmd{\II}_\vp} \eta_I^2 \ga_I^2 \htf_I^j \htf_I^\ell$.  Now using Lemma~\ref{lem:accountLem1} and Proposition~\ref{prop:prelimSummaries} we find (using $e_{\undvp} = e_\vp^{1/3} e_R^{2/3} \geq e_R$ and that $e(t) \geq e_{\undvp}$ is constant on $\suppt \tilde{R}_{[1]}$) that for $I \in \dmd{\II}_\vp$ we have,
\ALI{
\brh^{(M,2)}[e(t)^{-1} &\tilde{R}_{[1]}^{j\ell} ] \lsm e_{\undvp}^{-1} \brh^{(M,2)}[ \tilde{R}_{[1]} ] \\
&\lsm_M e_{\undvp}^{-1}\left( \brh^{(M,2)}[\kk_{G\ep}] + \brh^{(M,2)}[ R_{[2\ep]}] + \brh^{(M,2)}[R_{[G\ep]}] \right) \lsm_M e_{\undvp}^{-1} e_R \lsm 1.
}
For the latter term in \eqref{eq:varepDef}, we use Lemma~\ref{lem:accountLem1}
\ALI{
\brh^{(M,2)}[\eta_I^2 \ga_I^2 \htf_I^j \htf_I^\ell ] &\lsm_M \brh^{(M,2)}[\eta_I]^2 \brh^{(M,2)}[\ga_I]^2 \lsm 1,
}
where we applied \eqref{eq:conbcompbds}-\eqref{eq:Ddt2compBds} for $\ga_I$, $I \in \II_\vp$.

The bound for $\ovl{\varep}^{j\ell}$ is similar, except the term analogous to $\tilde{R}_{[1]}^{j\ell}$ is simply $-  \bar{e}(t)^{-1}R_{[\ep \circ]}^{j\ell}$, which we bound using Lemma~\ref{lem:accountLem1} and Proposition~\ref{prop:prelimSummaries}
\ALI{
\brh^{(M,2)}[\bar{e}(t)^{-1} R_{[\ep \circ]}^{j\ell} ] &\lsm_M \brh^{(M,2)}[\bar{e}^{-1/2}(t)]^2 \brh^{(M,2)}[R_{[\ep \circ]}^{j\ell}] \lsm_M \undl{e}_\vp^{-1} \cdot \undl{e}_\vp \lsm 1.
}
The term involving $\sum_{I \in \ovl{\II}_\vp} \eta_I^2 \ga_I^2 \htf_I^j \htf_I^\ell$ may be bounded exactly as was done for $\dmd{\II}_\vp$.
\end{proof}
\begin{proof}[Proof of \eqref{eq:conbcompbds}-\eqref{eq:Ddt2compBds} for $\ga_I$, $I \in \II_R$]  We consider $I \in \dmd{\II}_R$ since the case $I \in \ovl{\II}_R$ will be no different.  If $I \in \dmd{\II}_R$ we have that $\ga_I(t,x) = \ga_{f(I)}(\varep)$ is a smooth function of $\varep$.  To bound spatial derivatives, apply the chain rule and product rule repeatedly to obtain
\ALI{
\nb_{\va} [ \ga_f(\varep) ] &= \sum_{k = 0}^{|\va|} \sum_{\va_i} \pr^k \ga_f(\varep) \prod_{i=1}^k \nb_{\va_i} \varep \\
\co{\nb_{\va} [ \ga_f(\varep) ]} &\lsm_{\va} \sum_{k = 0}^{|\va|} \sum_{\va_i} \prod_{i=1}^k[\osn^{(|\va_i| +1-L)_+/L} \Xi^{|\va_i|} ]
}
where the sum ranges over multi-indices with $\sum_i |\va_i| = |\va|$ and the implicit constant depends also on the function $\ga_f(\cdot)$.  Using the elementary inequality $\sum_i (|\va_i| - (L - 1))_+ \leq (|\va| - (L -1))_+$ from \eqref{eq:countingIneq}, we obtain $\co{\nb_{\va} [ \ga_f(\varep) ]} \lsm_{\va} \osn^{(|\va| + 1-L)_+/L} \Xi^{|\va|}$, which we note is equivalent to $\brh^{(M,0)}[ \ga_f(\varep) ] \lsm_M 1$.  The same estimate clearly holds with $\ga_f(\cdot)$ replaced by any of its partial derivatives $\pr \ga_f(\cdot)$. 

To bound advective derivatives, note that schematically the chain rule gives
\ali{
\Ddt[\ga_f(\varep)] &=  \pr \ga_f(\varep) \Ddt \varep \label{eq:firstAdvecComp} \\
\Ddt^2[\ga_f(\varep) ] &= \Ddt[ \pr \ga_f(\varep) ] (\Ddt \varep) + \pr \ga_f(\varep) \Ddt^2 \varep \label{eq:2ndAdvecComp}
}
Applying Lemma~\ref{lem:accountLem1} to \eqref{eq:firstAdvecComp}, we obtain
\ali{
\brh^{(M,0)}[\Ddt[\ga_f(\varep)]] &\lsm_M \brh^{(M,0)}[\pr \ga_f(\varep)] \cdot \brh^{(M,0)}[\Ddt \varep] \notag \\
&\lsm_M 1 \cdot \ostu^{-1}. \label{eq:Advec1Compbd}
}
Applying Lemma~\ref{lem:accountLem1} to \eqref{eq:2ndAdvecComp}, we obtain
\ali{
\brh^{(M,0)}[\Ddt^2[\ga_f(\varep) ]] &\lsm_M \brh^{(M,0)}[\Ddt[\pr \ga_f(\varep)]] \cdot \brh^{(M,0)}[\Ddt \varep] + \brh^{(M,0)}[\pr \ga_f(\varep)] \brh^{(M,0)}[\Ddt^2 \varep] \label{eq:lineWeightAdd1} \\
&\lsm_M (\ostu^{-1})(\ostu^{-1}) + 1 \cdot \ostu^{-2} \lsm_M \ostu^{-2}, \notag
}
which implies our desired bound.  (In line \eqref{eq:lineWeightAdd1}, the result of \eqref{eq:Advec1Compbd} is used with $\ga_f$ replaced by $\pr \ga_f$.) \end{proof}
\begin{proof}[Proof of \eqref{eq:conbcompbds}-\eqref{eq:Ddt2compBds} for $\tilde{f}_I$]  The bound on $\co{\tilde{f}_I}$ is clear from formula \eqref{eq:tildeIdef}.  To bound derivatives, it suffices to estimate the term $|\nb \xi_I|^{-2} (\nb \xi_I \cdot \htf_I) \nb^\ell \xi_I$.  Using Lemma~\ref{lem:accountLem1} we have
\ali{
\brh^{(M,2)}[|\nb \xi_I|^{-2} (\nb \xi_I \cdot \htf_I) \nb^\ell \xi_I] &\lsm_M \brh^{(M,2)}[|\nb \xi_I|^{-2}] \brh^{(M,2)}[\nb \xi_I]^2
}
Proposition~\ref{prop:prelimSummaries} gives $\brh^{(M,2)}[\nb \xi_I] \lsm_M 1$.  
We deduce that $\brh^{(M,2)}[|\nb \xi_I|^{-2}] \lsm_M 1$ as well by the same argument we used to establish $\brh^{(M,2)}[\ga_f(\varep)] \lsm_M 1$ for $I \in \II_R$, but replacing $\varep$ by $\nb \xi_I$ and replacing the smooth function $\ga_f(\cdot)$ by the smooth function $|\cdot|^{-2}$ restricted to a neighborhood of $\nb \hxii_I$.
\end{proof}

From these bounds we obtain the following estimates for the amplitudes of the correction.
\begin{prop} \label{prop:correctAmpBds} The following estimates hold uniformly for $I \in \II$
\ali{
\co{\nb_{\va} v_I} + \ostu \co{\nb_{\va} \Ddt v_I} + \ostu^2 \co{\nb_{\va} \Ddt^2 v_I} &\lsm_{\va} \osn^{(|\va| + 1 - L)_+/L} \Xi^{|\va|} e_\vp^{1/2} \label{eq:ampBounds} \\
\co{\nb_{\va} \hat{v}_I} + \ostu \co{\nb_{\va} \Ddt \hat{v}_I} + \ostu^2 \co{\nb_{\va} \Ddt^2 \hat{v}_I} &\lsm_{\va} \osn^{(|\va| + 1 - L)_+/L} \Xi^{|\va|} e_\vp^{1/2} \label{eq:initAmpBds} \\
\co{\nb_{\va} \de v_I} + \ostu \co{\nb_{\va} \Ddt \de v_I} + \ostu^2 \co{\nb_{\va} \Ddt^2 \de v_I} &\lsm_{\va}  \osn^{(|\va| + 2 - L)_+/L} \Xi^{|\va|} (B_\la \osn)^{-1}e_\vp^{1/2} \label{eq:lowerOrderAmpBds} \\
\co{\nb_{\va} \hat{\de} v_I} + \ostu \co{\nb_{\va} \Ddt \hat{\de} v_I} + \ost{b} \ostu^2 \co{\nb_{\va} \Ddt^2 \hat{\de} v_I} &\lsm_{\va} (B_\la \osn)^{|\va|/2} \Xi^{|\va|} e_\vp^{1/2} \ost{b}. \label{eq:initAmpErrBds}
}
These estimates generalize to operators of the form $D^{(\va, \vec{r})}$ defined as in \eqref{eq:mixedSpaceAdvec}.
\end{prop}
\begin{proof}[Proof of \eqref{eq:ampBounds} and \eqref{eq:initAmpBds}] Applying Lemma~\ref{lem:accountLem1} and Proposition~\ref{prop:compBds} to $v_I^\ell = e_I^{1/2}(t) \eta_I \ga_I \tilde{f}_I^\ell$, we have
\ALI{
\brh^{(M,2)}[v_I] &\lsm_M \brh^{(M,2)}[e_I^{1/2}(t)] \brh^{(M,2)}[\eta_I] \brh^{(M,2)}[\ga_I] \brh^{(M,2)}[\tilde{f}_I] \\
&\lsm_M e_\vp^{1/2} \cdot 1 \cdot 1 \cdot 1 = e_\vp^{1/2},
}
which implies \eqref{eq:ampBounds}.  The same calculation with $v_I^\ell$ and $\tilde{f}_I^\ell$ replaced by $\hat{v}_I^\ell$ and $\htf_I^\ell$ gives \eqref{eq:initAmpBds}.
\end{proof}
\begin{proof}[Proof of \eqref{eq:initAmpErrBds}]  It suffices to prove \eqref{eq:initAmpErrBds} for $|\va| = 0$ and $0 \leq r \leq 1$, since the remaining bounds follow from \eqref{eq:ampBounds}, \eqref{eq:initAmpBds} and $\hat{\de} v_I = v_I - \hat{v}_I$ by comparison using $(B_\la \osn)^{1/2} \geq \ost{b}^{-1}$.  Using \eqref{eq:expressInitForm}, \eqref{eq:eDef} and \eqref{eq:importantnbxiprops} we have $\co{\hat{\de} v_I} \lsm \co{e^{1/2}(t)} \cdot \co{\nb \xi_I - \nb \hxii_I} \lsm e_\vp^{1/2} \ost{b}$ as desired.  We then obtain
\ALI{
\co{\Ddt \hat{\de} v_I } &\leq \sum_{r_1 + \cdots + r_4 = 1 }\co{\Ddt^{r_1} e_I^{1/2}(t)} \co{\Ddt^{r_2}\eta_I(t)} \cdot \co{\Ddt^{r_3}(\nb \xi_I - \nb \hxii_I)} \co{\Ddt^{r_4} \tilde{f}_I^\ell} \\
&\lsm \sum_{r_1 + \cdots + r_4 = 1 } [\ostu^{-r_1} e_\vp^{1/2}][\ostu^{-r_2}] [\ostu^{-r_3} \ost{b}] [\ostu^{-r_4} ] \lsm \ostu^{-1} \ost{b} e_\vp^{1/2}
}
using \eqref{eq:eDef},\eqref{eq:c0bdPhasediff}-\eqref{eq:Ddtphasediff} and \eqref{eq:conbcompbds}-\eqref{eq:Ddtcompbds}.
\end{proof}

\begin{proof}[Proof of \eqref{eq:lowerOrderAmpBds}]  The estimate \eqref{eq:lowerOrderAmpBds} follows from \eqref{eq:ampBounds} using the Microlocal Lemma of \cite[Lemma 4.1]{isettVicol}.  According to this Lemma, the result of the convolution in \eqref{eq:convolution} has the following form
\ali{
\PP_I[e^{i \la \xi_I} v_I^\ell] &= e^{i \la \xi_I}( \hat{K}_{Ib}^\ell(\la \nb \xi_I) v_I^b + \de v_I^\ell ) \label{eq:microlocalExpansion}
}
where $\hat{K}_{Ib}^\ell(\la \nb \xi_I)$ denotes the Fourier transform\footnote{As with all convolution kernels in this work, $K_I$ is viewed as a Schwartz function on $\R^3$.} of the Kernel $K_I$ (i.e. the symbol of $\PP_I$) evaluated at the frequency $\la \nb \xi_I$, and where $\de v_I^\ell$ is given by a certain explicit formula.  We have from \eqref{eq:HfreqProjectiondef} that the symbol of $\PP_I$ is an orthogonal projection multiplied by a cutoff $\hat{K}_{Ib}^\ell(m) = \hat{\psi}\left(\fr{m}{n_I \la}\right) (\de_b^\ell - |m|^{-2} m_b m^\ell )$.  From \eqref{eq:importantnbxiprops}, the cutoff term becomes $\hat{\psi}\left(\fr{\la \nb \xi_I}{n_I \la}\right) = 1$ when evaluated at $\la \nb \xi_I$.  Using that $v_I^\ell \in \langle \nb \xi_I \rangle^\perp$ pointwise, we have that $(\de_b^\ell - |\nb \xi_I|^{-2} \nb_b \xi_I \nb^\ell \xi_I) v_I^b = v_I^\ell$, so \eqref{eq:microlocalExpansion} simplifies to the form assumed in \eqref{eq:VIform}, namely $\PP_I[e^{i \la \xi_I} v_I^\ell] = e^{i \la \xi_I}(v_I^\ell + \de v_I^\ell)$.  The estimate \eqref{eq:lowerOrderAmpBds} for $\de v_I^\ell$ is essentially carried out in \cite[Lemma 7.5]{isettVicol}, but with the following minor notational differences.  Namely, the amplitude in \cite{isettVicol} is a scalar function denoted $\th_I$ rather than the vector amplitude $v_I^\ell$, it has size $e_R^{1/2}$ rather than $e_\vp^{1/2}$, and the convolution kernel is scalar-valued rather than matrix valued.  The bounds of \eqref{eq:ampBounds} for $v_I^\ell$ become identical to those used for $\th_I$ in the proof of \cite[Lemma 7.5]{isettVicol} when we replace $v_I^\ell$ by $e_R^{1/2} (v_I^\ell / e_\vp^{1/2})$ (see \cite[Lemma 7.4]{isettVicol}), while the convolution kernel $K_{Ib}^\ell$ obeys all the bounds used for the convolution kernel $\eta_{\approx \la}^I$ in \cite[Lemma 7.5]{isettVicol}, since they are both obtained by rescaling a fixed Schwartz function by the frequency $\la$ while preserving the $L^1$ norm.  We can thus obtain our desired bound \eqref{eq:lowerOrderAmpBds} for $\de v_I^\ell$ by following the proof of \cite[Lemma 7.5]{isettVicol} used to bound the term $\de \th_I$.
\end{proof}

We now have the following bounds for the high frequency corrections and the pressure increment.
\begin{prop} \label{prop:correctionBounds} The following bounds hold uniformly for $I \in \II$
\ali{
\co{\nb_{\va} V_I} + \ostu \co{ \nb_{\va} \Ddt V_I } + \ostu^2 \co{\nb_{\va} \Ddt^2 V_I } &\lsm_{\va} \la^{|\va|} e_\vp^{1/2} \label{eq:highFreqCorrect} \\
\co{\nb_{\va} \VR_I} + \ostu \co{ \nb_{\va} \Ddt \VR_I } + \ostu^2 \co{\nb_{\va} \Ddt^2 \VR_I } &\lsm_{\va} \la^{|\va|} e_\vp^{1/2} \label{eq:mainTermHFBd} \\
\co{\nb_{\va} \de V_I} + \ostu \co{ \nb_{\va} \Ddt \de V_I } + \ostu^2 \co{ \nb_{\va} \Ddt^2 \de V_I } &\lsm_{\va} \la^{|\va|} (B_\la \osn)^{-1}  e_\vp^{1/2} \label{eq:lowerTermHFBd} \\
\co{\nb_{\va} P} \lsm_{\va} \osn^{(|\va|- L)_+/L} \Xi^{|\va|} e_{\undvp}, \quad \co{\nb_{\va} \Ddt P} &\lsm_{\va} \osn^{(|\va| + 1 - L)_+/L} \Xi^{|\va|+1} e_v^{1/2} e_{\undvp}. \label{eq:pressIncBound}
}
The bound \eqref{eq:pressIncBound} holds also with $\Ddt$ replaced by $D_t$ for $0 \leq |\va| \leq L-1$.
\end{prop}
\begin{proof}[Proof of \eqref{eq:highFreqCorrect}-\eqref{eq:lowerTermHFBd}] We start with \eqref{eq:mainTermHFBd}.  Using $\VR_I^\ell = e^{i \la \xi_I} v_I^\ell$, we have that
\ali{
\nb_{\va}\Ddt^r[e^{i \la \xi_I} v_I^\ell ] &= \sum_{0 \leq m \leq |\va|} \sum_{\vcb, \vcc} \nb_{\vcb} \Ddt^r v_I^\ell e^{i \la \xi_I} (i \la)^m \prod_{i=1}^m \nb_{\vcc_i} (\nb_{c_i} \xi_I),
}
where the sum includes only multi-indices such that 
$|\vcb| + m + \sum_i |\vcc_i| = |\va|$.  For $0 \leq r \leq 2$, we obtain
\ali{
\co{\nb_{\va}\Ddt^r[e^{i \la \xi_I} v_I^\ell ]} &\lsm_M \sum_{0 \leq m \leq |\va|} \sum_{\vcb, \vcc} \left[\osn^{(|\vcb| + 1 - L)_+/L} \Xi^{|\vcb|} \ostu^{-r} e_\vp^{1/2} \right] \cdot \la^{m} \prod_{i=1}^m \osn^{(|\vcc_i| + 1 - L)_+/L} \Xi^{|\vcc_i|} \\
&\lsm_{\va} \sum_{0 \leq m \leq |\va|} \ostu^{-r} e_\vp^{1/2} \osn^{(|\va|-m - (L-1))_+/L} \Xi^{|\va| - m} \la^m \lsm_{\va} \ostu^{-r} e_\vp^{1/2} \la^{|\va|},  \label{eq:phaseBd}
}
where in the last line we used the elementary counting inequality $(|\vcb| - (L - 1))_+ + \sum_i (|\vcc_i| - (L - 1) )_+ \leq (|\va| - m - (L -1))_+$ in \eqref{eq:countingIneq} and the fact that $\osn^{1/L} \Xi \leq \la$.  The proof of \eqref{eq:lowerTermHFBd} is the same except that we replace $v_I^\ell$ by $\de v_I^\ell$, $\ostu^{-r} e_\vp^{1/2}$ by $(B_\la \osn)^{-1} \ostu^{-r} e_\vp^{1/2}$, and we replace $L - 1 \geq 0$ by $L - 2 \geq 0$ in the elementary counting inequality.  
Inequality \eqref{eq:highFreqCorrect} now follows by adding \eqref{eq:mainTermHFBd} and \eqref{eq:lowerTermHFBd}.
\end{proof}
\begin{proof}[Proof of \eqref{eq:pressIncBound}]  These bounds follow from the formula \eqref{eq:Pdefn}, since they hold for $e(t)$, $\kk_{[G\ep]}$, $R_{[2\ep]}$ and $R_{[G\ep]}$ by \eqref{eq:eDef} and Proposition~\ref{prop:prelimBds}.  To obtain the bound for $D_t$ in place of $\Ddt$, we will use the fact that for $0 \leq |\va_1| \leq L$ we have $\co{\nb_{\va_1}(v - v_\ep)} \lsm \Xi^{|\va_1|} e_v^{1/2}$, 
from which we can bound the remaining term in $\nb_{\va} D_t P = \nb_{\va} \Ddt P + \nb_{\va}\nb_i[(v^i - v_\ep^i)P]$ by
\ali{
\co{\nb_{\va} \nb_i [(v^i - v_\ep^i) P]} &\lsm \sum_{|\va_1| + |\va_2| = |\va| + 1} (\Xi^{|\va_1|} e_v^{1/2}) ( \Xi^{|\va_2|} e_{\undvp}) \lsm \Xi^{|\va| + 1} e_v^{1/2} e_{\undvp} \label{eq:PerrorMoll}
}
for all $0 \leq |\va| \leq L - 1$ as desired.
\end{proof}

With these bounds in hand, we can begin to estimate the terms of the new stress given in \eqref{eq:stressDecomp}.

\subsection{Preliminaries for estimating the new stress} \label{sec:newStressBounds}

In this Section we gather some preparatory estimates for bounding the new stress tensor $\ost{R}^{j\ell}$.

To prove Lemma~\ref{lem:mainLem}, we must estimate the new stress $\ost{R}^{j\ell}$ and also its advective derivative $\osdt \ost{R}^{j\ell}$ corresponding to the new velocity field $\osdt := \pr_t + \ost{v} \cdot \nb$,  $\ost{v}^\ell = v^\ell + V^\ell$.  Since the bounds we have encountered so far have been stated in terms of the velocity fields $v_\ep^\ell$ and $v^\ell$ and their corresponding advective derivatives $\Ddt$ and $D_t$ rather than $\ost{v}^\ell$ and $\osdt$, the following Lemma will be useful.
\begin{lem} \label{lem:comparisonLem}  For each advective derivative ${\mathring D}_t \in \{ D_t, \Ddt, \osdt\}$, define $\mathring{H} \in \{ H, \brh, \ost{H} \}$ by
\ali{
\mathring{H}[F] := \max_{0 \leq r \leq 1} \max_{0 \leq |\va| + r \leq L} \fr{\co{\nb_{\va} {\mathring D}_t^r F}}{\la^{|\va|}( \la e_\vp^{1/2})^{r}}. \label{eq:weightedNormDef}
}
Then each weighted norm $\mathring{H}[\cdot]$ satisfies the triangle inequality and product rule as in \eqref{eq:productRule}, these norms are comparable $\ost{H}[F] \lsm H[F] \lsm \brh[F] \lsm \ost{H}[F]$, and $\brh[F] \lsm \brh^{(L-1, 1)}[F]$.  The inequalities $\underline{\ost{H}}[F] \lsm \underline{H}[F] \lsm \underline{\brh}[F] \lsm \underline{\ost{H}}[F]$ hold also for the homogeneous weighted seminorms $\underline{\mathring{H}}[F]$, which we define similarly to \eqref{eq:weightedNormDef} but omitting the term $|\va| = r = 0$.  The homogeneous seminorms satisfy the triangle inequality $\undl{\mathring{H}}[F + G] \leq \undl{\mathring{H}}[F] + \undl{\mathring{H}}[G]$ and a product rule of the form 
\ali{
\undl{\mathring{H}}[F  G] \lsm \undl{\mathring{H}}[F] \mathring{H}[G] + \mathring{H}[F] \undl{\mathring{H}}[G]. \label{eq:homogProduct}
}
\end{lem}
\begin{proof} The proof of the triangle inequality and product rule \eqref{eq:homogProduct} is analogous to Lemma~\ref{lem:accountLem1}.  The inequality $\brh[F] \lsm \brh^{(L-1, 1)}[F]$ follows from the definitions using $\osn^{1/L} \Xi \leq \la$ and $\la e_\vp^{1/2} \geq \ostu^{-1}$.  The proof of $\ost{H}[F] \lsm H[F]$ is a simpler version of the proof of $\ost{H}[F] \lsm \brh[F]$, so we will focus on this latter inequality instead.  Let $\wtld{V}^\ell = \ost{v}^\ell - v_\ep^\ell = V^\ell + (v^\ell - v_\ep^\ell)$, and observe that, for $0 \leq |\va| \leq L$ we have
\ali{
\co{\nb_{\va}(v^\ell - v_\ep^\ell)} &\lsm_{\va} \la^{|\va|} (e_v^{1/2}/\osn), \label{eq:vminusvepnbabd} \\
\co{\nb_{\va} \wtld{V} } \leq \co{\nb_{\va} V^\ell} + \co{\nb_{\va}(v^\ell - v_\ep^\ell)} &\lsm \la^{|\va|} (e_\vp^{1/2} + e_v^{1/2}/\osn ) \lsm \la^{|\va|} e_\vp^{1/2}. \label{eq:wtldVnbbd}
}
In \eqref{eq:vminusvepnbabd} we use \eqref{eq:vpBd},\eqref{eq:vepBasicBd}-\eqref{eq:vminvepBd} and $\la \geq \osn \Xi$, while in \eqref{eq:wtldVnbbd} we use \eqref{eq:highFreqCorrect} and $\osn \geq (e_v^{1/2}/e_\vp^{1/2})$.

For any $0 \leq |\va| \leq L -1$, we then have that
\ali{
\begin{split} \label{eq:changingAdvecComp}
\nb_{\va} \osdt[F] &= \nb_{\va}[ \Ddt F + \wtld{V}^\ell \nb_\ell F ] \\
\co{\nb_{\va} \osdt[F]} &\lsm \co{ \nb_{\va} \Ddt F } + \sum_{|\va_1| + |\va_2| = |\va|} \co{\nb_{\va_1} \wtld{V}^\ell} \co{\nb_{\va_2} \nb_\ell F } \\
\co{\nb_{\va} \osdt[F]} &\lsm \la^{|\va| + 1} e_\vp^{1/2} \brh[F] + \sum_{|\va_1| + |\va_2| = |\va|} [\la^{|\va_1|} e_\vp^{1/2}][\la^{|\va_2| + 1}\brh[F] ] \lsm \la^{|\va|} (\la e_\vp^{1/2}) \brh[F],
\end{split}
}
which implies the bound $\ost{H}[F] \lsm \brh[F]$ we desire.  Interchanging the roles of $\ost{H}$ and $\brh[\cdot]$ and replacing $\wtld{V}^\ell$ by $- \wtld{V}^\ell$ gives the opposite inequality $\ost{H}[F] \lsm \brh[F]$.  Note that the same argument applies to the homogeneous versions $\underline{\ost{H}}[\cdot]$ and $\underline{\brh}[\cdot]$.
\end{proof}
In terms of the above weighted norms, we have the following Proposition, which follows from Propositions~\ref{prop:prelimBds}, \ref{prop:transportEstimates}, \ref{prop:correctAmpBds} and \ref{prop:correctionBounds} (noting that $\ost{b}^{-1} \ostu^{-1} \leq \la e_\vp^{1/2}$ in the case of $\hat{\de} v_I$).
\begin{prop} \label{prop:allTheEndBounds} Let $(F, h_F)^t$ be any column of the following formal matrix of variables
\ALI{
\mat{c|ccccccccccccc}{ {F} & \nb \xi_I & v_I & \hat{v}_I & P & R_{[2\ast]} &  R_{[3\ast]} & \vp_{(2)}
 & V_I & \VR_I & \hat{\de} \nb \xi_I  & \hat{\de} v_I & \de v_I  & \de V_I \\
		{h_F} &	1 & e_\vp^{1/2} & e_\vp^{1/2} & {e_{\undl{\vp}}} & e_R & e_G & e_{\undl{\vp}}^{3/2} 
	& e_\vp^{1/2} & e_\vp^{1/2} & \ost{b} & \ost{b} e_\vp^{1/2} & \la^{-1} \Xi e_\vp^{1/2} & \la^{-1} \Xi e_\vp^{1/2}	
		} 
}
Then $\brh[F] \lsm h_F$ and $\brh[\Ddt F] \lsm \ostu^{-1} h_F$.
\end{prop}
\begin{proof} From the inequality $\ostu^{-1} \leq \la e_\vp^{1/2}$, which follows from the choices of $\ost{b} = b_0 (e_v^{1/2}/e_\vp^{1/2}N)^{1/2}$, $\ostu^{-1} = \ost{b}^{-1} (\Xi e_v^{1/2})^{-1}$ and $\la = B_\la N \Xi$, as well as $N^{(|\va| + 1 -L)_+/L} \Xi^{|\va|} \leq \la^{|\va|}$ we have that $\brh[F] + \ostu \brh[\Ddt F] \lsm \brh^{(L,2)}[F]$.  Applying Propositions~\ref{prop:prelimSummaries},\ref{prop:correctAmpBds},\ref{prop:correctionBounds}, we obtain the bounds on each field except for $P$, $\hat{\de} \nb \xi_I$ and $\hat{\de} v_I$.  The bound for $P$ follows from the bound $\brh^{(M,2)}[P] \lsm_M e_{\undvp}$, which follows as in the proof of \eqref{eq:pressIncBound}.  The bounds for $\hat{\de} \nb \xi_I$ and $\hat{\de} v_I$ folllow from \eqref{eq:Ddtphasediff} and \eqref{eq:initAmpErrBds} after noting that $\la \geq (B_\la N)^{1/2}$ and $\ostu^{-1} \leq \ost{b}^{-1} \ostu^{-1} = b_0^{-2}B_\la (e_\vp^{1/2} N/e_v^{1/2}) (\Xi e_v^{1/2}) \lsm \la e_\vp^{1/2}$.
\end{proof}
The following Lemma will be useful when we encounter the terms of the form $F - F_\ep$ or $F - F_{\ep_x}$ that arise as the errors in regularizing the various tensor fields in the construction.

\begin{prop} \label{prop:mollErrProp} Let $(F,h_F)^t$ be any column of the following formal matrix 
\ALI{
\mat{c|ccccccccc}{ F & v & p & P & R_{[1\circ]} & R_{[2]} & R_{[G]} & \vp & \kk & \kk_{[G]} \\
						h_F & e_v^{1/2} & e_v & e_{\undvp} & e_\vp&  e_{R} & e_G & e_\vp^{3/2} & e_\vp & e_R
		} 
}
then we have the bounds
\ali{
\brh[F - F_{\ep_v}] &\lsm (h_F/\osn) \label{eq:spaceMollErrorBds}\\
\brh[F - F_\ep] &\lsm (e_v^{1/2}/e_\vp^{1/2})(h_F/\osn), \label{eq:mollErrorBds}
}
where $F_{\ep_v} = \chi_\ep \ast \chi_\ep \ast F$ is the spatial mollification of $F$ defined using the kernel of lines \eqref{eq:vepBasicBd}-\eqref{eq:vminvepBd}.
\end{prop}
For the functions $(v, p, P)$ we will only use \eqref{eq:spaceMollErrorBds} although \eqref{eq:mollErrorBds} is also true in these cases.
\begin{proof}  For each of the fields $F$ above, we have bounds of the form,
\ali{
\co{F - F_{\ep_x}} &\leq (h_F/N) \label{eq:coMollerRecall} \\
\co{F - F_\ep} &\leq (e_v^{1/2}/e_\vp^{1/2})(h_F/N)
}
from\footnote{The bound $\co{P - P_{\ep_v}} \lsm (e_{\undvp}/\osn)$ follows from Proposition~\ref{prop:prelimBds} using \eqref{eq:pressIncBound}.} 
\eqref{eq:vMinVep},\eqref{eq:errorMollBounds},\eqref{eq:errorMollPress}, as well as the bounds (using that $\la \geq \osn \Xi$) 
\ali{
\co{\nb_{\va} F} + \co{\nb_{\va} F_{\ep_v}} + \co{\nb_{\va} F_\ep} &\lsm \Xi^{|\va|} h_F \lsm \la^{|\va|} (h_F/\osn) \mbox{ if } |\va| \geq 1, 
\label{eq:spaceDerivsofFs} \\
\co{\nb_{\va} D_t F} + \co{\nb_{\va} \Ddt F_{\ep_v} } + \co{\nb_{\va} \Ddt F_\ep} &\lsm \Xi^{|\va| + 1} e_v^{1/2} h_F \leq \la^{|\va| + 1} e_\vp^{1/2} \left(\fr{e_v^{1/2} h_F}{e_\vp^{1/2} \osn}\right),
\label{eq:DtderivsOfFs}
}
where \eqref{eq:spaceDerivsofFs} holds for $1 \leq |\va| \leq L$ while \eqref{eq:DtderivsOfFs} holds for $0 \leq |\va| \leq L-1$.  These bounds are obtained in \eqref{eq:vpBd}-\eqref{eq:lastFrqenBd}, \eqref{eq:velocBd1}-\eqref{eq:velocBd2}, \eqref{eq:coBdFieldsmoll}-\eqref{eq:c0DddtbdFmoll}, \eqref{eq:mollxPressbd}-\eqref{eq:mollxPressDdt} and \eqref{eq:pressIncBound}.  The final inequality in \eqref{eq:DtderivsOfFs} holds also for the coarse scale advective derivative $\Ddt F$ in place of $D_t F$ as a consequence of the comparison inequality $\underline{\brh}[F] \lsm \underline{H}[F]$ of Lemma~\ref{lem:comparisonLem}.  
established in\footnote{Here we use again \eqref{eq:pressIncBound} for $P$, which implies the bounds required in Proposition~\ref{prop:prelimBds}.} 
\eqref{eq:vMinusVepbd},\eqref{eq:errorMollBounds} and \eqref{eq:errorMollPress} gives \eqref{eq:mollErrorBds}.  

We now consider the sharper bound of \eqref{eq:spaceMollErrorBds}.  The bounds on spatial derivatives and the $C^0$ norm follow from \eqref{eq:coMollerRecall},\eqref{eq:spaceDerivsofFs}.  The bounds on the advective derivative will require a more delicate argument.  
For $1 \leq |\va| \leq L-1$, we can use the following estimate, which applies \eqref{eq:DtderivsOfFs} and \eqref{eq:vminusvepnbabd}:
\ali{
\Ddt( F - F_{\ep_v}) &= (v_\ep^j - v^j) \nb_j F + D_t F - \Ddt F_{\ep_x} \label{eq:DdtMollerExpress} \\
\co{ \nb_{\va}\Ddt (F - F_{\ep_v})} &\lsm \sum_{|\va_1| + |\va_2| = |\va|} \co{\nb_{\va_1} (v^j - v_\ep^j)} \co{\nb_{\va_2} \nb_j F } + \Xi^{|\va| + 1} e_v^{1/2} h_F \notag \\
&\lsm  \sum_{|\va_1| + |\va_2| = |\va|} [\la^{|\va_1|} (e_v^{1/2} / \osn ) ][ \la^{|\va_2|} \Xi h_F] +  \la^{|\va| - 1} \Xi^2 e_v^{1/2} h_F \notag \\
\co{ \nb_{\va}\Ddt (F - F_{\ep_v})} &\lsm \la^{|\va| +1} e_\vp^{1/2} (h_F / \osn), \qquad \mbox{ if } 1 \leq |\va| \leq L-1. \label{eq:desiredmollErrDtbd}
}
In the last two lines we used that $|\va| \geq 1$, $\la \geq \osn \Xi$ and $\osn \geq e_v^{1/2} / e_\vp^{1/2}$.  It now suffices to prove \eqref{eq:desiredmollErrDtbd} in the case $|\va| = 0$ to obtain \eqref{eq:spaceMollErrorBds}.  Since the first term in \eqref{eq:DdtMollerExpress} obeys $\co{(v_\ep^j - v^j) \nb_j F} \leq (e_v^{1/2} / \osn) \Xi h_F \leq \la e_\vp^{1/2} (h_F/\osn)$, it is enough to bound the term $D_t F - \Ddt F_{\ep_x}$.  

We will use that the mollifier used to define $v_\ep$ in \eqref{eq:vepBasicBd}-\eqref{eq:vminvepBd} and $F_{\ep_v} := \chi_\ep \ast \chi_\ep \ast F := \chi_{\ep+\ep} \ast F$ 
satisfies the vanishing moment conditions $\int_{\R^3} h^{\va} \chi_\ep(h) dh = 0$ for all multi-indices $1 \leq |\va| \leq L$.  To reduce the appearances of minus signs in the following, we will also be using that $\chi_\ep(\cdot)$ is even.  Writing
\ali{
D_t F - \Ddt F_{\ep_v} &= (D_t F - \chi_{\ep+\ep} \ast D_t F) + \chi_{\ep+\ep} \ast D_t F - \Ddt F_{\ep_v}, \notag
}
the vanishing moment condition allows us to apply the inequality $\co{G - \chi_\ep \ast G} \lsm_k \ep^k \co{\nb^k G}$ with $k = L-1$ and $\ep = \ep_v = c_0 \osn^{-1/L} \Xi^{-1}$ to estimate the first term by
\ali{
\co{D_t F - \chi_{\ep + \ep} \ast D_t F } &\leq \co{D_t F - \chi_\ep \ast D_t F} + \co{ \chi_\ep \ast ( D_t F - \chi_\ep \ast D_t F) } \\
&\lsm (\osn^{-1/L} \Xi^{-1} )^{L-1} \co{\nb^{L-1} D_t F} \lsm \osn^{-1 + 1/L} \Xi e_v^{1/2} h_F.
}
Using $\la \geq \osn \Xi$, $L \geq 2$ and $\osn \geq e_v/e_\vp$, the last term is bounded by $\la e_\vp^{1/2}(h_F/\osn)$ as desired.  

For the commutator term $Q_\ep[v,F] := \Ddt F_{\ep_x} - \chi_{\ep+\ep} \ast D_t F$, we follow \cite[Section 16]{isett} to express
\ali{
Q_\ep[v,F] &:= v_\ep^a \nb_a(\chi_\ep \ast \chi_\ep F) - \chi_\ep \ast \chi_\ep \ast(v^a \nb_a F) \label{eq:commutermDef} \\
\begin{split}
&= \chi_\ep \ast(v_\ep^a \nb_a(\chi_\ep \ast F)) + [v_\ep^a \nb_a, \chi_\ep \ast](\chi_\ep \ast F) \\
&- \chi_\ep \ast(v^a \nb_a(\chi_\ep \ast F) ) + \chi_\ep \ast([v^a \nb_a, \chi_\ep \ast] F) 
\end{split} \label{eq:allCommuterms}
}
We introduce here the notation $F_{\ep1} := \chi_\ep \ast F$.  Each of these terms when properly viewed has a form amenable to our desired estimate, the simplest term being
\ali{
\co{\chi_\ep \ast(v_\ep^a - v^a)\nb_a F_{\ep_1} )} &\lsm \co{v - v_\ep} \Xi h_F \lsm (e_v^{1/2}/\osn) \Xi h_F \leq \la e_\vp^{1/2}(h_F/\osn). \label{eq:simplecommuterm}
}
We consider the last commutator term in \eqref{eq:allCommuterms} since the other remaining term is similar.  We write
\ali{
-[v^a \nb_a, \chi_\ep \ast] F(x) &= \int_{\R^3} (v^a(x+h) - v^a(x))[\nb_aF](x +h) \chi_\ep(h) dh. \label{eq:commutExpress}
}
We will apply Taylor expansion to order $L - 1$ (noting $L \geq 2$) with
\ali{
\begin{split}
G(1) &= \sum_{i=0}^{L-2} \fr{G^{(i)}(0)}{i!} + \fr{1}{(L-2)!} \int_0^1 (1-\si)^{L-2} G^{(L-1)}(\si) \tx{d}\si \label{eq:GtaylorExpand} \\
G(\si) &:= (v^a(x+\si h) - v^a(x))[\nb_aF](x + \si h) = \si \int_0^1\nb_b v^a(x + \tau \si h)[\nb_aF](x + \si h) \tx{d}\tau h^b  .
\end{split}
}
The vanishing moment condition for $\chi_\ep(h)$ implies that, when we substitute $G(1)$ in \eqref{eq:GtaylorExpand} for the integrand in \eqref{eq:commutExpress}, all the terms involving $\fr{G^{(i)}(0)}{i!}$ vanish because they are proportional to $\int_{\R^3} h^{\va} \chi_\ep(h) \tx{d}h$ for some $1 \leq |\va| \leq L$, and \eqref{eq:commutExpress} becomes an expression of the form
\ALI{
\sum_{\va_1, \va_2}  \fr{c_{\va_1, \va_2}}{(L {-} 2)!}\iiint_{\R^3} \nb_{\va_1}\nb_bv^a(x + \tau \si h) [\nb_{\va_2} \nb_a F](x + \si h) h^b h^{\va_1} h^{\va_2} \chi_\ep(h) \tx{d} h \, \si(1{-}\si)^{L-2} \tau^{|\va_1|} \tx{d}\tau \tx{d}\si,
}
where the sum runs over multi-indices with $|\va_1| + |\va_2| = L-1$ and the bounds of integration for $\tx{d}\tau$ and $\tx{d}\si$ are from $0$ to $1$.  This long expression gives us the following bound
\ali{
\co{\eqref{eq:commutExpress}} &\lsm \sum_{|\va_1| + |\va_2| = L-1} [\Xi^{|\va_1| + 1} e_v^{1/2}]\cdot[\Xi^{|\va_2| + 1} h_F] \cdot \| \, |h|^L \chi_\ep \|_{L^1(\R^3)} \notag \\
&\lsm \ep_v^L \Xi^{L+1} e_v^{1/2} h_F \lsm \Xi e_v^{1/2} (h_F/\osn) \leq \la e_\vp^{1/2} (h_F/\osn). \label{eq:commuterm2}
}
This bound implies that $\co{\chi_\ep \ast ([ v^a \nb_a, \chi_\ep \ast] F)} = \co{\chi_\ep \ast \eqref{eq:commutExpress}} \lsm \la e_\vp^{1/2} (h_F/\osn)$ as desired, and the bound for the remaining term $[v_\ep^a \nb_a, \chi_\ep \ast](F_{\ep 1})$ in \eqref{eq:allCommuterms} follows by the same argument.
\end{proof}

Our next Lemma will be useful when we apply order $-1$ operators of the form $Q \ast$ that are restricted to frequencies of the order $\la$.  Similar estimates appear in \cite[Lemma 7.2]{isettVicol} and \cite{buckShkVicSQG}.  
\begin{lem} \label{lem:commutator} Suppose $Q \colon \R^3 \to \R$ is Schwartz and satisfies the following bounds
\ali{
\max_{0 \leq |\va| \leq |\vcb| \leq L} \la^{-(|\vcb| - |\va|)} \| h^{\va} \nb_{\vcb} Q] \|_{L^1(\R^3)} &\leq \La^{-1}
}
for some real number $\La^{-1} \geq 0$.  Then for any smooth $U$ on $\T^3$ we have
\ali{ 
\brh[Q \ast U] &\lsm \La^{-1} \left( \co{U} + \la^{-1} e_\vp^{-1/2} \co{\Ddt U} \right) \label{eq:solvedDivBd}
}
As a consequence, we have $\brh[Q \ast U] \lsm \La^{-1} \brh[U]$ and, if $U = e^{i \la \xi} u$ with $\Ddt \xi = 0$, we have that
\ali{
\brh[Q \ast (e^{i \la \xi} u)] &\lsm \La^{-1}(\co{u} +  \la^{-1} e_\vp^{-1/2} \co{\Ddt u}) \lsm \La^{-1} \brh[u].
}
\end{lem}
\begin{proof}  The estimate on spatial derivatives follows from the standard bound
\ALI{
\co{\nb_{\va}[ Q \ast U ]} = \co{ (\nb_{\va} Q)\ast U } &\leq \co{U} \| \nb_{\va} Q \|_{L^1(\R^3)} \leq \La^{-1} \la^{|\va|} \co{U}, \qquad 0 \leq |\va| \leq L.
}
For the advective derivative calculation, we use ${}^{(\bar{x})}\nb$ to denote differentiation in the $\bar{x}$ variable and calculate $\Ddt[Q \ast U]$ at the point $x \in \T^3$ as follows
\ali{
\Ddt[Q \ast U] &= (\pr_t + v_\ep^j(x) \nb_j) \int_{\R^3} U(x-h) Q(h) \tx{d}h =  \int_{\R^3} \Ddt U(x-h) Q(h) \tx{d}h - B \notag \\
B &= \int_{\R^3}(v_\ep^j(x-h) - v_\ep^j(x)) \nb_j U(x-h) Q(h) dh. \label{eq:commuterm}
}
The first term is equal to $Q \ast \Ddt U(x)$, which we bound as before by
\ALI{
\co{\nb_{\va} Q \ast \Ddt U } &\leq \La^{-1} \la^{|\va|} \co{\Ddt U}, \qquad 0 \leq |\va| \leq L.
}
For the commutator term \eqref{eq:commuterm} we first use ${}^{(x)}\nb_j[U(x-h)] = -{}^{(h)} \nb_j[U(x-h)]$ to integrate by parts and, using $\nb_j v_\ep^j = 0$ and the Fundamental Theorem, we write
\ALI{
\eqref{eq:commuterm} &= \int_{\R^3} ( v_\ep^j(x-h) - v_\ep^j(x)) U(x-h)  \nb_j Q(h) \tx{d} h \\
&= -\int_0^1 \int_{\R^3} U(x-h) \nb_b v_\ep^j(x-\si h) h^b \nb_j Q(h) \tx{d} h \tx{d} \si.
}
Now differentiating $0 \leq |\va| \leq L - 1$ times in $x$ and performing a similar integration by parts when the derivative hits $U(x-h)$, we obtain our desired estimate
\ALI{
\nb_{\va} \eqref{eq:commuterm} &= \sum_{|\va_1| + |\va_2| = |\va|} c_{\va_1,\va_2, \va} \int_0^1 \int_{\R^3} U(x-h) {}^{(h)}\nb_{\va_2}[ \nb_{\va_1} \nb_b v_\ep^j(x - \si h) h^b \nb_j Q(h) ] \tx{d} h \tx{d} \si \\
\co{\nb_{\va} \eqref{eq:commuterm}} &\lsm \sum_{|\vcb_1| + |\vcb_2| = |\va|} \co{U} \co{\nb_{\vcb_1} \nb_b v_\ep^j} \| \nb_{\vcb_2}[h^b \nb_j Q] \|_{L^1(\R^3)} \\
\co{\nb_{\va} \eqref{eq:commuterm}} &\lsm \sum_{|\vcb_1| + |\vcb_2| = |\va|}  \co{U} [\Xi^{|\vcb| + 1} e_v^{1/2}] [\la^{|\vcb_2|} \La^{-1}] \lsm \La^{-1} \co{U} \la^{|\va|+1} e_\vp^{1/2}.
}
In the last line we used $\la \geq \osn \Xi$ and $\osn \geq (e_v^{1/2}/e_\vp^{1/2})$.
\end{proof}

With the above preliminaries in hand, we can now begin estimating the new stress.

\subsection{Estimating the new stress} \label{sec:boundNewStress}

\, 
\, Recall from \eqref{eq:stressDecomp} that the new stress $\ost{R}^{j\ell} = R_T^{j\ell} + R_S^{j\ell} + R_H^{j\ell} + R_M^{j\ell}$ will be decomposed into four terms.  The terms $R_T^{j\ell}$ and $R_H^{j\ell}$ are symmetric tensors that must solve the equations
\ali{
\nb_j R_T^{j\ell} &= \pr_t V^\ell + \nb_j[v_\ep^j V^\ell + V^j v_\ep^\ell] \notag \\
\nb_j R_T^{j\ell} &= \pr_t V^\ell + v_\ep^j  \nb_j V^\ell + V^j  \nb_j v_\ep^\ell \label{eq:transportTermThird} \\
\nb_j R_H^{j\ell} &= \sum_{J \neq \bar{I}} \nb_j( V_I^j V_J^\ell). \label{eq:highFreqTermThird}
}
Observe that right hand sides of \eqref{eq:transportTermThird} and \eqref{eq:highFreqTermThird} both have frequency support contained in the annulus 
\ali{
\Fsupp \eqref{eq:transportTermThird} \cup \Fsupp \eqref{eq:highFreqTermThird} \subseteq \{ B_\la \osn \Xi \leq |p|  \lsm B_\la \osn \Xi \} \subseteq \widehat{\R^3}. \label{eq:containFreq}
}
This containment uses that the Fourier transform maps products to convolutions, the containment $\Fsupp V_I \subseteq \{ 2 B_\la \osn \Xi \leq (n_I/3) B_\la \osn \Xi \leq |p| \leq 3 n_I B_\la \osn \Xi \}$ for $n_I$ chosen appropriately in Section~\ref{sec:constrShape} (noting \eqref{eq:noWrongCascades}), and the containment $\Fsupp v_\ep \subseteq \{ |p| \leq (1/10) \osn^{1/2} \Xi \}$ for $c_0$ chosen small in \eqref{eq:vMinusVepbd}.

Now let $Q_a^{j\ell} \ast$ be a right-inverse for the divergence operator with order $-1$ taking values in symmetric tensors.  Let $Q_{a, \approx \la}^{j\ell}  \ast = P_{\approx \la} Q_a^{j\ell} \ast$ be the same operator but localized to frequencies of order $\la$ by a Fourier multiplier $P_{\approx \la} = \chi_{\approx \la} \ast$ whose symbol is a smooth, rescaled cutoff function $\widehat{\chi_{\approx \la}}(p) = \widehat{\chi_{\approx 1}}(p/\la)$ supported in $\Fsupp \chi \subseteq \{ \la/2 \leq |p| \lsm \la \}$ with $\widehat{\chi_{\approx \la}}(p) = 1$ on \eqref{eq:containFreq}.  We solve \eqref{eq:transportTermThird}-\eqref{eq:highFreqTermThird} by setting
\ali{
R_T^{j\ell} &= Q_{a, \approx \la}^{j\ell} \ast [ \Ddt V^a + V^b \nb_b v_\ep^a  ], \qquad 
R_H^{j\ell} = \sum_{J \neq \bar{I}} Q_{a, \approx \la}^{j\ell} \ast [ \nb_b( V_I^b V_J^a)   ]. \label{eq:defOfRTRH}
}
One example of such an order $-1$ inverse to the divergence is the Fourier-multiplier whose symbol is
\ali{
\begin{split}
\widehat{Q}_{a}^{j\ell}(p) &= (-i) \left( \widehat{Q}_{a \parallel}^{j\ell}(p) + \widehat{Q}_{a\perp}^{j\ell}(p) \right) \\
\widehat{Q}_{a \parallel}^{j\ell}(p) = |p|^{-2} p_a \de^{j\ell}, \quad &\widehat{Q}_{a\perp}^{j\ell}(p) = |p|^{-2}(p^j \de_a^\ell + \de_a^j p^\ell) - 2 |p|^{-4}(p^j p^\ell p_a ).
\end{split}
}
The properties we need are that $\widehat{Q}_{a}^{j\ell}(p)$ is symmetric in $j\ell$, degree $-1$ homogeneous, smooth away from $0$, and satisfies $(i p_j) \widehat{Q}_a^{j\ell}(p) = \de_a^\ell$.  With these properties, we have that $Q_{a,\approx \la}^{j\ell}(h)$ 
is a rescaling of a Schwartz function with $\widehat{Q}_{a,\approx \la}^{j\ell}(p) = \la^{-1} \widehat{\chi_{\approx 1}}(p/\la) \widehat{Q}_a^{j\ell}(p/\la)$, and that $R_T^{j\ell}$ and $R_H^{j\ell}$ defined in \eqref{eq:defOfRTRH} solve \eqref{eq:transportTermThird}-\eqref{eq:highFreqTermThird} and satisfy the containment \eqref{eq:containFreq}.  In physical space, the kernel takes the form $Q_{a,\approx \la}^{j\ell}(h) = \la^{-1} ( \la^3 Q_{a, \approx 1}^{j\ell}(\la h) ) $ with $Q_{a, \approx 1}^{j\ell}$ a fixed Schwartz function.  Consequently, $Q_{a,\approx \la}^{j\ell}(h)$ satisfies the hypotheses of Lemma~\ref{lem:commutator} with $\La^{-1} \lsm \la^{-1}$.  Applying Lemma~\ref{lem:commutator} to \eqref{eq:defOfRTRH} and using the estimates
\ali{
\co{ \Ddt V^a + V^b \nb_b v_\ep^a } &\lsm \ostu^{-1} e_\vp^{1/2}  + e_\vp^{1/2} \Xi e_v^{1/2} \label{eq:transportTermBoundsnow} \\ 
\co{\Ddt ( \Ddt V^a + V^b \nb_b v_\ep^a )} &\lsm \ostu^{-1} (\ostu^{-1} e_\vp^{1/2}  + e_\vp^{1/2} \Xi e_v^{1/2}  ), \label{eq:transportTermBoundsnow2}
}
which follow from \eqref{eq:velocBd3} and \eqref{eq:highFreqCorrect}, we obtain the estimate $\brh[R_T] \lsm \la^{-1} \ostu^{-1} e_\vp^{1/2}$.

Turning now to $R_H^{j\ell}$, we calculate
\ali{
 \nb_j( V_I^j V_J^\ell) = V_I^j \nb_j V_J^\ell &= \VR_I^j \nb_j \VR_J^\ell + \de V_I^j \nb_j V_J^\ell + \VR_I^j \nb_j \de V_J^\ell \label{eq:highHighExpress1} \\
\VR_I^j \nb_j \VR_J^\ell &= (i \la) e^{i \la (\xi_I + \xi_J)} v_I^j \nb_j \xi_J v_I^\ell = (i \la) e^{i \la (\xi_I + \xi_J)} ( \hat{\de} v_I^j \nb_j \hxii_J + v_I^j \hat{\de} \nb_j \xi_J ) v_I^\ell, \label{eq:highHighExpress2}
}
where we have used $v_I^j = \hat{v}_I^j + \hat{\de} v_I^j$, $\nb \xi_J = \nb \hxii_J + \hat{\de} \nb \xi_J$ and $\hat{v}_I^j \nb_j \hxii_J = 0$.

Using \eqref{eq:highHighExpress1}-\eqref{eq:highHighExpress2}, Propositions~\ref{prop:transportEstimates},~\ref{prop:correctAmpBds}, and~\ref{prop:correctionBounds}, and that $\Xi \leq \ost{b} \la $ 
 for $B_\la$ large and $\osn \geq (e_v/e_\vp)$, we obtain the estimates
\ali{
\co{\VR_I^j \nb_j \VR_J^\ell} &\lsm \la (\ost{b} e_\vp^{1/2} \cdot 1 + e_\vp^{1/2} \cdot \ost{b} ) e_\vp^{1/2} \lsm \la \ost{b} e_\vp \\
\co{V_I^j \nb_j V_J^\ell } &\lsm \la \ost{b} e_\vp + (\la^{-1} \Xi e_\vp^{1/2})(\la e_\vp^{1/2}) + e_\vp^{1/2} \cdot [\Xi e_\vp^{1/2}] \lsm \la \ost{b} e_\vp. \label{eq:c0HHbd}
}
We may also apply $\Ddt$ to \eqref{eq:highHighExpress2} and use the product rule (noting $\Ddt e^{i \la (\xi_I + \xi_J)} = 0$) to obtain that
\ali{
\co{\Ddt[\VR_I^j \nb_j \VR_J^\ell]} &\lsm \la \ostu^{-1} (\ost{b} e_\vp^{1/2} \cdot 1 + e_\vp^{1/2} \cdot \ost{b} ) e_\vp^{1/2} \lsm \la \ostu^{-1} \ost{b} e_\vp \\
\co{\Ddt[V_I^j \nb_j V_J^\ell] } &\lsm \la \ostu^{-1} \ost{b} e_\vp + \ostu^{-1} \Xi e_\vp \lsm \la \ostu^{-1} \ost{b} e_\vp.  \label{eq:c0DdtHHbd}
}
Here we have again used the bounds of Propositions~\ref{prop:transportEstimates},~\ref{prop:correctAmpBds}, and~\ref{prop:correctionBounds} in combination.

Applying Lemma~\ref{lem:commutator} for \eqref{eq:defOfRTRH} with $\La^{-1} \lsm \la^{-1}$ and using \eqref{eq:transportTermBoundsnow}-\eqref{eq:transportTermBoundsnow2} and \eqref{eq:c0HHbd}, \eqref{eq:c0DdtHHbd}, we obtain 
\ali{
\begin{split}
\brh[R_T] + \brh[R_H] &\lsm \la^{-1} \ostu^{-1} e_\vp^{1/2} + \ost{b} e_\vp \lsm \ost{b} e_\vp \\
\brh[R_T] + \brh[R_H] &\lsm B_\la^{-1/2} (e_v^{1/2}/e_\vp^{1/2} \osn)^{1/2}e_\vp. \label{eq:boundedHighTrans}
\end{split}
}
In particular, for $B_\la$ sufficiently large we can ensure that
\ali{
\co{R_T} + \co{R_H} &\leq 10^{-3} (e_v^{1/2}/e_\vp \osn)^{1/2}e_\vp. \label{eq:RTHbdd}
}
Using that $V_I^\ell = e^{i\la \xi_I} v_I^\ell + \de V_I^\ell$ and \eqref{eq:vMinusVepbd},\eqref{eq:errorMollBounds}, an even better $C^0$ estimate holds for the term
\ali{
R_M^{j\ell} &= \sum_I [(v^j - v_\ep^j) V_I^\ell + V_I^j (v^\ell - v_\ep^\ell) ] + (R^{j\ell} - R_\ep^{j\ell}) \notag \\
\co{R_M^{j\ell}} &\leq 10^{-3} (e_v^{1/2}/\osn) e_\vp^{1/2} + \co{v - v_\ep} \co{\sum_I \de V_I} + 200^{-1} (e_\vp/\osn) \notag \\
\co{R_M^{j\ell}}&\leq 10^{-2}(e_v^{1/2}/\osn) e_\vp^{1/2} + A (e_v^{1/2}/\osn) B_\la^{-1} (e_\vp^{1/2} / \osn) \leq 10^{-1} (e_v^{1/2}/\osn) e_\vp^{1/2} \label{eq:c0RMbd}
}
for $B_\la$ sufficiently large.  Also, using Proposition~\ref{prop:mollErrProp} and the product rule of Lemma~\ref{lem:comparisonLem}, we have
\ali{
\brh[R_M] &\lsm \brh[v - v_\ep] \brh[V] + \brh[R - R_\ep] \lsm (e_v^{1/2}/\osn) e_\vp^{1/2} + (e_v^{1/2}/e_\vp^{1/2})(e_\vp/\osn) \lsm (e_v^{1/2}/ \osn) e_\vp^{1/2}. \label{eq:RMbd}
}
The last term we must estimate is $R_S^{j\ell}$.  We write $R_S^{j\ell} = P \de_{[2\ast]}^{j\ell} + R_{[2\ast]}^{j\ell} + R_{[3\ast]}^{j\ell} + \wtld{R}_S^{j\ell}$, with $\wtld{R}_S^{j\ell}$ the lower order term calculated in line \eqref{eq:stressTerm}.  By Proposition~\ref{prop:allTheEndBounds} and Lemma~\ref{lem:comparisonLem}, we then obtain
\ali{
\brh[P] + \brh[R_{[2\ast]}] &\lsm e_{\undvp}, \qquad \brh[R_{[3\ast]}] \lsm e_G  \label{eq:PR2bds} \\
\brh[\wtld{R}_S] &\lsm \max_I \left[\brh[\de V_I] (\brh[V_I] + \brh[\VR_I]) + \brh[\hat{\de} v_I](\brh[v_I] + \brh[\hat{\de} v_I]) \right] \notag \\
\brh[\wtld{R}_S] &\lsm [\la^{-1} \Xi e_\vp^{1/2}][e_\vp^{1/2}] + [\ost{b} e_\vp^{1/2} ][e_\vp^{1/2}] \lsm B_\la^{-1/2} (e_v^{1/2}/e_\vp^{1/2} \osn)^{1/2} e_\vp. \label{eq:lastLowerOrdBd}
}
The last inequality uses that $\la = B_\la \osn \Xi$, $\ost{b} = b_0 B_\la^{-1/2} (e_v^{1/2}/e_\vp^{1/2} \osn)^{1/2}$, and \eqref{eq:Nastdef}. 


We have now proved our desired estimates for $P, R_{[2\ast]}, R_{[3\ast]}$ and $R_{[G\ast]}^{j\ell} := \wtld{R}_S^{j\ell} + R_M^{j\ell} + R_T^{j\ell} + R_H^{j\ell}$ that will be sufficient for the bounds required in Lemma~\ref{lem:mainLem}.  We now move on to defining and estimating the new unresolved flux density and current $\ost{\kk}$ and $\ost{\vp}$.



\subsection{Calculating the unresolved flux density and current} \label{sec:newUnresolvedTerms}
In this Section, we define and estimate the new unresolved flux density and current, $\ost{\kk}$ and $\ost{\vp}$.  

According to the definition of a dissipative Euler-Reynolds flow, we must write the new resolved energy flux density $\ost{\DD} = \DD[\ost{v}, \ost{p}]$ in the form
\ali{
\ost{\DD} &= \ost{D}_t \ost{\kk} + \nb_j[\ost{v}_\ell \ost{R}^{j\ell}] + \nb_j \ost{\vp}^j - \ost{\mu}, \label{eq:newFluxGoal}
}
where $\ost{R}^{j\ell} = R_T^{j\ell} + R_S^{j\ell} + R_H^{j\ell} + R_M^{j\ell}$ is exactly the new stress that has been estimated in Section~\ref{sec:boundNewStress}, $\ost{v}^\ell = v^\ell + V^\ell$ and $\ost{D}_t = (\pr_t + \ost{v}^j \nb_j)$.  To see that this goal may be accomplished, we must properly rewrite the terms introduced in Section~\ref{sec:mainStep1} in lines \eqref{eq:newDensity}, \eqref{eq:phiTermLeft}-\eqref{eq:highkkTerm}.  In the process, we will frequently encounter the vector field $\wtld{V}^\ell := \ost{v}^\ell - v_\ep^\ell = V^\ell + (v^\ell - v_\ep^\ell)$, which we note satisfies the bound
\ali{
\brh[\wtld{V}] &\leq  \brh[V] + \brh[v - v_\ep] \lsm e_\vp^{1/2},  \label{eq:tildeVLevel}
}
where we used Propositions~\ref{prop:allTheEndBounds} and \ref{prop:mollErrProp} (specifically $\brh[v - v_\ep] \lsm e_v^{1/2}/\osn$) and $\osn \geq (e_v^{1/2}/e_\vp^{1/2})$.

We construct $\ost{\kk}$ to have the form $\ost{\kk} = \kk_L + \kk_M + \kk_H$, where $\kk_M = \kk - \kk_\ep$ appears in $\DDR_{\kk H}$ in \eqref{eq:highkkTerm}, $\kk_L$ is defined in \eqref{eq:leftOverDensity}, and $\kk_H$ will be extracted by writing the term \eqref{eq:transFluxTerm} in the following form
\ali{
\begin{split}
\mathring{\DD}_T &= \DDR_{TA} + \DDR_{T\kk} + \DDR_{TB}  \\
\DDR_{TA} := \Ddt(v_\ep \cdot V) + \nb_j &\left[ \left( \fr{|v_\ep|^2}{2} + p_{\ep_v}\right) V^j \right], \qquad 
\DDR_{T\kk} := \Ddt( (v - v_\ep) \cdot V ) = \Ddt( \kk_H). \label{eq:firstTransTermsFlux}
\end{split}
}
In \eqref{eq:firstTransTermsFlux} we have chosen the $\kk_H = (v - v_\ep) \cdot V$ part of $\ost{\kk} = \kk_L + \kk_M + \kk_H$.  Now \eqref{eq:fltermLow}, \eqref{eq:leftOverDensity}, \eqref{eq:firstTransTermsFlux} give
\ali{
\DD_{\kk L} + \Ddt (\kk - \kk_\ep) + \DDR_{T\kk} &= \Ddt[\kk_L + \kk_M + \kk_H] - \Ddt e(t) = \ost{D}_t[\ost{\kk}] - \nb_j[\ost{\kk}\wtld{V}^j] - \Ddt e(t). \label{eq:densityFluxTerm}
}
Recalling that $\ost{\mu} = \mu - \Ddt e(t)$, we see already the terms $\ost{D}_t \ost{\kk}$ and $- \ost{\mu}$ that are required in \eqref{eq:newFluxGoal}.  We find the remaining terms by manipulating the components of \eqref{eq:newDensity}-\eqref{eq:vpTerms} such as.
\ali{
\DDR_{S} &= \nb_j\left[v_\ell \left(\sum_I V_I^j \overline{V_I^\ell} + P \de^{j\ell} + R^{j\ell} \right) \right] = \nb_j\left[ v_\ell ( R_S^{j\ell} + (R^{j\ell} - R_\ep^{j\ell}) ] \right]  \notag \\
\DDR_{S} &= \nb_j[ \ost{v}_\ell ( R_S^{j\ell} + (R^{j\ell} - R_\ep^{j\ell}) )   ] - \nb_j[ V_\ell (R_S^{j\ell} + (R^{j\ell} - R_\ep^{j\ell}) ]. \label{eq:stressCurrTerm1}
}
Our general aim will be to absorb terms such as the latter terms in \eqref{eq:stressCurrTerm1} into the new unresolved flux current $\ost{\vp}$.  However, the particular term \eqref{eq:stressCurrTerm1} poses a difficulty as the term $V_\ell R_S^{j\ell}$ is in fact too large to be absorbed into the current $\ost{\vp}$.  
We will treat this difficulty later in Section~\ref{sec:mainTermProblemCurrent}.  For now we move on to the flux term $\DDR_H$, which we treat as follows
\ali{
\DDR_H &= \nb_j\left[ v_\ell\left( \sum_{J \neq \bar{I}} V_I^j V_J^\ell\right) \right]  \notag \\
\DDR_H &= \sum_{J \neq \bar{I}} \nb_j v_{\ep\ell} V_I^j V_J^\ell  + \sum_{J \neq \bar{I}} \nb_j[ (v - v_\ep)_\ell  V_I^j V_J^\ell ] + \DDR_{HB}, \label{eq:DDRH1term} \\
 \DDR_{HB} &= \sum_{J \neq \bar{I}} v_{\ep\ell} \nb_j (V_I^j V_J^\ell ) \notag \\
\DDR_{HB} &\stackrel{\eqref{eq:highFreqTermThird}}{=}  v_{\ep\ell} \nb_j R_H^{j\ell} = \nb_j [ v_{\ep\ell} R_H^{j\ell} ] - R_H^{j\ell} \nb_j v_{\ep\ell}. \label{eq:DDRHBterm} 
}
We expand $\DDR_T$ using that $|v|^2 = |v_\ep|^2 + 2 (v - v_\ep) \cdot v + |v - v_\ep|^2$ to obtain
\ali{
\DDR_T &= \pr_t(v_\ell V^\ell) + \nb_j[ v_\ell V^\ell v^j] + \nb_j \left( \left( \fr{|v|^2}{2} + p \right) V^j \right) \notag \\
\DDR_T &= \Ddt(v_\ell V^\ell) + \nb_j[v_\ell( (v^j - v_\ep^j) V^\ell + V^j (v^\ell - v_\ep^\ell) ) ] \notag\\
&+ \nb_j\left( \left(\fr{|v_\ep|^2}{2} + p_{\ep_v}\right)V^j \right) + \nb_j \left[ \left( \fr{|v - v_\ep|^2}{2} + (p - p_{\ep_v}) \right) V^j \right] \notag \\
\DDR_T &= \DDR_{TA} + \Ddt[(v-v_\ep)\cdot V] + \DDR_{TB}, \label{eq:DDRTexpand} \\
\DDR_{TB} &\stackrel{\eqref{eq:firstTransTermsFlux}}{=} \nb_j[ v_\ell R_{Mv}^{j\ell} ] + \nb_j \left[ \left( \fr{|v - v_\ep|^2}{2} + (p - p_{\ep_v}) \right) V^j \right].
}
For the term $\DDR_{TA}$ appearing in \eqref{eq:firstTransTermsFlux},\eqref{eq:DDRTexpand}, we calculate as follows, using that $\nb_j V^j = 0$
\ali{
\DDR_{TA} &= \Ddt(v_\ep \cdot V) + v_{\ep\ell} V^j \nb_j v_\ep^\ell + V^j \nb_j p_{\ep_v} \notag \\
&= v_{\ep\ell}( \Ddt V^\ell + V^j \nb_j v_\ep^\ell ) + V_\ell( \Ddt v_\ep^\ell + \nb^\ell p_{\ep_v} ) \notag  \\
\begin{split}
\DDR_{TA} &\stackrel{\eqref{eq:transportTermThird}}{=} v_{\ep\ell}\nb_j R_T^{j\ell} +  V_\ell( \Ddt v_\ep^\ell + \nb^\ell p_{\ep_v} )  \\
&= \nb_j[v_{\ep\ell} R_T^{j\ell}] - \nb_j v_{\ep\ell} R_T^{j\ell} + V_\ell( \nb_j R_{\ep_v}^{j\ell} + Q_\ep^\ell[v,v] ).
\end{split} \label{eqDdrTAform}
}
In the last line, we have used the Euler-Reynolds equations for $v^\ell$ (mollified by the kernel $\chi_\ep \ast$).  The term $Q_\ep^\ell[v,v] = v_\ep^j \nb_j v_\ep^\ell - \chi_\ep \ast(v^j \nb_j v^\ell)$ is the commutator that arises from this mollification, which we encountered previously in line \eqref{eq:commutermDef}.

From these calculations, we have isolated terms of the form $\nb_j[v_{\ep\ell} \wtld{R}^{j\ell}]$ or $\nb_j[v_{\ell} \wtld{R}^{j\ell}]$ for all of the terms $\wtld{R}^{j\ell} \in \{ R_S, R_T, R_H, R_M = R_{Mv}^{j\ell} + (R^{j\ell} - R_\ep^{j\ell}) \}$ that compose the new stress $\ost{R}$.  Since we have specified $\ost{\kk}$ and $\ost{\mu}$ (see \eqref{eq:densityFluxTerm}), our goal of achieving the form \eqref{eq:newFluxGoal} will be accomplished if we choose a $\ost{\vp}^j$ that satisfies the following
\ali{
\ost{\vp}^j &= \vp_S^j + \vp_H^j + \vp_T^j + \vp_M^j + \vp_{\kk}^j + \vp_L^j + \vp_{\calT}^j \label{eq:allTheVps} \\
\nb_j \vp_S^j &= - \nb_j[V_\ell R_S^{j\ell}] \label{eq:stressFlux} \\
\nb_j \vp_H^j &= \nb_j v_{\ep\ell}\Big( \sum_{J \neq \bar{I}} (V_I^j V_J^\ell) - R_H^{j\ell} \Big) - \nb_j[\wtld{V}_\ell R_H^{j\ell}] \label{eq:HHflux} \\
\nb_j \vp_T^j &= - \nb_j v_{\ep\ell} R_T^{j\ell} + V_\ell(\nb_j R_{\ep_v}^{j\ell} + Q_\ep^\ell[v,v]) -\nb_j[\wtld{V}_\ell R_T^{j\ell}], \label{eq:Tflux}
}
we recall that $\vp_L^j$ is defined in \eqref{eq:vpLdef}, and the remaining terms satisfy
\ali{
\vp_M^j &= -V_\ell R_M^{j\ell} + \sum_{J \neq \bar{I}}(v - v_\ep)_\ell V_I^j V_J^\ell + \left( \fr{|v - v_\ep|^2}{2} + (p - p_{\ep_v}) \right) V^j \label{eq:mollTerm} \\
\nb_j \vp_\kk^j &= - \nb_j(\ost{\kk} \wtld{V}^j) + \sum_{J \neq \bar{I}} \Ddt \left( \fr{V_I \cdot V_J}{2} \right) + \nb_j\left(  \left(\fr{|V|^2}{2} + \kk\right)(v^j - v_\ep^j) \right) \label{eq:kkCurrterm} \\
\nb_j \vp_\calT^j &= \nb_j\Big[ \sum_{I,J,K\notin \calT} (V_I)_\ell V_J^j V_K^\ell + P V^j + (\vp^j - \vp_\ep^j)\Big]. \label{eq:lastFluxTerm}
}
The above equations have been derived from \eqref{eq:stressCurrTerm1},\eqref{eq:DDRH1term}-\eqref{eq:DDRHBterm} ,\eqref{eq:DDRTexpand}-\eqref{eqDdrTAform}, \eqref{eq:densityFluxTerm}, \eqref{eq:DDflmL}-\eqref{eq:highkkTerm}, and \eqref{eq:phiTermLeft}.

\subsection{The main term in the unresolved flux current} \label{sec:mainTermProblemCurrent}
As in the statement of Lemma~\ref{lem:mainLem}, we require that $\ost{\vp}^j$ have the form $\ost{\vp}^j = \ost{\vp}_{[2]}^j + \ost{\vp}_{[G]}^j$, where $\ost{\vp}_{[2]}^j$ is the larger term, which takes values in $\ker dx^2$ and has size $e_{\undvp}^{3/2}$, while $\ost{\vp}_{[G]}^j$ is the smaller term, which is required to have size $\lsm e_{\undvp}^{1/2} e_G$.  We define $\vp_{[2\ast]}^j$ to have two parts, one being the term $\vp_{(2)}^j$ from \eqref{eq:vpLdef}, and the other will be extracted from the term $- V_\ell R_S^{j\ell}$ in \eqref{eq:stressFlux}.

The term $- V_\ell R_S^{j\ell}$ presents a special difficulty.  Naively, one might try to absorb this term into the new $\ost{\vp}^j = \ost{\vp}^j_{[2]} + \ost{\vp}_{[G]}^j$.  However, the main part of the term $V_\ell R_S^{j\ell} = V_\ell[ P \de_{[2\ast]}^{j\ell} + R_{[2 \ast]}^{j\ell}] + \ldots$, though it does take values in $\ker dx^2$, is too large to be acceptable.  Most seriously, it is impossible to achieve the inequality $e(t)^{-3/2}\co{V_\ell[ P \de_{[2\ast]}^{j\ell}} \leq \bar{\de}$ required in the definition of $\bar{\de}$-well-preparedness, as the term $V_\ell P \de_{[2\ast]}^{j\ell}$ is already proportional to $e(t)^{3/2} \cdot O(1)$ by design.  This term therefore threatens the entire success of being able to continue the iteration.  



What saves us from this difficulty is that, even though the term $V_\ell P \de_{[2\ast]}^{j\ell}$ is not permissibly small, the {\it divergence} of this vector field exhibits an important cancellation.  To see this fact, recall the formula for $\de_{[2\ast]} = e_1 \otimes e_1 + (e_3 \otimes e_3)/2$ expressed following line \eqref{eq:prop3basis}.   Recall also that $\nb \hxii_I \in \langle (1, 0 , 0) \rangle$, while $\hat{v}_I = \eta_I \ga_I \htf_I$ takes values in $\langle (1,0,0) \rangle^\perp = \langle (0,1,0), (0,0,1) \rangle$.  Combining these together we find
\ali{
\begin{split} \label{eq:amazingCancel}
(\nb_j\hxii_I \hat{v}_{I\ell} )\de_{[2\ast]}^{j\ell} &= 0  \\
\nb_j[V_\ell P \de_{[2\ast]}^{j\ell}] &= \sum_I (i \la) e^{i \la \xi_I} \nb_j \xi_I v_{I\ell} P \de_{[2\ast]}^{j\ell} + \ldots \\
\nb_j[V_\ell P \de_{[2\ast]}^{j\ell}] &= \sum_I (i \la) e^{i \la \xi_I} (\nb_j \hxii_I  (\hat{\de}v_I)_\ell + (\hat{\de} \nb_j \xi_I) v_{I\ell}) P \de_{[2\ast]}^{j\ell} + \ldots
\end{split}
}
Taking advantage of the cancelation in \eqref{eq:amazingCancel} will be crucial for closing the estimates of Section~\ref{sec:bdNewUnresolvedTerms}.

For the next term $\nb_j[V_\ell R_{[2\ast]}^{j\ell}]$ within $\nb_j[V_\ell R_S^{j\ell}]$, the same cancellation is not available.  In this case, we will absorb the term into the new $\ost{\vp}_{[2]}^j$, which is reasonable since the term takes values in $\ker dx^2$ (using that $R_{[2\ast]}^{j\ell}$ takes values in $\ker dx^2 \otimes \ker dx^2$) and has a permissible order of magnitude $\lsm e_\vp^{1/2} e_R$.

The final term within $V_\ell R_S^{j\ell}$ that presents a difficulty is the term $V_\ell R_{[3\ast]}^{j\ell}$.   This term cannot be absorbed into $\ost{\vp}_{[2]}^j$ as it does not take values in $\ker dx^2$, and its order of magnitude $e_\vp^{1/2} e_R$ is too large to be absorbed into the lower order term $\ost{\vp}_{[G]}^j$, which cannot be any larger than $\lsm e_{\undvp}^{1/2} e_G$ to continue the iteration.  What allows us to treat this term is a combination of the above ideas.  Namely, using that $R_{[3\ast]}^{j\ell}$ takes values in $\ker dx^3 \otimes \ker dx^3$, we can decompose this tensor using projection operators $\pi_\times$ and $\pi_{\parallel}$ on $\ker dx^3 \otimes \ker dx^3$.  We require for these operators that, on $\ker dx^3 \otimes \ker dx^3$, we have  $ \pi_{\times} + \pi_{\parallel} = 1$, and $\pi_{\parallel}$ takes values in the span of $\langle e_1 \otimes e_1, e_2 \otimes e_2 \rangle$ while $\pi_{\times}$ takes values in the span of $\langle e_1 \otimes e_2 + e_2 \otimes e_1 \rangle$.  For $\pi_{\parallel} R_{[3\ast]}^{j\ell}$, we have a cancellation similar to that of the pressure term: 
\ali{
\begin{split} \label{eq:amazingCancelParallel}
\nb_j[V_\ell \pi_{\parallel} R_{[3\ast]}^{j\ell}] = \sum_I (i \la) e^{i \la \xi_I} (\nb_j \hxii_I  (\hat{\de}v_I)_\ell + (\hat{\de} \nb_j \xi_I) v_{I\ell}) \pi_{\parallel} R_{[3\ast]}^{j\ell} + \ldots
\end{split}
}
where we used that $(\nb_j\hxii_I \hat{v}_{I\ell} )\pi_{\parallel} R_{[3\ast]}^{j\ell} = 0$.

What remains is the term $V_\ell (\pi_\times R_{[3\ast]}^{j\ell})$, which can be safely absorbed into $\ost{\vp}_{[2]}^j$ as follows.  Note that $V^\ell$ takes values essentially in $\ker dx^1$ in that we can write $V_I^\ell = \widehat{V}_I^\ell + \hat{\de} V_I^\ell$, where $\widehat{V}_I^\ell = e^{i \la \xi_I} \hat{v}_I^\ell$ takes values in $\ker dx^1 = \langle e_2, e_3 \rangle$.  On the other hand, $\pi_\times R_{[3\ast]}^{j\ell}$ takes values in $\langle e_1 \otimes e_2 + e_2 \otimes e_1 \rangle$ and hence maps $\ker dx^1$ to $\langle e_1 \rangle \subseteq \ker dx^2$.  Therefore, if we now define
\ali{
\ost{\vp}_{[2]}^j &:= \vp_{(2)}^j - V_\ell R_{[2\ast]}^{j\ell} - \widehat{V}_\ell {\pi_{\times} }R_{[3\ast]}^{j\ell} \label{eq:chosenMain2termflux}
}
with $\vp_{(2)}^j$, $R_{[l\ast]}^{j\ell}$ as in \eqref{eq:secondComponentVp},\eqref{eq:threeStressTerms} and $\widehat{V}^\ell = \sum_I \widehat{V}_I^\ell$, we have that $\ost{\vp}_{[2]}^j$ takes values in $\ker dx^2$ as desired.

With $\ost{\vp}_{[2]}^j$ chosen we can now verify the condition of $(\bar{\de}, \hc)$-well-preparedness in Definition~\ref{defn:wellPrepared} for the new dissipative Euler-Reynolds flow.  Recall from \eqref{eq:threeStressTerms}, \eqref{eq:choseMainNewterm},\eqref{eq:Pdefn} that the main term in the new stress $\ost{R}_{[2]}^{j\ell} = P \de_{[2\ast]}^{j\ell} + R_{[2\ast]}^{j\ell}$ expands to become
\ali{
\ost{R}_{[2]}^{j\ell} &= -(2e(t)/3) \de_{[2\ast]}^{j\ell} + \ost{R}_{[2\circ]}^{j\ell} \\
\ost{R}_{[2\circ]}^{j\ell} &= (2 \kk_{[G\ep]} + \mbox{tr}(R_{[2\ep]} + R_{[G\ep]} ))\de_{[2\ast]}^{j\ell} + R_{[2\ep]}^{j\ell} + \pi_{[2]} R_{[G\ep]}^{j\ell}, \label{eq:circTermNew} \\
e^{1/2}(t) &= 2(K_0 e_{\undvp})^{1/2} \eta_\tau \ast_t 1_{\leq \sup \tilde{I}_{[G]} + \tau/100}(t)
}
where the last formula is as in \eqref{eq:enIncdef}.  We will set the new $\ost{\bar{e}}^{1/2}(t)$ to be $\ost{\bar{e}}^{1/2}(t) = \sqrt{2/3} e^{1/2}(t)$.  We set the parameter $\ost{\undl{e}}_\vp$ to be $\ost{\undl{e}}_\vp = (4/3) K_0 e_{\undvp}$, which is the constant value obtained by $\ost{\bar{e}}(t)$ on the interval $t \leq \sup \tilde{I}_{[G]} + 10^{-2}(\Xi e_v^{1/2})^{-1}$.  Note in particular that $\ost{\undl{e}}_\vp \geq \hc^{-1} \ost{e}_\vp$, since $\ost{e}_\vp = \hc e_{\undvp}$.

To check condition \eqref{eq:barDeAssump}, observe from our choice of $\ep_t \leq 10^{-3} (\Xi e_v^{1/2})^{-1}$ that $\ost{R}_{[2\circ]}$ and $\ost{\vp_{[2]}}$ are supported in
\ali{
\suppt (\ost{R}_{[2\circ]}, \ost{\vp_{[2]}}) &\subseteq \tilde{I}_{[2]} := \{ t + \bar{t} ~:~ t \in \tilde{I}_{[G]}, |\bar{t}| \leq 10^{-3} (\Xi e_v^{1/2})^{-1} \}
}
while we have $\suppt e^{1/2}(t) \subseteq \tilde{I}_{[G\ast]} := \{ t \leq \sup \tilde{I}_{[G]} + (\Xi e_v^{1/2})^{-1} / 40 \}$ as desired in \eqref{eq:suppContained}.  

We next check \eqref{eq:underBdsEninc}.  Recall that the new frequency energy levels have the form $(\ost{\Xi}, \ost{e}_v) {=} (\hc N \Xi, \hc e_\vp)$ for some $\hc \geq B_\la \geq 1$ as in \eqref{eq:frEnLvlsNew}.  In particular, since $(\ost{\Xi} \ost{e}_v^{1/2})^{-1} \leq 10^{-4} (\Xi e_v^{1/2})^{-1}$ for $\hc$ large enough and $N \geq e_v^{1/2}/e_\vp^{1/2}$, we have that $\ost{\bar{e}}(t)$ takes on the constant value $\ost{\undl{e}}_\vp$ on the interval $t \leq \sup \tilde{I}_{[2]} + (\ost{\Xi} \ost{e}_v^{1/2})^{-1}$ by our choice of $\ep_t \leq 10^{-3} (\Xi e_v^{1/2})^{-1}$, which immediately implies the bounds \eqref{eq:underBdsEninc} for $\ost{\bar{e}}^{-1/2}$ and $\ost{\undl{e}}_\vp$.

We now confirm that the crucial smallness condition \eqref{eq:barDeAssump} holds for $K_0$ chosen sufficiently large depending only on $L$.  Recall that $\ost{\bar{e}}(t) = \ost{\undl{e}}_\vp$ is constant and equal to its maximum value on the interval $\tilde{I}_{[2]}' = \{ t \leq \sup \tilde{I}_{[2]} + (\ost{\Xi} \ost{e}_v^{1/2}) \}$ where condition \eqref{eq:barDeAssump} is required.  
We will assume the following bound:
\ali{
\max_I \co{\de V_I} + \max_I \co{\hat{\de} v_I} &\unlhd 10^{-40} e_\vp^{1/2}. \label{eq:lowerOrderAmpassump}
}
The bound \eqref{eq:lowerOrderAmpassump} will follow easily from \eqref{eq:initAmpErrBds},\eqref{eq:lowerTermHFBd} and $\ost{b} \leq B_\la^{-1/2}$ by later taking $B_\la$ large depending on our choice of $K_0$.  Assuming \eqref{eq:lowerOrderAmpassump}, we have for all  $t \in \tilde{I}_{[2]}'$ that
\ali{
\co{V} \leq \co{\sum_{I \in \ovl{\II}} V_I } + \co{\sum_{I \in \dmd{\II}} V_I } &\leq A (\co{\bar{e}^{1/2}(t)} + \co{e^{1/2}(t)} + e_\vp^{1/2} ) \label{eq:checkampBds} \\
\ost{\bar{e}}(t)^{-3/2} (V_\ell R_{[2\ast]}^{j\ell}) &= (3/2) e(t)^{-3/2}  (V_\ell R_{[2\ast]}^{j\ell}) \\
\co{ \ost{\bar{e}}(t)^{-3/2} (V_\ell R_{[2\ast]}^{j\ell}) } &\stackrel{\eqref{eq:checkampBds},\eqref{eq:givenAmpBds}}{\leq} A e(t)^{-3/2} ( \co{e(t)^{1/2}} + e_\vp^{1/2} ) e_R, \label{eq:maintroublePrepterm}
}
where the constant $A$, which varies line by line, crucially does not depend on the choice of $K_0$.  Recalling that $e_{\undvp}^{3/2} = e_\vp^{1/2} e_R$ (hence $e_{\undvp} \geq e_R$) and the constant value of $e(t) = 4 K_0 e_{\undvp} = \co{e(t)}$, the right hand side of \eqref{eq:maintroublePrepterm} is bounded by $A (K_0^{-1} + K_0^{-3/2})$ for some $A$ independent of $K_0$.  Taking $K_0$ large, we can bound \eqref{eq:maintroublePrepterm} by $\bar{\de}/40$.  We may also bound $\ost{\bar{e}}(t)^{-3/2} \co{\widehat{V}_\ell \pi_\times R_{[3\ast]}^{j\ell}}$ for $\in \tilde{I}_{[2]}'$ by $\bar{\de}/40$ for $K_0$ sufficiently large by a simpler version of the same argument.

Using the bounds $\co{\ost{R}_{[2\circ]}} \leq A e_R$ and $\co{\vp_{(2)}^j} \leq A e_{\undvp}^{3/2}$ with $A$ independent of $K_0$, the latter deriving from $\vp_{(2)}^j = \pi_{[2]} \vp_{[G\ep]}^j$, we finally choose $K_0$ large enough such that for all $t \in \tilde{I}_{[2]}'$
\ali{
\bar{e}(t)^{-1} \co{\ost{R}_{[2\circ]}} + \bar{e}(t)^{-3/2} \co{\vp_{(2)}} &\leq A (K_0 e_{\undvp})^{-1} e_R + (K_0 e_{\undvp})^{-3/2} e_{\undvp}^{3/2} \leq \bar{\de}/10.  \label{eq:atLastK0}
}
This choice of $K_0$ guarantees the smallness condition \eqref{eq:barDeAssump} and also the stated inequality \eqref{eq:goalForK0}.  

We are now ready to prove the bounds \eqref{eq:cicVp1bds}, \eqref{eq:givenAmpBds}, which must be checked for the new frequency energy levels $(\ost{\Xi}, \ost{e}_v) = (\hc N \Xi, \hc e_\vp)$ where $\hc \geq B_\la > 1$ is a large constant to be chosen at the conclusion of the proof.  Since we have that $\bar{e}(t) \leq \ost{\undl{e}}_\vp$, the $C^0$ bound in \eqref{eq:cicVp1bds} follows from the smallness condition~\eqref{eq:barDeAssump} so it is enough to prove \eqref{eq:cicVp1bds} for $0 < r + |\va| \leq L$.  We will use the weighted seminorm
\ali{
\undl{\widehat{H}}[F] &:= \max_{0 \leq r \leq 1} (\ost{\Xi} \ost{e}_v^{1/2})^{-r} \max_{0 < |\va| + r \leq L } \ost{\Xi}^{-|\va|} \co{\nb_{\va} \ost{D}_t F }
}
and the property that $\undl{\widehat{H}}[F] \leq \hc^{-1} B_\la \undl{\ost{H}}[F] \leq \hc^{-1} B_\la \ost{H}[F]$ 
where $\undl{\ost{H}}[F]$ is defined as in Lemma~\ref{lem:comparisonLem}, which follows from $\ost{\Xi} = \hc B_\la^{-1} \la$ and $\ost{e}_v = \hc e_\vp$.  Using Lemma~\ref{lem:comparisonLem} and \eqref{eq:circTermNew} we obtain
\ALI{
\ost{H}[\ost{R}_{[2\circ]}] \lsm \brh[\ost{R}_{[2\circ]}] &\lsm e_R \lsm \ost{\undl{e}}_\vp \\
\ost{H}[\vp_{(2)}] \lsm \brh[\vp_{[G\ep]}] &\lsm e_G \lsm \ost{\undl{e}}_\vp \\
\ost{H}[V_\ell R_{[2\ast]}^{j\ell} + \widehat{V}_\ell \pi_\times R_{[3\ast]}^{j\ell} ] &\lsm \brh[ V_\ell ] \brh[ R_{[2\ast]}^{j\ell} ] + \brh[\widehat{V}_\ell] \brh[ R_{[3\ast]}^{j\ell}] \\
&\lsm e_\vp^{1/2} e_R + e_\vp^{1/2} e_G \leq 2 \undl{\ost{e}}_\vp^{3/2}.
}
Assuming $\hc$ large enough compared the above implied constants and $B_\la$, we obtain $\undl{\widehat{H}}[R_{[2\circ]}] \leq 2^{-1} \undl{\ost{e}}_\vp $ and $\undl{\widehat{H}}[\ost{\vp}_{[2]}^j] \leq 2^{-1} \undl{\ost{e}}_\vp^{3/2}$, which is sufficient for the $0 < r + |\va| \leq L$ cases of the desired bound \eqref{eq:cicVp1bds}.  

The desired bound \eqref{eq:givenAmpBds} of $\co{\ost{D}_t^s \ost{\bar{e}}^{1/2} } \unlhd (\ost{\Xi} \ost{e}_v^{1/2})^s \ost{e}_\vp^{1/2}$ for $0 \leq s \leq 2$ and $\ost{e}_\vp := \hc e_{\undvp}$ follows by taking $\hc$ large and using the bounds $\co{\pr_t^s e^{1/2}(t) } \lsm (\Xi e_v^{1/2})^s e_{\undvp}^{1/2}$ and the fact that $\ost{\bar{e}}(t) = (2/3) e(t)$ depends only on $t$.  This bound proves \eqref{eq:formOfR1de}-\eqref{eq:cicVp1bds} for stage $2$ with the new frequency energy levels, the time intervals $\tilde{I}_{[2]} \subseteq \tilde{I}_{[G\ast]}$, and the function $(2/3)e(t)$.

\subsection{Estimating the unresolved flux density and current} \label{sec:bdNewUnresolvedTerms}
With our choice of $\ost{\vp}_{[2]}^j$ made in \eqref{eq:chosenMain2termflux} and the list of terms \eqref{eq:allTheVps}, \eqref{eq:vpLdef} in hand, we are now ready to estimate the components of the unresolved flux density and current, $\ost{\kk}$ and  $\ost{\vp}^j = \ost{\vp}_{[2]}^j + \ost{\vp}_{[G]}^j$, beginning with $\ost{\kk}$. the leading order terms $\ost{\kk}_{[2]}$ and $\ost{\vp}_{[2]}^j$.

Recall the decomposition of $\ost{\kk} = \kk_L + \kk_M + \kk_H$ with $\kk_L$ defined in \eqref{eq:leftOverDensity}.   We decompose $\kk_L$ as $\ost{\kk}_{[2]} + \kk_{LG}$, where $\ost{\kk}_{[2]} =(1/2) \de_{j\ell}(P \de_{[2\ast ]}^{j\ell} + R_{[2\ast]}^{j\ell}) = (1/2) \de_{j\ell} \ost{R}_{[2]}^{j\ell}$, and we set 
$\ost{\kk}_{[G]} = \kk_{LG} + \kk_M + \kk_H$.  We will use the following estimates, stated in terms of the weighted norms of Lemmas~\ref{lem:accountLem1},\ref{lem:comparisonLem}.   

\begin{prop} \label{eq:lfrqkkbds}  The following bounds hold for all $M \geq L$ with constants that may on $M$
\ali{
\brh[\ost{\kk}_{[2]}] &\lsm \brh^{(M,1)}[ \ost{\kk}_{[2]} ] \lsm e_{\undvp} \label{eq:ostk2bds} \\
\brh[ \kk_{LG}] &\lsm\brh^{(M,1)}[ \kk_{LG} ] \lsm_M e_G \\
\brh[\kk_L] + \brh[\kk_{L\ep_v}] \lsm& \, \brh^{(M,1)}[\kk_L] + \brh^{(M,1)}[\kk_{L\ep_v}] \lsm e_{\undvp} \label{eq:lowkkbdsum} \\
\brh[\kk_L - &\kk_{L\ep_v} ] \lsm (e_{\undvp}/\osn) \label{eq:mollErrkkLbd} \\
\brh[\ostk_{[G]}] &\lsm e_G. \label{eq:ostkkGbdd} 
}
\end{prop}
\begin{proof} The bound \eqref{eq:ostk2bds} for $\ostk_{[2\ast]}$ follows from the bounds for $P$ and $R_{[2\ast]}^{j\ell}$ in Proposition~\ref{eq:pressIncBound} and \eqref{eq:coBdFieldsmoll}-\eqref{eq:c0DtbdFmoll}.  For $\kk_{LG}$ we apply \eqref{eq:coBdFieldsmoll}-\eqref{eq:c0DtbdFmoll} to $(F,h_F) = (R_{[3]}, e_R)$ and use Proposition~\ref{eq:lfrqkkbds} to obtain
\ali{
\brh^{(M,1)}[\kk_{LG}] &\lsm \brh^{(M,1)}[R_{[3\ast]}] + \max_I \brh^{(M,1)}[\de v_I]\left(\brh^{(M,1)}[v_I] + \brh^{(M,1)}[\de v_I]\right) \notag \\
&+ \max_I \brh^{(M,1)}[\hat{\de} v_I]\left(\brh^{(M,1)}[v_I] + \brh^{(M,1)}[\hat{\de} v_I]\right) \notag \\
\brh^{(M,1)}[\kk_{LG}] &\lsm e_G + (B_\la \osn)^{-1} e_\vp + \ost{b} e_\vp \lsm e_G 
}
where we have used \eqref{eq:Nastdef} and 
$\ost{b}\lsm(e_v^{1/2}/e_\vp^{1/2}\osn)^{1/2}$.  The bound \eqref{eq:lowkkbdsum} for $\kk_L$ now follows from $\kk_L = \ostk_{[2]} + \kk_{LG}$, while \eqref{eq:lowkkbdsum} for $\kk_{L\ep_v}$ follows by applying \eqref{eq:coBdFieldsmoll}-\eqref{eq:c0DtbdFmoll} to $(F,h_F) = (\kk_L, e_\vp)$ (which satisfies the bounds required of Proposition~\ref{prop:prelimBds} thanks to \eqref{eq:lowkkbdsum} for $\kk_L$).  Similarly, we may apply Proposition~\ref{prop:mollErrProp} to $(F,h_F) = (\kk_L, e_\vp)$ to obtain \eqref{eq:mollErrkkLbd}.

Recalling $\ostk_{[G]} = \kk_{LG} + (\kk - \kk_\ep) + (v - v_\ep) \cdot V$ and applying Proposition~\ref{prop:mollErrProp} to $\kk$ and $v$ gives
\ali{
\brh[\ostk_{[G]}] &\lsm e_G + (e_v^{1/2}/e_\vp^{1/2})(e_\vp/\osn) + (e_v^{1/2}/\osn)e_\vp^{1/2} \lsm e_G,
}
where we used again \eqref{eq:Nastdef} .  
\end{proof}

We now turn to estimating the components of $\ost{\vp}^j$.  The main term $\ost{\vp}_{[2]}^j$ has been estimated in Section~\ref{sec:mainTermProblemCurrent} by $\brh[\ost{\vp}_{[2]}^j] \lsm \ost{\undl{e}}_\vp^{3/2} \lsm e_{\undvp}^{3/2}$, which is the desired bound for this term.  We gather the remaining terms into $\ost{\vp}_{[G]}^j = \ost{\vp}^j - \ost{\vp}_{[2]}^j$, which will require a more involved treatment.

The terms in $\ost{\vp}_{[G]}^j$ can be gathered into three types $\ost{\vp}_{[G]}^j = \vp_{[GM]}^j + \vp_{[GA]}^j + \vp_{[GB]}^j$.  ``Type $M$'' terms arise from errors in mollifications.  We estimate these as follows

\begin{prop} \label{prop:mollVp} Let $\vp_{[GM]}^j$ be the sum of terms in \eqref{eq:mollTerm}-\eqref{eq:lastFluxTerm} that are in divergence form and involve mollification errors of the form $F - F_\ep$ or $F - F_{\ep_v}$.  Then
\ali{
\brh[\vp_{[GM]}^j] &\lsm (e_ve_\vp^{1/2}/\osn) \lsm e_{\undvp}^{1/2} e_G,
}
where the second inequality follows from the assumption \eqref{eq:Nastdef}. 
\end{prop}
\begin{proof} We start by bounding $\vp_M^j$ in \eqref{eq:mollTerm} using Proposition~\ref{prop:mollErrProp}
\ALI{
\brh[\vp_{M}^j] &\lsm \brh[V_\ell R_M^{j\ell}] + \max_{I,J} \brh[v - v_\ep]\brh[V_I]\brh[V_J] + (\brh[v - v_\ep]^2 + \brh[p - p_{\ep_v}])\brh[V]  \\
\brh[\vp_{M}^j] &\lsm \brh[V](\brh[v - v_\ep] \brh[V] + \brh[R - R_\ep] ) + (e_v^{1/2}/\osn) e_\vp + ((e_v^{1/2}/\osn)^2 + (e_v/\osn))e_\vp^{1/2} \\
&\lsm e_\vp (e_v^{1/2}/\osn) + e_\vp^{1/2}(e_v^{1/2}/e_\vp^{1/2})(e_\vp/\osn) + (e_v/\osn)e_\vp^{1/2} \lsm e_ve_{\vp}^{1/2}/\osn 
}
The mollification errors in \eqref{eq:kkCurrterm} including the terms from $\ost{\kk} = \kk_L + (\kk - \kk_\ep) + (v - v_\ep) \cdot V $ may be bounded using Lemma~\ref{lem:comparisonLem} (including the inequality $\brh[F] \lsm H[F]$), and Proposition~\ref{prop:mollErrProp} to give
\ALI{
\vp_{\kk M}^j = -\wtld{V}^j[ (\kk - \kk_\ep) + (v-v_\ep)\cdot V] &+  \left(\fr{|V|^2}{2} + \kk \right) ( v - v_\ep ) \\
\brh[\vp_{\kk M}] \lsm e_\vp^{1/2}\left[ \left( \fr{e_v^{1/2}}{e_\vp^{1/2}} \right) \left( \fr{e_\vp}{\osn}\right) + \left( \fr{e_v^{1/2}}{\osn}\right) e_\vp^{1/2} \right] &+ (\brh[V]^2 + H[\kk] )(e_v^{1/2}/\osn) \lsm e_v^{1/2} e_\vp/\osn.
}
The bound $H[\kk] \lsm e_\vp$ follows from \eqref{eq:quadc0bdlvl}-\eqref{eq:quadDdtbdlvl} and $\osn \Xi = \la \geq \Xi$, $ \Xi e_v^{1/2} \leq \la e_\vp^{1/2}$ for $\osn \geq e_v^{1/2}/e_\vp^{1/2}$.

The mollification error in \eqref{eq:lastFluxTerm} is then bounded by 
\ali{
\brh[\vp^j - \vp_\ep^j] \lsm (e_v^{1/2}/e_\vp^{1/2})(e_\vp^{3/2}/\osn) \lsm e_v^{1/2} e_\vp/\osn,
}
where we have used Proposition~\ref{prop:mollErrProp} for $\vp^j$.
\end{proof}

We next consider terms of ``Type A'', which are those terms found in \eqref{eq:vpLdef},\eqref{eq:allTheVps}-\eqref{eq:lastFluxTerm} that are already in the divergence form $\nb_j \wtld{\vp}^j$ for some $\wtld{\vp}^j$ obeying the required bound $\brh[\wtld{\vp}^j] \lsm e_{\undvp}^{1/2} e_G$.
\begin{prop} \label{prop:typeAprop} The term $\vp_{[GA]}^{j}$  identified within the course of the following proof as the sum of Type A terms obeys the bound $\brh[\vp_{[GA]}] \lsm e_{\undl{\vp}}^{1/2} e_G$.
\end{prop}
The proof will make use of the inequality
\ali{
(\la^{-1} \ostu^{-1} e_\vp + \ost{b} e_\vp^{3/2}) &\lsm e_{\undvp}^{1/2} e_G, \label{eq:critConditComparevp}
}
which follows from $\ost{b} = b_0 (e_v^{1/2}/(e_\vp^{1/2} B_\la N))^{1/2}$, $\la = B_\la N \Xi$, $\ostu^{-1} = \ost{b}^{-1} (\Xi e_v^{1/2})$ and \eqref{eq:Nastdef}.  
\begin{proof}  From \eqref{eq:vpLdef}, we isolate the term $\wtld{\vp}_{\eqref{eq:vpLdef}}^j$ defined as the sum of terms of the form $O(\hat{\de} v_I v_J v_K )$ (including terms that are higher order in $\hat{\de} v_I$) or of the form $O(\de V \cdot  V \cdot V )$ (including terms that are higher order in $\de V_I$).  These terms are bounded by
\ALI{
\brh[\wtld{\vp}_{\eqref{eq:vpLdef}}^j] &\lsm \max_{I,J,K} \brh[\hat{\de} v_I] ( \brh[v_J] + \brh[\hat{\de} v_J] ) ( \brh[v_J] + \brh[\hat{\de} v_K]) \\
&+ \max_{I,J,K} \brh[\de V_I] ( \brh[V_J] + \brh[\de V_J ] )( \brh[V_K] + \brh[\de V_K] )  \\
\brh[\wtld{\vp}_{\eqref{eq:vpLdef}}^j] &\lsm \ost{b} e_\vp^{1/2} \cdot e_\vp + (B_\la \osn \Xi)^{-1} \Xi e_\vp^{1/2} \cdot e_\vp \lsm \ost{b} e_\vp^{3/2} \lsm e_{\undl{\vp}}^{1/2} e_G.
}
In the last line we used our choice of $\ost{b} = (e_v^{1/2}/(e_\vp^{1/2} B_\la \osn))^{1/2}$, $N \geq e_v^{1/2}/e_\vp^{1/2}$ and \eqref{eq:Nastdef}.

We next isolate the acceptable divergence-form terms in \eqref{eq:HHflux}-\eqref{eq:Tflux}, which we bound using \eqref{eq:boundedHighTrans} 
\ALI{
\brh[- \tilde{V}_\ell R_H^{j\ell}] + \brh[- \tilde{V}_\ell R_T^{j\ell}]  &\lsm \brh[\tilde{V}](\brh[R_H] + \brh[R_T]) \\
\brh[- \tilde{V}_\ell R_H^{j\ell}] + \brh[- \tilde{V}_\ell R_T^{j\ell}]&\lsm e_\vp^{1/2}(\ost{b} e_\vp ) \lsm e_{\undl{\vp}}^{1/2} e_G.
}
In the last line we used \eqref{eq:tildeVLevel} and \eqref{eq:critConditComparevp}.

The term $-V_\ell R_S^{j\ell}$ in \eqref{eq:stressFlux} contains various type $A$ terms.  These terms may be identified by first decomposing $R_S^{j\ell} = P \de_{[2]}^{j\ell} + R_{[2\ast]}^{j\ell} + R_{[3\ast]}^{j\ell} + R_{SG}^{j\ell}$, where $R_{SG}^{j\ell}$ contains terms of the form $O(\de v_I v_I)$, $O(\hat{\de} v_I v_I)$ along with quadratic terms in $\hat{\de} v_I$ or $\de v_I$.  Several terms in \eqref{eq:stressFlux} have already been incorporated into the $\ost{\vp}_{[2]}^j$ chosen in \eqref{eq:chosenMain2termflux}.  The remaining terms in\eqref{eq:stressFlux} will satisfy 
\ali{
\nb_j \vp_{SB}^j &= - \nb_j[V_\ell P \de_{[2\ast]}^{j\ell} + V_\ell \pi_\parallel R_{[3\ast]}^{j\ell} ] \label{eq:vpsbdef} \\
\vp_{SA}^j &= - \hat{\de} V_\ell \pi_\perp R_{[3\ast]}^{j\ell} - V_\ell R_{SG}^{j\ell}, \\
\hat{\de} V^\ell &= \sum_I (e^{i \la \xi_I} \hat{\de} v_I^\ell) + \de V_I^\ell. \label{eq:errHatAppCorrect}
}
We can bound \eqref{eq:errHatAppCorrect} by $\brh[\hat{\de} V] \lsm \ost{b} e_\vp^{1/2} + (B_\la \osn)^{-1} e_\vp^{1/2} \lsm \ost{b} e_\vp^{1/2}$, where we use that $\brh[\de v_I] \lsm \ost{b}$, while $\brh[e^{i \la \xi_I}] \lsm 1$.  (The last estimate follows from $\Ddt e^{i \la \xi_I} = 0$ and $\co{\nb_{\va} e^{i \la \xi_I}} \lsm_{\va} \la^{|\va|}$, which can be checked by replacing $(F, h_F) = (v_I, e_\vp^{1/2})$ by $(F,h_F) = (1,1)$ in the computation of line \eqref{eq:phaseBd}.)  

We now obtain using \eqref{eq:critConditComparevp} and $\ost{b} \leq (B_\la N)^{-1}$ for $N \geq e_v^{1/2}/e_\vp^{1/2}$ the bound 
\ALI{
\brh[\vp_{SA}] &\lsm \brh[ \hat{\de} V] \brh[R_{[3\ast]}] + + \brh[V]\brh[\hat{\de} v_I](\brh[v_I] + \brh[\hat{\de} v_I] ) \\
&+ \brh[V] \brh[\de v_I](\brh[v_I] + \brh[\de v_I] ) \\
\brh[\vp_{SA}] &\lsm \ost{b} e_\vp^{1/2} e_G + e_\vp^{1/2}(\ost{b} e_\vp + (B_\la N)^{-1} e_\vp ) \\
&\lsm \ost{b} e_\vp^{3/2} \lsm e_{\undl{\vp}}^{1/2} e_G.
}
This bound concludes our estimates for the type $A$, divergence form terms.
\end{proof}

\subsection{Solving the divergence equation for the unresolved flux current} \label{sec:solvingDivEqn}
We now calculate and estimate the remaining error terms in $\ost{\vp}^j $, which are all of ``Type $B$'', meaning that they require solving the divergence equation $\nb_j \tilde{\vp}^j = U$ in order to obtain an appropriate estimate.  We gather all these terms into $\vp_{[GB]}^j$, where $\ost{\vp}_{[G]}^j = \vp_{[GM]}^j + \vp_{[GA]}^j + \vp_{[GB]}^j$ is the lower order part of the new unresolved flux current $\ost{\vp}^j$, and $\vp_{[GM]}^j, \vp_{[GA]}^j$ have been defined and bounded in Propositions~\ref{prop:mollVp},\ref{prop:typeAprop}.  We require $\vp_{[GB]}^j$ to satisfy the following 
\ali{
\vp_{[GB]}^j &= \vp_{[MB]}^j + \vp_{[LH]}^j + \vp_{[HB]}^j + \vp_{[TB]}^j + \vp_{[QB]}^j + \vp_{[\kk B]}^j + \vp_{[\calT B]}^j \label{eq:lastCurrentTerms} \\
\begin{split}
\vp_{[MB]}^j &= (P - P_{\ep_v}) (V^j - V_\ell \de_{[2\ast]}^{j\ell}) - (\kk_L - \kk_{L \ep_v}) V^j \\
&\, \,- (v^j - v_\ep^j) \kk_{L} - V_\ell (\pi_{\parallel} R_{[3\ast]}^{j\ell} - R_{\parallel \ep}^{j\ell} ) \label{eq:lastMollCurrentTerm} 
\end{split} \\
\nb_j \vp_{[LH]}^j &= \nb_j[(P_{\ep_v} + \kk_{L \ep_v})V^j] \label{eq:lowHighDiv} \\
\nb_j \vp_{[PR]}^j &= -\nb_j[ V_\ell P_{\ep_v} \de_{[2\ast]}^{j\ell}] - \nb_j[V_\ell R_{\parallel \ep}^{j\ell} ] \label{eq:VRDiv} \\
\nb_j \vp_{[HB]}^j &= \nb_j v_{\ep \ell} \Big( \sum_{J \neq \bar{I}} (V_I^j V_J^\ell) - R_H^{j\ell} - R_T^{j\ell} \Big) \label{eq:vpHBterms} \\
\nb_j \vp_{[TB]}^j &= V_\ell(\nb_j R_{\ep_v}^{j\ell} + Q_\ep^{\ell}[v,v]), \quad Q_\ep^{\ell}[v,v] = v_\ep^a \nb_a v_\ep^\ell - \chi_\ep \ast \chi_\ep \ast(v^a \nb_a v^\ell) \label{eq:vpQBterms} \\
\nb_j \vp_{[\kk B]}^j &= \sum_{J \neq \bar{I}} \Ddt \left( \fr{V_I \cdot V_J}{2} \right) \label{eq:vpTBterms} \\
\nb_j \vp_{[\calT B]}^j &= \nb_j \Big[ \sum_{I,J,K \notin \cal T} (V_I)_\ell V_J^j V_K^\ell \Big].  \label{eq:trilinDiv}
}
Here $R_{\parallel \ep}^{j\ell} = \chi_\ep \ast \chi_\ep \ast \pi_\parallel R_{[3\ast]}^{j\ell}$ is the mollification of $\pi_\parallel R_{[3\ast]}^{j\ell}$ in space using the frequency-localized kernel $\chi_\ep$ that was also used to define $v_\ep$.  The reason we have mollified each of $P, \kk_L$ and $\pi_\parallel R_{[3\ast]}^{j\ell}$ is to arrange that all of the equations \eqref{eq:lowHighDiv}-\eqref{eq:trilinDiv} have frequency support on $|p| \geq \la/3$:
\ali{
\Fsupp \eqref{eq:lowHighDiv}{-}\eqref{eq:trilinDiv} &\subseteq \{ |p| \geq (B_\la N \Xi)/3 \} \subseteq \widehat{\R^3}. \label{eq:freqLocalized}
}
This observation uses the fact that $\chi_\ep \ast$ localizes to frequencies of size $|p| \leq N^{-1/L} \Xi \leq 10^{-2} B_\la N \Xi$ while each $V_I$ is localized to frequency $ 2^{-1} n_I B_\la N \Xi \leq |p| \leq 2 n_I B_\la N \Xi$ from the presence of $\PP_I$ in \eqref{eq:HfreqProjectiondef}, and also uses that the Fourier transform maps products to convolutions, and the various cascade conditions \eqref{eq:noWrongCascades},\eqref{eq:goodHighCascades} that constrain the $n_I$ to ensure the above products of high frequency terms all have frequency support in $\{ |p| \geq B_\la N \Xi \}$.

Similarly to Section~\ref{sec:boundNewStress}, we take advantage of this frequency localization as follows.  Choose an order $-1$ right inverse to the (scalar) divergence equation $Q^j \ast$ and choose a Fourier multiplier $P_{\approx \la} = \chi_{\approx \la} \ast$ whose symbol is a rescaled Schwartz function of the form $\widehat{\chi_{\approx \la}}(p) = \widehat{\chi_{\approx 1}}(p/\la)$ with $\widehat{\chi_{\approx \la}}(p) = 1$ for all $p$ in \eqref{eq:freqLocalized}.  Let our solution to $\nb_j \wtld{\vp}^j = U$ be defined to be
\ali{
\wtld{\vp}^j &= P_{\approx \la} Q^j \ast U = Q_{\approx \la}^j \ast U. \label{eq:defineSolnVp}
}
For example, if we use the standard solution operator given by the Fourier multiplier $Q^j \ast$ with symbol
\ali{
\widehat{Q}^j(p) &= (-i) |p|^{-2} p^j,
}
then \eqref{eq:defineSolnVp} solves $\nb_j \wtld{\vp}^j = U$ whenever $\Fsupp U$ is contained in \eqref{eq:freqLocalized}.  Furthermore, the kernel of $Q_{\approx \la}^j \ast$ is a rescaled Schwartz function with Fourier transform $\la^{-1} \widehat{\chi_{\approx 1}}(p/\la) \widehat{Q}^j(p/\la)$, which has the form $Q_{\approx \la}^j(h) = \la^{-1} [ \la^{3} Q_{\approx 1}^j(\la h) ]$ in physical space for the Schwartz function $Q_{\approx 1}^j$.  We may thus apply Lemma~\ref{lem:commutator} with $\La^{-1} \lsm \la^{-1}$ to the operator in \eqref{eq:defineSolnVp}.  Doing so gives the following estimates.
\begin{prop} \label{prop:divEqnPropCurrent}  Every term in \eqref{eq:lastCurrentTerms} including $\vp_{[GB]}^j$ satisfies the estimate $\brh[\vp_{[GB]}^j ] \lsm e_{\undvp}^{1/2} e_G$.
\end{prop}
\begin{proof}  We start by estimating the term in \eqref{eq:lastMollCurrentTerm} using Lemmas~\ref{lem:comparisonLem} and Proposition~\ref{prop:mollErrProp}.  (Note that\footnote{Here we use the calculation of line \eqref{eq:PerrorMoll} to deduce the bounds for $D_t$ from those of $\Ddt$.} Proposition~\ref{prop:prelimBds} and therefore Proposition~\ref{prop:mollErrProp} apply to $F = \pi_\parallel R_{[3\ast]}^{j\ell} = \pi_{\parallel} \pi_{[3]} R_{[G\ep]}^{j\ell}$ with $h_F = e_G$.)
\ali{
\brh[\vp_{[MB]}^j] &\lsm \brh[V] \left( \brh[P - P_{\ep_v}] + \brh[\kk_L - \kk_{L \ep_v}] + \brh\left[\pi_\parallel R_{[3\ast]} - R_{\parallel \ep}^{j\ell} \right] \right) + \brh[v-v_\ep] \brh[\kk_L]  \\
\brh[\vp_{[MB]}^j] &\stackrel{\eqref{eq:spaceMollErrorBds}-\eqref{eq:mollErrkkLbd}}{\lsm} e_\vp^{1/2}( (e_{\undvp} / N) + (e_G/N) ) + (e_v^{1/2}/N) e_{\undvp} \lsm e_v^{1/2} e_{\undvp}/N \stackrel{\eqref{eq:Nastdef}}{\lsm} e_{\undvp}^{1/2} e_G.
}
As for the remaining terms in \eqref{eq:lowHighDiv}-\eqref{eq:trilinDiv}, to apply Lemma~\ref{lem:commutator} we need only estimates on $H_0[U] := \max \{ \co{U}, (\la e_\vp^{1/2})^{-1} \co{\Ddt U} \}$, where $U$ is the right hand side of the equation.  We denote the right hand sides of equations \eqref{eq:lowHighDiv}-\eqref{eq:trilinDiv} by $U_{[LH]}, U_{[HB]}, \ldots, U_{[\calT B]}$.  The bounds for \eqref{eq:vpHBterms} and \eqref{eq:vpTBterms} are
\ali{
\brh[U_{[\kk B]}] &\lsm \max_{I, J} \brh[\Ddt V_I] \brh[V_J] \lsm \ostu^{-1} e_\vp \label{eq:usedPropDdtbrh} \\
\brh[\vp_{[\kk B]}] &\stackrel{\eqref{eq:solvedDivBd}}{\lsm} \la^{-1} \ostu^{-1} e_\vp \stackrel{\eqref{eq:critConditComparevp}}{\lsm} e_{\undvp}^{1/2} e_G. \label{eq:ultimvpTBbd} \\
\co{U_{[HB]}} &\lsm \co{\nb v_\ep}( \max_{I, J} \co{V_I} \co{V_J} + \co{R_H} + \co{R_T} ) \stackrel{\eqref{eq:highFreqCorrect},\eqref{eq:boundedHighTrans}}{\lsm} (\Xi e_v^{1/2}) e_\vp \notag \\
\co{\Ddt U_{[HB]}} \lsm& \, \, \co{\Ddt(\nb v_\ep) } e_\vp + (\Xi e_v^{1/2}) (\max_{I,J} \co{\Ddt V_I} \co{V_J} + \co{\Ddt R_H} + \co{\Ddt R_T} ) \notag \\
\stackrel{\eqref{eq:velocBd3} ,\eqref{eq:highFreqCorrect},\eqref{eq:boundedHighTrans}}{\lsm}& (\Xi e_v^{1/2})^2 e_\vp + (\Xi e_v^{1/2})( \ostu^{-1} e_\vp + \la e_\vp^{1/2} \ost{b} e_\vp ) \lsm \la e_\vp^{1/2} (\Xi e_v^{1/2}) e_\vp \notag \\
\brh[\vp_{[HB]}] &\lsm \la^{-1} (\Xi e_v^{1/2}) e_\vp \lsm \la^{-1} \ostu^{-1} e_\vp \stackrel{\eqref{eq:critConditComparevp}}{\lsm} e_{\undvp}^{1/2} e_G. \label{eq:ultimvpHBbd}
}
In line \eqref{eq:usedPropDdtbrh} we used Proposition~\ref{prop:allTheEndBounds}, which relies the bounds for $\Ddt^2 V_I$ in \eqref{eq:highFreqCorrect}.  In lines \eqref{eq:ultimvpTBbd} and \eqref{eq:ultimvpHBbd} we applied 
Lemma~\ref{lem:commutator} with $\La^{-1} \lsm \la^{-1}$.

We now consider \eqref{eq:lowHighDiv}.  Using the divergence free property $\nb_j V^j = 0$, this term becomes
\ALI{
U_{[LH]} &= V^j(\nb_jP_{\ep_v} + \nb_j \kk_{L \ep_v}), \\
\Ddt U_{[LH]} &= \Ddt V^j (\nb_jP_{\ep_v} + \nb_j \kk_{L \ep_v}) + V^j(\nb_j \Ddt P_{\ep_v} + \nb_j \Ddt \kk_{L \ep_v})
\\
&- \nb_j v_\ep^a V^j(\nb_aP_{\ep_v} + \nb_a \kk_{L \ep_v}), 
}
which leads to the following estimates (using Proposition~\ref{prop:prelimBds} for $F \in  \{P, \kk_L\}$ and $h_F = e_{\undvp}$)
\ALI{
\co{U_{[LH]}} &\lsm \co{V}(\co{\nb P_{\ep_v}} + \co{\nb \kk_{L \ep_v}}) \lsm e_\vp^{1/2}(\Xi e_{\undvp}) \\
\co{\Ddt U_{[LH]}} &\lsm \sum_{r_1+r_2 = 1} \co{\Ddt^{r_1} V^a} (\co{\Ddt^{r_2} \nb_a P_{\ep_v} } + \co{\Ddt^{r_2} \nb_a \kk_{L\ep_v} } ) \\
\co{\Ddt U_{[LH]}} &\lsm \ostu^{-1} e_\vp^{1/2}(\Xi e_{\undvp}) \\
\brh[ \vp_{[LH]} ] &\lsm \la^{-1} (\Xi e_\vp^{1/2}) e_{\undvp} \lsm N^{-1} e_\vp^{1/2} e_\vp \stackrel{\eqref{eq:Nastdef}}{\lsm} e_{\undvp}^{1/2} e_G.
}

We turn now to \eqref{eq:vpQBterms}.  We will use the following bounds for the commutator term:
\ali{
\co{Q_\ep^\ell[v,v]} \lsm \Xi e_v/N, \quad \co{\Ddt Q_\ep^\ell[v,v]} \lsm \Xi^2 e_v e_\vp^{1/2}. \label{eq:commutermBounds}
}
The $C^0$ estimate in \eqref{eq:commutermBounds} was proved already by taking $(F, e_V) = (v, e_v^{1/2})$ in lines \eqref{eq:simplecommuterm}, \eqref{eq:commuterm2}.  We will establish the bound for $\Ddt Q_\ep^\ell$ in Section~\ref{sec:commutadvec}.  For now we proceed assuming \eqref{eq:commutermBounds} to obtain
\ALI{
\co{U_{[TB]}} &\lsm \co{V}(\co{\nb R_{\ep_v}} + \co{Q_\ep}) \lsm e_\vp^{1/2}( \Xi e_\vp + \Xi (e_v/N)) \\
\co{\Ddt U_{[TB]}} &\lsm \sum_{r_1 + r_2 = 1} \co{\Ddt^{r_1} V}(\co{\Ddt^{r_2} \nb R_{\ep_v}} + \co{\Ddt^{r_2} Q_\ep^\ell[v,v]} ) \\
&\lsm \ostu^{-1} (e_\vp^{1/2} \Xi e_\vp + \Xi (e_v/N) ) + e_\vp^{1/2} (\Xi^2 e_v e_\vp^{1/2} ) \lsm [\la e_\vp^{1/2}] e_\vp^{1/2}(\Xi e_\vp + \Xi (e_v/N) ) \\
\brh[\vp_{[TB]}] &\lsm \la^{-1} e_\vp^{1/2}( \Xi e_\vp + \Xi (e_v/N)) \lsm N^{-1} e_\vp^{3/2} \stackrel{\eqref{eq:Nastdef}}{\lsm} e_{\undvp}^{1/2} e_G.
}
We now consider \eqref{eq:trilinDiv}.  For this term, the key cancellation arises by using that $\nb_j V_J^j = 0$ and the calculation in line \eqref{eq:keyCancelTerm}
\ali{
U_{[\calT B]} &= \sum_{I,J,K \notin \cal T} V_J^j \nb_j (V_I \cdot V_K) = \sum_{I,J,K \notin \cal T} [ \mathring{U}_{[IJK\calT]} + \de U_{[IJK\calT]} ]  \notag \\
\begin{split} \label{eq:trilinDivUmain} 
\mathring{U}_{[IJK\calT]} &= (i\la) e^{i \la \xi_J} v_J^j (\nb_j \xi_I + \nb_j \xi_K) (v_I \cdot v_K) + e^{i \la \xi_J} v_J^j \nb_j(v_I \cdot v_K) 
\end{split} \\
\de U_{[IJK\calT]} &= \de V_J^j \nb_j(V_I \cdot V_K) + \mathring{V}_J^j \nb_j(\de V_I \cdot V_K) + \mathring{V}_J^j \nb_j( \mathring{V}_I \cdot \de V_K) \label{eq:lowOrderTermIJKcT}
}
For the lower order term \eqref{eq:lowOrderTermIJKcT} we obtain the following estimates using the formula $\Ddt \nb_j F = \nb_j \Ddt F - \nb_j v_\ep^a \nb_ a F$ and the inequality $\max \{ \co{F}, \la^{-1} \co{\nb F} \} \leq \brh[F]$
\ali{
\co{\de U_{[IJK\calT]}} &\lsm \la \max_{I,J,K} \brh[\de V_I]( \brh[V_J] + \brh[\de V_J] )(\brh[V_K] + \brh[\de V_K]) \notag \\
&\lsm \la N^{-1} e_\vp^{1/2} \cdot e_\vp^{1/2} \cdot e_\vp^{1/2} \lsm N^{-1} \la e_\vp^{3/2} \label{eq:codeTrilinterm} \\
\co{\Ddt \de U_{[IJK\calT]}} &\lsm \sum_{r_1 + r_2 = 1} \la \max_{I,J,K} \brh[\Ddt^{r_1} \de V_I]( \brh[\Ddt^{r_2} V_J] + \brh[\Ddt^{r_2} \de V_J] )(\brh[\nb V_K] + \brh[\nb \de V_K]) \notag \\
&+ \co{\nb v_\ep} \la \max_{I,J,K} \brh[\de V_I]( \brh[V_J] + \brh[\de V_J] )(\brh[V_K] + \brh[\de V_K]) \notag \\
\co{\Ddt \de U_{[IJK\calT]}} &\stackrel{\eqref{eq:highFreqCorrect}-\eqref{eq:lowerTermHFBd}}{\lsm} \la \ostu^{-1} N^{-1} e_\vp^{3/2} \lsm \ostu^{-1}(\la \ost{b} e_\vp^{3/2}) \label{eq:coDdtdeTrlinterm}
}
For the leading order term \eqref{eq:trilinDivUmain}, we use that $\hat{v}_J^j \nb_j \hxii_I = \hat{v}_J^j \nb_j \hxii_K = 0$ to obtain
\ali{
\begin{split} \label{eq:keyCancelTerm}
\mathring{U}_{[IJK\calT]} &= (i\la) e^{i \la \xi_J} v_J^j (\hat{\de} \nb_j \xi_I + \hat{\de}\nb_j \xi_K) (v_I \cdot v_K)  \\
&\, \,+   (i\la) e^{i \la \xi_J} \hat{\de} v_J^j (\nb_j \hxii_I + \nb_j \hxii_K) (v_I \cdot v_K) + e^{i \la \xi_J} v_J^j \nb_j(v_I \cdot v_K) 
\end{split} \\
\co{\mathring{U}_{[IJK\calT]}} &\lsm \la \max_{I,J} ( \co{v_I}^3 \co{\hat{\de} \nb \xi_J} + \co{v_I}^2\co{\hat{\de} v_J} + \la^{-1} \co{v_I}^2\co{\nb v_J}) \notag \\
\co{\mathring{U}_{[IJK\calT]}} &\lsm \la (\ost{b} e_\vp^3 + N^{-1} e_\vp^3 ) \lsm \la \ost{b} e_\vp^3 \label{eq:c0maintrilin} \\
\co{\Ddt \mathring{U}_{[IJK\calT]}} &\lsm \sum_{r_1 + r_2 = 1} \max_{I,J,K} \left\{\la \co{\Ddt^{r_1} v_I} \co{\Ddt^{r_2} \nb \xi_J} \co{v_K}^2 + \co{\Ddt^{r_1} v_I} \co{\Ddt^{r_2} \nb v_J} \co{v_K} \right\} \notag  \\
\co{\Ddt \mathring{U}_{[IJK\calT]}} &\lsm (\Xi e_v^{1/2}) \la e_\vp^{3/2} \lsm \ostu^{-1} (\la \ost{b} e_\vp^3 ). \label{eq:DdtBdMainTrilin}
}
To obtain \eqref{eq:DdtBdMainTrilin} we applied $\Ddt$ to the formula \eqref{eq:trilinDivUmain}  and used  Proposition~\ref{prop:correctAmpBds} and that $\ost{b}^{-1} \Xi e_v^{1/2} = \ostu^{-1}$.  Applying Lemma~\ref{lem:commutator} with $\La^{-1} \lsm \la^{-1}$ and using \eqref{eq:trilinDivUmain}-\eqref{eq:DdtBdMainTrilin} we obtain
\ali{
\brh[\vp_{[\calT B]}] &\lsm \ost{b} e_\vp^{3/2} \stackrel{\eqref{eq:critConditComparevp}}{\lsm} e_{\undvp}^{1/2} e_G. \label{eq:ultimvpCaltbd}
}
The last term we must bound is the solution to \eqref{eq:VRDiv}.  For this term, we rely on the cancellation explained in lines \eqref{eq:amazingCancel} and \eqref{eq:amazingCancelParallel} (replacing $P \mapsto P_{\ep_v}$ and $\pi_\parallel R_{[3\ast]}^{j\ell} \mapsto R_{\parallel \ep}^{j\ell}$ by their frequency-localized versions and noting that $R_{\parallel \ep}^{j\ell}$ still takes values in the span of $\langle e_1 \otimes e_1, e_2 \otimes e_2 \rangle$).  We decompose
\ali{
U_{[PR]} &= - U_{[PR1]} - U_{[PR2]} \notag \\
 U_{[PR1]} = \nb_j V_\ell(P_{\ep_v} \de_{[2\ast]}^{j\ell} + R_{\parallel \ep}^{j\ell}), \quad &U_{[PR2]} = V_\ell (\nb_j P_{\ep_v}\de_{[2\ast]}^{j\ell} + \nb_j R_{\parallel \ep}^{j\ell}).  \label{eq:UPRexpand1}
}
We bound the lower order term using the bounds of Proposition~\ref{prop:prelimBds} applied to $P, \pi_\parallel R_{[3\ast]}$
\ali{
\co{U_{[PR2]}} &\lsm \co{V}(\co{\nb P_{\ep_v}} + \co{\nb R_{\parallel \ep}} )  \lsm e_\vp^{1/2}(\Xi e_{\undvp}) \label{eq:c0UPR2} 
}
We then expand the main term in \eqref{eq:UPRexpand1} as $U_{[PR1]} =  \sum_I \mathring{U}_{[PRI]} + \de U_{[PRI]}$ with
\ali{
\mathring{U}_{[PRI]} &= (i \la) e^{i \la \xi_I} \nb_j \xi_I v_{I \ell}(P_{\ep_v} \de_{[2\ast]}^{j\ell} + R_{\parallel \ep}^{j\ell} ), \qquad 
\de U_{[PRI]} = \nb_j \de V_{I \ell} (P_{\ep_v} \de_{[2\ast]}^{j\ell} + R_{\parallel \ep}^{j\ell} )
}
We bound the lower order term using
\ali{
\co{\de U_{[PRI]}} &\lsm \co{\nb \de V_I}(\co{P_{\ep_v}} + \co{R_{\parallel \ep}}) \lsm \la [\la^{-1} \Xi e_\vp^{1/2}] e_{\undvp} \label{eq:c0UPRI}
}
We treat the main term using the observations of lines \eqref{eq:amazingCancel} and \eqref{eq:amazingCancelParallel} to find that
\ali{
\mathring{U}_{[PRI]}&= (i \la) e^{i \la \xi_I} (\nb_j \hxii_I (\hat{\de} v_I)_\ell + \hat{\de} \nb_j \xi_I v_{I\ell})(P \de_{[2\ast]}^{j\ell} + R_{\parallel \ep}^{j\ell} )  \notag \\
\co{\mathring{U}_{[PRI]}} &\lsm \la (\co{\hat{\de} v_I} + \co{\hat{\de} \nb \xi_I} \co{v_I})(\co{P_{\ep_v}} + \co{R_{\parallel \ep}} ) \lsm \la [\ost{b} e_\vp^{1/2} ] e_{\undvp}. \label{eq:mainUPRIc0}
}
Combining \eqref{eq:c0UPR2},\eqref{eq:c0UPRI}, \eqref{eq:mainUPRIc0} with $\Xi \leq \la \ost{b}$ for $N \geq (e_v^{1/2}/e_\vp^{1/2}), \ost{b} \approx (e_v^{1/2}/e_\vp^{1/2} N)^{1/2}$, we obtain
\ali{
\co{U_{[PR]}} &\lsm \Xi e_\vp^{1/2} e_{\undvp} + \la \ost{b} e_\vp^{1/2} e_{\undvp} \lsm \la [\ost{b} e_\vp^{3/2}]. \label{eq:c0UPRbd}
}
To bound the advective derivative of $U_{[PR]}$, we introduce the weighted $C^1$ norm 
\ali{
H_1[F] := \max \{ \co{F}, \la^{-1} \co{\nb F} \},
} 
noting the properties $\co{\nb F} \leq \la H_1[F]$, $H_1[F + G] \leq H_1[F] + H_1[G]$, $H_1[F G] \lsm H_1[F] H_1[G]$.  

Using these properties and commuting $\nb_j$ with $\Ddt$, we obtain 
\ali{
\Ddt U_{[PR]} &= \nb_j[ \Ddt[ V_\ell(P_{\ep_v} \de_{[2\ast]}^{j\ell} + R_{\parallel \ep}^{j\ell} ) ] - \nb_j v_\ep^a \nb_a[V_\ell(P_{\ep_v} \de_{[2\ast]}^{j\ell} + R_{\parallel \ep}^{j\ell} )] \notag \\
\co{\Ddt U_{[PR]}} &\lsm \sum_{r_1 + r_2 = 1} \la H_1[\Ddt^{r_1} V] \left(H_1[\Ddt^{r_2} P_{\ep_v}] + H_1[\Ddt^{r_2} R_{\parallel \ep}] \right) \notag \\
&+ \co{\nb v_\ep} \cdot \la H_1[V] (H_1[P_{\ep_v} ] + H_1[R_{\parallel \ep}]) \notag \\
\co{\Ddt U_{[PR]}} &\lsm \sum_{r_1 + r_2 = 1}\la [\ostu^{-r_1} e_\vp^{1/2}][\ostu^{-r_2} e_{\undvp}] + \la (\Xi e_v^{1/2}) e_\vp^{1/2} e_{\undvp}  \notag \\
\co{\Ddt U_{[PR]}} &\lsm \la \ostu^{-1} e_\vp^{1/2} e_{\undvp} \approx [\la e_\vp^{1/2}][\ostu^{-1} e_\vp ], \qquad  \label{eq:c0DdtUPRbd}
}
Using \eqref{eq:c0UPRbd},\eqref{eq:c0DdtUPRbd} and applying Lemma~\ref{lem:commutator} with $\La^{-1} \lsm \la^{-1}$, we obtain
\ali{
\brh[\vp_{[PR]}] &\lsm \la^{-1}[\la \ost{b} e_\vp^{3/2}  + \ostu^{-1} e_\vp] \stackrel{\eqref{eq:critConditComparevp}}{\lsm} e_\vp^{1/2}e_{\undvp}.
}
This estimate concludes the proof of Proposition~\ref{eq:lastCurrentTerms}.
\end{proof}

\subsection{A commutator estimate for the advective derivative} \label{sec:commutadvec}
In this Section we establish a general commutator estimate to establish inequality \eqref{eq:commutermBounds}, which states that $\co{\Ddt Q_\ep^\ell[v,v]} \lsm \Xi^2 e_v e_\vp^{1/2}$.  This bound follows by taking $(F, h_F)$ to be $(v, e_v^{1/2})$ in Proposition~\ref{eq:advecCommutEstimate} below and noting that $N \geq (e_v/e_\vp)$ by \eqref{eq:Nastdef}.  The proposition will involve the weighted semi-norm
\ali{
\undl{H}_{\diamond}^R[F] &= \max_{0 \leq r \leq R} \max_{0 < r+ |\va| \leq L} \fr{\co{\nb_{\va} D_t^r F}}{\Xi^{|\va|} (\Xi e_v^{1/2})^r}. \label{eq:homogDmdNorm}
}
It will later be useful to note the comparison inequality
\ali{
\undl{H}[F] &\leq (e_v/e_\vp)^{-1} \undl{\widehat{H}}_{\diamond}[F], \label{eq:dmdCompare}
}
which uses that $\la \geq N \Xi$ and that $N \geq (e_v/e_\vp)^{3/2}$ from~\eqref{eq:Nastdef} implies $(\la e_\vp^{1/2})^{-1} \leq (e_v/e_\vp)^{-1} (\Xi e_v^{1/2})^{-1}$.

As in Section~\ref{sec:regPrelimBds}, we let $\chi_{\ep_x}$ be a mollifier at length scale $\ep_x := c_0 N^{-1/L} \Xi^{-1}$, $c_0 \leq 1$, that satisfies the vanishing moment condition $\int_{\R^3} h^{\va} \chi_{\ep_x}(h) dh = 0$ for $1 \leq |\va| \leq L$.  We define $\dmd{\chi}_\ep = \chi_{\ep_x} \ast \chi_{\ep_x}$ and note that the same vanishing moment condition holds for $\dmd{\chi}_\ep$, as can be shown using the formula $\int f(h) \dmd{\chi}_\ep(h) dh = \iint f(y{+}z) \chi_{\ep_x}(y) \chi_{\ep_x}(z) dy dz$.
We will use only the case where $\dmd{\chi}_\ep$ is exactly the mollifier used in Section~\ref{sec:regPrelimBds} to define $v_\ep$ and $c_0 = c_1$ is the parameter specified in that section.
\begin{prop} \label{eq:advecCommutEstimate} If $L \geq 2$, $\undl{H}_{\diamond}^1[F] \leq h_F$ and $\dmd{\chi}_\ep = \chi_{\ep_x} \ast \chi_{\ep_x}$ is as above, then 
\ali{
\co{\Ddt Q_{\ep}[v, F] } &\lsm N^{-1 + \fr{1}{L}} (\Xi e_v^{1/2})^2h_F, \quad Q_\ep[v,F] = v_\ep^a \nb_a (\dmd{\chi}_\ep \ast G) - \dmd{\chi}_\ep \ast[v^a \nb_a F]
}
\end{prop}
We remark that the implied constant depends only on the particular Schwartz function that is rescaled to define the mollifying kernels and will be independent of the choice of $c_0, c_1 \leq 1$ in the mollifying parameters since the proof will avoid having any losses due to differentiating the mollifiers.
\begin{proof}  We decompose $Q_\ep$ and its advective derivative into
\ALI{
Q_\ep[F] &= [v_\ep^a \nb_a, \dmd{\chi}_\ep \ast](F) + \dmd{\chi}_\ep \ast[ (v_\ep^a - v^a) \nb_a F ] \\
\Ddt Q_\ep[F] &= \left(\Ddt[v_\ep^a \nb_a, \dmd{\chi}_\ep \ast](F)\right) + \left([\Ddt, \dmd{\chi}_\ep \ast][(v_\ep^a - v^a) \nb_a F]\right) +  \left(\dmd{\chi}_\ep \ast[ \Ddt[ (v_\ep^a - v^a) \nb_a F] ]\right) \\
\Ddt Q_\ep[F] &:= (T_1) + (T_2) + (T_3).
}
We first consider the term $T_3$, for which we obtain the following bounds
\ali{
\Ddt[\nb_a F] &= \nb_a \Ddt F - \nb_a v_\ep^b \nb_ b F , \qquad
\co{\Ddt[\nb_a F]} \lsm (\Xi e_v^{1/2})(\Xi h_F) \notag \\
\co{(v - v_\ep)} &\leq N^{-1} e_v^{1/2}, \qquad
\co{\Ddt(v - v_\ep)} \lsm N^{-1 + 1/L} \Xi e_v^{1/2} \cdot e_v^{1/2} \label{eq:Ddtvmvepused}\\
T_3 \lsm \co{\Ddt[ &(v_\ep^a - v^a) \nb_a F]} \lsm N^{-1 + 1/L} (\Xi e_v^{1/2})^2 h_F.
}
Here we obtained \eqref{eq:Ddtvmvepused} by taking $(F, h_F)$ to be $(v, e_v^{1/2})$ in lines \eqref{eq:DdtMollerExpress}-\eqref{eq:commuterm2} and recalling \eqref{eq:vMinVep}.

We record the estimate $\co{G} \leq N^{-1} (\Xi e_v^{1/2}) h_F$ satisfied by $G := (v_\ep^a - v^a) \nb_a F$.  Our desired bound for $T_2$ can be achieved from this estimate by integrating by parts in the $h$ variable as follows
\ali{
- T_2 = &- [\Ddt, \dmd{\chi}_\ep \ast] G = - [v_\ep^a \nb_a, \dmd{\chi}_\ep \ast] G = \int_{\R^3} (v_\ep^a(x + h) - v_\ep^a(x) ) \nb_a G(x+h) \dmd{\chi}_\ep(h) dh \label{eq:firstCommuform} \\
&= - \int_{\R^3} (v_\ep^a(x + h) - v_\ep^a(x) ) G(x+h) \nb_a \dmd{\chi}_\ep(h) dh  \notag \\
&= - \int_0^1 \int_{\R^3} \nb_b v_\ep^a(x + \si h) G(x+h) h^b \nb_a\dmd{\chi}_\ep(h) dh d\si \notag \\
\co{T_2} &\lsm \co{\nb v_\ep} \co{G} \| |h| |\nb \dmd{\chi}_\ep| \|_{L^1(\R^3)} \lsm (\Xi e_v^{1/2})(N^{-1} \Xi e_v^{1/2} h_F) \cdot 1 \lsm N^{-1} (\Xi e_v^{1/2})^2 h_F. \label{eq:T2commutBounded}
}
Here we used that $\nb_a v_\ep^a = 0$, although the same estimate would follow without this cancellation.

Using the notation $[\Ddt,]L := \Ddt L - L \Ddt$ to denote the commutator operation, we expand $T_1$ as
\ali{
T_1 &= [\Ddt,] [v_\ep^a \nb_a, \dmd{\chi}_\ep \ast](F) + [v_\ep^a \nb_a, \dmd{\chi}_\ep \ast](\Ddt F) = T_{1a} + T_{1b} + T_{2} \label{eq:T1expand1} \\
T_{1a} &= [\Ddt,] [v_\ep^a \nb_a, \dmd{\chi}_\ep \ast](F), \qquad T_{1b} = [v_\ep^a \nb_a, \dmd{\chi}_\ep \ast](D_t F).
}
To treat $T_{1b}$, we replace $F$ in the derivation of inequality \eqref{eq:commuterm2} by $U = D_t F$ to express $T_{1b}$ in the form
\ALI{
\sum_{\va_1, \va_2}  \fr{c_{\va_1, \va_2}}{(L {-} 2)!}\iiint_{\R^3} \nb_{\va_1}\nb_bv^a(x + \tau \si h) [\nb_{\va_2} \nb_a U](x + \si h) h^b h^{\va_1} h^{\va_2} \dmd{\chi}_\ep(h) \tx{d} h \, \si(1{-}\si)^{L-2} \tau^{|\va_1|} \tx{d}\tau \tx{d}\si,
}
where the sum runs over multi-indices with $|\va_1| + |\va_2| = L-1$ and the bounds of integration for $\tx{d}\tau$ and $\tx{d}\si$ are from $0$ to $1$.  

The number of derivatives of $U$ appearing when $|\va_2| + 1 = L$ and $|\va_1| = 0$ exceeds the number we control, but we can observe that $\si[\nb_{\va_2} \nb_a U](x + \si h) = {}^{(h)} \nb_a [\nb_{\va_2} U(x + \si h)]$.  Integrating by parts in $h$ and using that $\nb_a v^a = 0$, we express the integrals appearing in such terms in the form
\ALI{
- \iiint_{\R^3} [\nb_{\va_2} U](x + \si h) \nb_bv^a(x + \tau \si h) \nb_a[h^b h^{\va_2} \dmd{\chi}_\ep(h)] \tx{d} h \, (1{-}\si)^{L-2} \tau^{|\va_1|} \tx{d}\tau \tx{d}\si.
}
From the above two formulas, we obtain a bound of 
\ali{
\begin{split} \notag 
\co{T_{1b}} &\lsm \sum_{|\va_1| + |\va_2| = L-1} [\Xi^{|\va_1|+1} e_v^{1/2}][\Xi^{|\va_2| +2} e_v^{1/2} h_F]\| |h|^L \dmd{\chi}_\ep \|_{L^1}  \\
&+ [\Xi^{L+1} e_v^{1/2} h_F][\Xi e_v^{1/2}](\| |h|^{L-1} \dmd{\chi}_\ep \|_{L^1} + \| |h|^{L} \nb \dmd{\chi}_\ep \|_{L^1} )  
\end{split} \\
\co{T_{1b}} &\lsm(\ep_x^L \Xi^L + \ep_x^{L-1} \Xi^{L-1}) (\Xi e_v^{1/2})^2 h_F \lsm N^{-1+\fr{1}{L}} (\Xi e_v^{1/2})^2 h_F. \label{eq:T1bbounding2}
}

The last term to consider is the term $T_{1a}$, which we decompose as $T_{1a} =  -T_{11} - T_{12}$, where
\ali{
T_{11} &= \int_{\R^3} (\Ddt v_\ep^i(x + h) - \Ddt v_\ep^i(x)) \nb_i F(x+h) \dmd{\chi}_\ep(h) dh \\
T_{12} &= \int_{\R^3} (v_\ep^a(x + h) - v_\ep^a(x))\nb_a v_\ep^i(x+h) \nb_i F(x+h) \dmd{\chi}_\ep(h) dh.
}
The term $T_{11}$ can be rewritten by Taylor expanding as in \eqref{eq:GtaylorExpand} to have the form
\ALI{
\sum_{\va_1, \va_2}  \fr{c_{\va_1, \va_2}}{(L {-} 2)!}\iiint_{\R^3} \nb_{\va_1}\nb_b\Ddt v_\ep^a(x + \tau \si h) [\nb_{\va_2} \nb_a F](x + \si h) h^b h^{\va_1} h^{\va_2} \chi_\ep(h) \tx{d} h \, \si(1{-}\si)^{L-2} \tau^{|\va_1|} \tx{d}\tau \tx{d}\si,
}
where the sum runs over multi-indices with $|\va_1| + |\va_2| = L-1$ and the bounds of integration for $\tx{d}\tau$ and $\tx{d}\si$ are from $0$ to $1$.  In the cases where $|\va_1| = L-1$, we observe that $\tau \si \nb_{\va_1}\nb_b\Ddt v_\ep^a(x + \tau \si h) = {}^{(h)} \nb_b[\nb_{\va_1}\Ddt v_\ep^a(x + \tau \si h) ]$ and integrate by parts in $h$ to write the integrals in such terms in the form
\ALI{
- \iiint_{\R^3} \nb_{\va_1}\Ddt v_\ep^a(x + \tau \si h){}^{(h)}\nb_b\left[ [\nb_a F](x + \si h) h^b h^{\va_1} \dmd{\chi}_\ep(h) \right] \tx{d} h \, (1{-}\si)^{L-2} \tau^{L-2} \tx{d}\tau \tx{d}\si.
}
The bounds we deduce from the above two formulas are the same as those in \eqref{eq:T1bbounding2}, thus yielding the desired bound of $\co{T_{11}} \lsm N^{-1+\fr{1}{L}}(\Xi e_v^{1/2})^2 h_F$ for this term. 

We may similarly Taylor expand the term $T_{12}$ as a finite linear combination of terms of the form
\ALI{
\iiint_{\R^3} \nb_{\va_1}\nb_bv_\ep^a(x + \tau \si h) [\nb_{\va_2}(\nb_a v_\ep^i \nb_i F)](x + \si h) h^b h^{\va_1} h^{\va_2} \chi_\ep(h) \tx{d} h \, \si(1{-}\si)^{L-2} \tau^{|\va_1|} \tx{d}\tau \tx{d}\si,
}
where the multi-indices satisfy $|\va_1| + |\va_2| = L-1$ and the bounds of integration for $\tx{d}\tau$ and $\tx{d}\si$ are from $0$ to $1$.  In this case, we do not have any terms involving derivatives beyond order $L$, and we obtain $\co{T_{12}} \lsm \ep_x^L \Xi^L (\Xi e_v^{1/2})^2 h_F \lsm N^{-1} (\Xi e_v^{1/2})^2 h_F$.

It now follows that $\co{T_1} \lsm N^{-1+\fr{1}{L}} (\Xi e_v^{1/2})^2 h_F$ since we have already proved this estimate for the term $T_2$ appearing in \eqref{eq:T1expand1}.  With this bound we have concluded the proof of Proposition~\ref{eq:advecCommutEstimate}.  
\end{proof}

\subsection{Verifying the conclusions of the Main Lemma} \label{sec:verifying}
In this Section, we verify all of the conclusions of Lemma~\ref{lem:mainLem}. 

The choice of $\bar{\de} > 0$ depending on $L$ is made in the concluding paragraph of Section~\ref{sec:solvingMainCoeffs}.  

The bounds in \eqref{eq:frEnLvlsNew} for the components of $\ost{R}^{j\ell} = \ost{R}_{[2]}^{j\ell} + \ost{R}_{[3]}^{j\ell} + \ost{R}_{[G]}^{j\ell}$, of $\ost{\kk} = \ost{\kk}_{[2]} + \ost{\kk}_{[G]}$ and of $\ost{\vp}^j = \ost{\vp}^j_{[2]} + \ost{\vp}_{[G]}^j$ follow from the bounds we have already proven.  The bound $\co{\ost{R}_{[G]}} \leq (e_v^{1/2}/(e_\vp^{1/2} N))^{1/2} e_\vp$ for $\ost{R}_{[G]} = R_T^{j\ell} + R_H^{j\ell} + R_M^{j\ell} + \wtld{R}_S^{j\ell}$, follows by combining the estimates \eqref{eq:RTHbdd} and \eqref{eq:c0RMbd}, using the comparison inequality $\ost{H}[F] \lsm \brh[F]$ of Lemma~\ref{lem:comparisonLem} and taking $B_\la$ sufficiently large in \eqref{eq:lastLowerOrdBd}.  The remaining bounds on derivatives of $\ost{R}_{[G]}$ and the bounds on the other components of $\ost{R}^{j\ell}$ (namely $\ost{R}_{[2]}^{j\ell} = P \de_{[2\ast]}^{j\ell} + R_{[2\ast]}^{j\ell}$ and $\ost{R}_{[3]}^{j\ell} = R_{[3\ast]}^{j\ell}$ follow from \eqref{eq:RMbd}, \eqref{eq:PR2bds}, \eqref{eq:lastLowerOrdBd} for all $\hc$ sufficiently large compared to $B_\la$.  The bounds for the components $\ost{\kk} = \ost{\kk}_{[2]} + \ost{\kk}_{[G]}$ follow from \eqref{eq:ostk2bds},\eqref{eq:ostkkGbdd}, using again that $\ost{H}[F] \lsm \brh[F]$ and taking $\hc$ large compared to $B_\la$.  The bounds on the components of $\ost{\vp}_{[2]}^j$ follow from the inequality $\brh[\ost{\vp}_{[2]}^j] \lsm \ost{\undl{e}}_\vp^{3/2} \lsm e_{\undvp}^{3/2}$ proven in Section~\ref{sec:mainTermProblemCurrent}, while the bounds for $\ost{\vp}_{[G]}^j$ follow from Propositions~\ref{prop:mollVp},\ref{prop:typeAprop} and \ref{prop:divEqnPropCurrent} after taking $\hc$ sufficiently large.

We obtain the bounds in \eqref{eq:frEnLvlsNew} for $\ost{v}^\ell = v^\ell + V^\ell$ as follows.  Using \eqref{eq:vpBd}, \eqref{eq:highFreqCorrect}, and $\la \geq N \Xi$, $N \geq e_v^{1/2}/e_\vp^{1/2}$, we have for all $1 \leq |\va| \leq L$ that 
\ALI{
\co{\nb_{\va} \ost{v}} \leq \co{\nb_{\va} v} + \co{\nb_{\va} V} \lsm \Xi^{|\va|} e_v^{1/2} + \la^{|\va|} e_\vp^{1/2} \lsm \la^{|\va|} e_\vp^{1/2}.
}
For $\hc$ sufficiently large compared to $B_\la$, we obtain $\co{\nb_{\va} \ost{v}} \leq \ost{\Xi}^{|\va|} \ost{e}_v^{1/2}$ for all $1 \leq |\va| \leq L$.

To establish \eqref{eq:frEnLvlsNew} for $\ost{p} = p + P$, we first recall the seminorms $\undl{\mathring{H}}[\,\cdot\,]$, $ \undl{\widehat{H}}_{\diamond}[\,\cdot\,]$ of Lemma~\ref{lem:comparisonLem} and \eqref{eq:homogDmdNorm}.  
We now apply Lemma~\ref{lem:comparisonLem} and the bounds \eqref{eq:vpBd}-\eqref{eq:pAdvecBd}, \eqref{eq:PR2bds} and \eqref{eq:dmdCompare} with $F = p$ to obtain 
\ALI{
\undl{\ost{H}}[\ost{p}] &\leq \undl{\ost{H}}[p] + \undl{\ost{H}}[P] \lsm \undl{H}[p] + \undl{\brh}[P] \\
\undl{\ost{H}}[\ost{p}] &\lsm (e_v/e_\vp)^{-1} \undl{\widehat{H}}_{\diamond}[p] + \brh[P] \lsm (e_v/e_\vp)^{-1} e_v + e_\vp \lsm e_\vp.
}
This bound achieves our desired estimate for $\ost{p}$ after taking $\hc$ sufficiently large.

The conditions of $(\bar{\de}, \hc)$-well-preparedness (for stage $[2]$) as defined in Definition~\ref{defn:wellPrepared} have all been checked throughout the course of Section~\ref{sec:mainTermProblemCurrent}.  It is clear from the formula \eqref{eq:enIncdef} that $e(t)$ depends only on the given frequency-energy levels, $L$ and $\tilde{I}_{[G]}$ since $K_0$ depends only on $L$.

The bound \eqref{eq:c0CorrectBd} on $\co{V}$ follows from \eqref{eq:highFreqCorrect}.  Using that $V^\ell = P_{\approx \la} V^\ell$ is localized to frequencies of the order $\la = B_\la N \Xi$, we have for any smooth test function $\phi_\ell(t,x)$ that
\ALI{
\int_{\T^3} \phi_\ell(t,x) V^\ell(t,x) dx &= \int_{\T^3} P_{\approx \la} \phi_\ell(t,x)  V^\ell(t,x) dx \\
\sup_t \left\| \int_{\T^3} \phi_\ell(t,x) V^\ell(t,x) dx \right\| &\lsm \sup_t \| P_{\approx \la} \phi_\ell(t,\cdot) \|_{L^1} \co{ V} \lsm \la^{-1} \sup_t \| \nb \phi(t,\cdot) \|_{L^1} e_\vp^{1/2}, 
}
which implies \eqref{eq:testFunction} for $B_\la$ sufficiently large.

Since $\ost{\mu} - \mu = - \Ddt e(t)$ as stated after \eqref{eq:fltermLow} and $\Ddt e(t) = \pr_t e(t) \leq 0$ as required in \eqref{eq:eDef}, it is clear that $\ost{\mu} - \mu$ is non-negative and also depends only on the given frequency-energy levels, $L$ and $\tilde{I}_{[G]}$.  We arranged for $e(t)$ to be constant on a neighborhood of $\tilde{I}_{[G]}$, so we have that $\ost{\mu} = \mu$ on $\tilde{I}_{[G]}$.

The choices of $\tilde{I}_{[l+1]}$ and $\tilde{I}_{[G\ast]}$ stated in \eqref{eq:tildeIlp1}-\eqref{eq:suppContained} coincide with those of Section~\ref{sec:mainTermProblemCurrent}.  From the overall construction, we have that $V^\ell$ as well as all the components of $(\ost{R}, \ost{\kk}, \ost{\vp})$ have support contained in $\suppt e(t)$, and the latter support is contained in $\tilde{I}_{[G\ast]}$ by the remarks of Section~\ref{sec:mainTermProblemCurrent}.

The final remaining task is to construct the second dissipative Euler-Reynolds flow $v_2^\ell = v^\ell + V_2^\ell$.  Let $\tilde{J}$ be any subinterval of $\tilde{I}_{[1]}$ with length $|\tilde{J}| \geq (\Xi e_v^{1/2})^{-1}$.  Choose an index $I^* \in \ovl{\II}_R$ from the ``larger waves'' such that the lifespan interval $\{ |t - t(I^*)| \leq \ostu \}$ is entirely contained in $\tilde{J}$.  This choice is possible since $\ostu \leq 10^{-1}(\Xi e_v^{1/2})^{-1}$, and the union of the lifespan intervals covers $\tilde{J}$.  

We define $v_2^\ell$ by simply replacing the wave $V_{I^*}^\ell$ by its negation $- V_{I^*}^\ell$, and its conjugate $V_{\bar{I}^*}^\ell$ by $- V_{\bar{I}^*}^\ell$ throughout the entire construction.  In other words, $v_2^\ell := \ost{v}^\ell - 2 V_{I*}^\ell - 2 V_{\bar{I}^*}^\ell$.  Making this replacement disturb does not disturb any estimate or important equality obtained in the construction.  Most importantly, the equation \eqref{eq:principOvlReq} remains satisfied since the summation $\sum_{I \in \ovl{\II}_R} v_I^j \overline{v_I^\ell}$ that appears there remains unchanged if one replaces any of the amplitudes $v_I^\ell$ by its negation.   It is also clear that making this change affects the values of $(\ost{v}, \ost{p})$ and the components of the new $(\ost{R}, \ost{\kk}, \ost{\vp})$ only within the interval $\{ |t - t(I^*)| \leq \tau \} \subseteq \tilde{J}$ that contains the lifespan of $V_{I^*}^\ell$ and its conjugate and not elsewhere.  This observation implies the containment \eqref{eq:diffSupp}.  What remains to prove is the lower bound \eqref{eq:lowerBdNew} for 
\ali{
\int_{\T^3} |v_2(t,x) - \ost{v}(t,x)|^2 dx &= 4 \int_{\T^3} |V_{I^*} +V_{\bar{I}^*}|^2 dx. \label{eq:lowBoundThis}
}
Since the frequency supports of $V_{I^*}^\ell$ and its conjugate are disjoint, $V_{I^*}^\ell$ and $V_{\bar{I}^*}^\ell$ are $L^2$-orthogonal.  Using that $V_{I}^\ell = \mathring{V}_I^\ell + \de V_I^\ell$, $\mathring{V}_I^\ell = e^{i \la \xi_I} v_I^\ell$ and that the $v_I^\ell = v_{\bar{I}}^\ell$ are real-valued, we then have that 
\ALI{
\int_{\T^3} |V_{I^*} +V_{\bar{I}^*}|^2 dx &= \int_{\T^3} |V_{I^*}|^2 +|V_{\bar{I}^*}|^2 dx = 2 \int_{\T^2} |v_{I^*}|^2(t,x) dx + 2 \, \mbox{Re} \int_{\T^3} (2 \de V_{I^\ast} \cdot V_{I^*} - |\de V_{I^\ast}|^2) dx
}
For $B_\la$ chosen sufficiently large, the last term may be bounded using Proposition~\ref{prop:correctionBounds} by
\ali{
2 \sup_t \left| \mbox{Re} \int_{\T^3} (2 \de V_{I^\ast} \cdot V_{I^*} - |\de V_{I^\ast}|^2) dx \right| &\leq A_L (B_\la N)^{-1} e_\vp \leq 4^{-1} N^{-1} e_\vp. \label{eq:errBoundedTerm}
}
Achieving inequality \eqref{eq:errBoundedTerm} together with the desired bound on $\co{\ost{R}_{[G]}}$ are the last requirements we impose on the parameter $B_\la$.  We thus specify $B_\la$ at this point to be any sufficiently large value (depending on all previous choices of parameters) to ensure the estimates stated so far, and to ensure that $\la = B_\la N \Xi$ is an integer-multiple of $2 \pi$.  The final constants $\hc, C_L$ are then allowed to be large depending on the implied upper bound for $B_\la$ (although this condition is not necessary for $C_L$).   

Recall now the form of $v_I^\ell = \bar{e}^{1/2}(t) \eta_I \ga_I \tilde{f}_I^\ell$, where $\tilde{f}_I^\ell = \htf_I^\ell - |\nb \xi_I|^{-2} \htf_I \cdot \nb \xi_I \nb^\ell \xi_I$.  At time $t^* = t(I^*)$, we have that $\eta_{I^*}(t^*) = 1$, $\tilde{f}_I^\ell(t^*,x) = \htf_I^\ell$, and $v_{I^*}^\ell = \bar{e}^{1/2}(t^*) \ga_{I^*}(t^*, x) \htf_{I^*}^\ell$.  We obtained a lower bound of $\ga_I \geq 1/4$ for all $I \in \ovl{\II}_R$ at the end of Section~\ref{sec:solvingMainCoeffs}.  We also have that $\bar{e}(t) \geq 10^{-2} \undl{e}_{\vp} \geq 10^{-2} M^{-1} e_\vp$ for all $t \in \tilde{I}_{[1]}$, which follows from the $C^0$ bound in \eqref{eq:underBdsEninc}.  Combining these observations, we obtain 
\ali{
2 \int_{\T^2} |v_{I^*}|^2(t^*,x) dx &= 2 \int_{\T^2} \bar{e}(t^*) |\ga_{I^*}(t^*,x)|^2 |\htf_{I^*}|^2 dx \geq 2^{-3} \cdot 10^{-2} |\htf_{I^*}|^2 M^{-1} e_\vp. \label{eq:mainLowBdEnincdiff}
}
We note that the directions $\htf_I$ chosen at the end of Section\ref{sec:solvingMainCoeffs} are all nonzero vectors.  We may therefore require $C_L \geq 2 \cdot 10^2 \max_{\htf \in \ovl{\BB}_R} |\htf|^{-2}$ to finalize the choice of $C_L$.  
Combining \eqref{eq:lowBoundThis},\eqref{eq:errBoundedTerm}, and \eqref{eq:mainLowBdEnincdiff}, we obtain the desired lower bound stated in \eqref{eq:lowerBdNew}.  This bound concludes the proof of Lemma~\ref{lem:mainLem}.



\section{Extensions to general dimensions} \label{sec:othDimRemarks}
In this Section we discuss how our approach may be adapted to other dimensions $n \neq 3$.

The above proof can be extended very directly to higher dimensions $n \geq 3$.  To accomplish this extension, we redefine the sets $\dmd{\BB}_R \subseteq \ker dx^1$ and $\ost{\BB}_R \subseteq \ker dx^2$ so that equation \eqref{eq:prop3basis} becomes 
\ali{
\begin{split} \label{eq:goodMetDecomp}
\sum_{\htf \in \dmd{\BB}_{R}} \htf^j \htf^\ell &+ \sum_{\hat{g} \in \ost{\BB}_{R}} \hat{g}^j \hat{g}^\ell = \de^{j\ell},  \\
\sum_{\hat{f} \in \dmd{\BB}_R} \htf \otimes \htf = e_2 \otimes e_2 + \sum_{j=3}^n (1/2) (e_j \otimes e_j),& \quad 
\sum_{\hat{g} \in \ost{\BB}_R} \hat{g} \otimes \hat{g} = e_1 \otimes e_1 + \sum_{j=3}^n (1/2) (e_j \otimes e_j).
\end{split}
}
These sets now have cardinality $\# \dmd{\BB}_R = \# \ost{\BB}_R = n(n-1)/2$ so that the tensors $(\htf^j \htf^\ell)_{\htf \in \dmd{\BB}_R}$ and $(\hat{g}^j \hat{g}^\ell)_{\hat{g} \in \ost{\BB}_R}$ may form bases for the vector spaces $\ker dx^1 \otimes \ker dx^1$, $\ker dx^2 \otimes \ker dx^2$ respectively.

The key observations from linear algebra that we require to carry out the other parts of the construction remain true.  Namely, since every element of $\R^d \otimes \R^d$ may be written as a sum of three tensors that belong to $\ker dx^i \otimes \ker dx^i$ for $1 \leq i \leq 3$, we may construct projection operators $\pi_{(1)}, \pi_{(2)}, \pi_{(3)}$ on $\R^d \otimes \R^d$ that take values respectively in $\pi_{(i)} \in \ker dx^i \otimes \ker dx^i$ and preserve the subspace of symmetric tensors while satisfying $\pi_{(1)} + \pi_{(2)} + \pi_{(3)} = \mbox{Id}$ on $\R^3 \otimes \R^3$ .  Similarly, since every element of $\R^d$ belongs to $\ker dx^1 + \ker dx^2$, there exist projections $\pi_{[1]}, \pi_{[2]}$ with values in $\ker dx^1$ and $\ker dx^2$ that satisfy $\pi_{[1]} + \pi_{[2]} = \mbox{Id}$ on $\R^d$.  

One sees also that the important cancellations of lines \eqref{eq:amazingCancel} and \eqref{eq:amazingCancelParallel} still hold once one makes the appropriate modifications in defining the terms.  Namely, the projection operators $\pi_\times$ and $\pi_\parallel$ on $\ker dx^3 \otimes \ker dx^3$ should be chosen such that $\pi_\parallel$ takes values in the span $\langle e_i \otimes e_i \rangle_{i \neq 3} + \langle e_i \otimes e_j + e_j \otimes e_i \rangle_{i, j \neq 1}$ while $\pi_\times$ takes values in $\langle e_1 \otimes e_j + e_j \otimes e_1 \rangle_{j \neq 1}$.  With these choices, one maintaines the cancellation of line \eqref{eq:amazingCancelParallel} (namely that $\pi_\parallel R_{[3 \ast]}$ takes values orthogonal to $dx^1 \otimes \langle e_1 \rangle^\perp$), while $\pi_\times R_{[3\ast]}$ maps $\ker dx^1$ to the span of $\langle e_1 \rangle \subseteq \ker dx^2$ as desired.  

Extending the above proof to dimension $d = 2$ would require substantially more work.  The main issue is that even in two dimensions one requires waves to take values in three distinct hyperplanes (such as $\ker dx^1$, $\ker dx^2$, $\ker dx^3$ above) in order to span the space of symmetric tensors (as in the decomposition $\R^d \otimes \R^d = \sum_{i=1}^3 \ker dx^i \otimes \ker dx^i$ if $d \geq 3$).  However, a decomposition such as \eqref{eq:goodMetDecomp} in dimension $2$ would require the directions of the different stages to be orthogonal, which would limit the number of allowed directions to only $2$.  The closest analogue to \eqref{eq:goodMetDecomp} would be to choose distinct rational lines $\ker du^1, \ker du^2, \ker du^3$ in $\R^2$ for which there exists a decomposition of the form
\ali{
\de^{j\ell} = \de_{[1]}^{j\ell} + \de_{[2]}^{j\ell} + \de_{[3]}^{j\ell}, \notag 
}
with $\de_{[1]}^{j\ell} = \htf^j \htf^\ell$, $\htf \in \ker du^1$, $\de_{[2]}^{j\ell} = \hat{g}^j \hat{g}^\ell$, $\hat{g} \in \ker du^2$ and $\de_{[3]}^{j\ell} = \hat{h}^j \hat{h}^\ell$, $\hat{h} \in \ker du^3$ all nonzero.  The existence of a decomposition of $\de^{j\ell}$ of this form is equivalent to the existence of vectors $\hat{f}, \hat{g}, \hat{h}$ spanning the lines $\ker du^1, \ker du^2, \ker du^3$ such that the angle between each pair of vectors is obtuse.  (One concrete example would be to take $du^1 = dx^1$, $du^2 = d(x^1 + 2 x^2)$, $du^3 = d(x^1 - 2x^2)$ with $\htf = (\sqrt{3}/2) e_2$, $\hat{g} = (1/\sqrt{8})( 2 e_1 - e_2)$ and $\hat{h} = (1/\sqrt{8})(-2e_1 - e_2)$.)  One would then hope to perform the iteration as above with waves taking values essentially in the periodic sequence of subspaces $(\ker du^1, \ker du^2, \ker du^3, \ker du^1, \ldots)$.  The difficulty this approach faces is that the cancellation of line \eqref{eq:amazingCancel} no longer appears to be available in the presence of such additional terms.  For this reason we have not been able to extend our results to two dimensions without incurring a loss of regularity and we leave the consideration of the two-dimensional case for future work.

\section{Proof of the Main Theorem} \label{sec:mainLImpMThm}
With the Main Lemma, Lemma~\ref{lem:mainLem}, now proven we proceed to the proofs of Theorems~\ref{thm:enDispNunq} and \ref{thm:enEqNunq}.  

The proof of Theorems~\ref{thm:enDispNunq} and \ref{thm:enEqNunq} will proceed by iteratively applying Lemma~\ref{lem:mainLem} to obtain a sequence of dissipative Euler-Reynolds flows $(v,p,R,\kk,\vp, \mu)_{(k)}$ that converge uniformly to an Euler flow that satisfies the local energy inequality.  The size of the error and the growth of frequencies during the iteration are dictated by the bound \eqref{eq:frEnLvlsNew}, which we restate here as
\ali{
\mat{ccccc}{ \ost{\Xi} & \ost{e}_v & \ost{e}_\vp & \ost{e}_R & \ost{e}_G \\
							\hc N \Xi & \hc e_\vp & \hc e_{\undvp}  & \hc e_G & (N^{-1} (e_v^{1/2}/e_\vp^{1/2}))^{1/2} e_\vp
					}, \qquad e_{\undvp} := e_\vp^{1/3} e_R^{2/3}. \label{eq:evolutionRule}
	}
The sequence of (compound) frequency-energy levels will be dictated by specifying that $e_{G,(k+1)} = Z^{-1} e_{G,(k)}$ for some large constant\footnote{Taking $Z$ to depend on $k$ would be important for achieving an endpoint type result as in \cite{isettEndpt}.} $Z$.  We will therefore choose $N = N_{(k)} = Z^2 (e_v^{1/2}/e_\vp^{1/2})_{(k)}(e_\vp/e_G)_{(k)}^2$ when applying Lemma~\ref{lem:mainLem}.  

The fact that the evolution rule \eqref{eq:evolutionRule} contains the term $e_{\undvp} = e_\vp^{1/3} e_R^{2/3}$ that couples together distinct energy levels introduces some additional complexity into the iteration compared to previous schemes.  To approach this issue, we introduce the notation $Q_{v\vp} = (e_v/e_\vp)$, $Q_{\vp \undvp} = (e_\vp/e_{\undvp})$, $Q_{\vp R} = (e_\vp/e_R)$, etc.\ to denote the quotients of the corresponding energy levels, and we write $L_{v \vp} = \log Q_{v \vp}$, $L_{\vp R} = \log Q_{\vp R}$, etc.\ to denote their natural logarithms.  We also write $L_G = \log e_G$, $L_\Xi = \log \Xi$, $L_Z = \log Z$ and $L_{\hc} = \log \hc$.  We will write $Q_{v\vp}^{(k)}$, $L_G^{(k)}$, etc. when we wish to consider the dependence of the parameters on the stage $k$ of the iteration, and will often suppress this dependence.  In the above notation, for example, we have $N = Z^2 Q_{v\vp}^{1/2} Q_{\vp G}^2$, suppressing the dependence on $k$.

We will mainly work with the variables $Q_{\vp \undvp}, Q_{\undvp G}$ rather than $Q_{\vp R}$ and $Q_{R G}$ to simplify some calculations.  The equations $Q_{\vp \undvp} = (e_\vp/e_R)^{2/3} = Q_{\vp R}^{2/3}$ and $Q_{\undvp G} = (e_\vp^{1/3} e_R^{2/3}/e_G) = Q_{\vp R}^{1/3} Q_{RG}$ allow us to convert between the two pairs.  In the new variables we have $\ost{Q}_{\vp G} = \ost{Q}_{\vp R}^{1/3} \ost{Q}_{RG} = Q_{\undvp G}^{1/3}(\hc Z)$, $\ost{Q}_{\vp \undvp} = \ost{Q}_{\vp R}^{2/3} = Q_{\undvp G}^{2/3}$ and $\ost{\Xi} = \hc Z^2 Q_{v\vp}^{1/2} Q_{\vp \undvp}^2 Q_{\undvp G}^2 \Xi_{(k)}$.  We write these rules in matrix notation as
\ali{
\vect{L_Z \\ L_G \\ L_{\undvp G} \\ L_{\vp \undvp} \\ L_{v \vp} \\ L_\Xi }_{(k+1)} &= 
			\mat{cccccc}{1 & 0 & 0 & 0 & 0 & 0 \\
									-1 & 1 & 0 & 0 & 0 & 0	\\
									1 & 0 & 1/3 & 0 & 0 & 0 \\
									0 & 0 & 2/3 & 0 & 0 & 0 \\
									0 & 0 & 0   & 1 & 0 & 0 \\
									2 & 0 & 2 & 2 & 1/2 & 1
									} \vect{L_Z \\ L_G \\ L_{\undvp G} \\ L_{\vp \undvp} \\ L_{v \vp} \\ L_\Xi }_{(k)} + \vect{0 \\ 0 \\ L_{\hc} \\ 0 \\ 0 \\ L_{\hc} }. \label{eq:matrixParamEvol}
}
We let $L_{(k)}$ denote the column vector of parameters in \eqref{eq:matrixParamEvol} and $T$ denote the $6 \times 6$ matrix appearing above.  In terms of the difference operator $\de_{(k)} f_{(k)} := f_{(k+1)} - f_{(k)}$, equation \eqref{eq:matrixParamEvol} may be rewritten as a difference equation $\de_{(k)} L_{(k)} = (T - 1)L_{(k)} + L_{\hc} (e_3 + e_6)$.  We then derive the identity
\ali{
15 \de_{(k)} L_G + 11 \de_{(k)} L_{\undvp G} + 5 \de_{(k)} L_{\vp\undvp} + \de_{(k)} L_{v \vp} + 2 \de_{(k)} L_\Xi &= 13 L_{\hc}. \label{eq:evolRule}
}
The coefficients $15$ and $2$ appearing in \eqref{eq:evolRule} will ultimately lead to the regularity $1/15$ for the iteration.  This equation may be obtained by locating the row vector $[0, 15,  11, 5, 1, 2]$ in the row null space of $(T - 1)$, and applying this vector to the difference equation for $L_{(k)}$.

To ensure the iteration proceeds in a well-defined way, it is necessary that we check that the parameter $N$ chosen above satisfies the admissibility conditions \eqref{eq:Nastdef} assumed in the Main Lemma.  To do so it is only necessary to check the latter two inequalities in \eqref{eq:Nastdef}, since the first inequality, which requires $N \geq Q_{vG} Q_{\vp \undvp}^{1/2} = Q_{v\vp}Q_{\vp \undvp}^{3/2} Q_{\undvp G}$, follows from the latter two by taking their geometric mean.  The latter inequalities 
will be addressed in Section~\ref{sec:admissConditions} below.

After establishing sufficient conditions on the parameters to ensure a well-defined iteration, the proofs of Theorems~\ref{thm:enDispNunq} and \ref{thm:enEqNunq} will be completed in the concluding Section~\ref{sec:constructDissipSolns}.

\subsection{Admissibility conditions and asymptotics for continuing the iteration} \label{sec:admissConditions}
In this section we consider the admissibility conditions in \eqref{eq:Nastdef} for the parameter $N = Z^2 Q_{v\vp}^{1/2} Q_{\vp G}^2$.  The main result of this section is to isolate a range of initial parameters that satisfy the conditions in \eqref{eq:Nastdef} and allow the iteration to proceed by induction in a well-defined way.  We will use the notation $u^t$ to indicate the transpose of a column or row vector $u$.


As remarked in the previous section, we need only check the latter two inequalities in  \eqref{eq:Nastdef}, since these imply the first one.  These inequalities may be written as $N \geq Q_{v\vp}^{1/2} Q_{\vp G}^2 Q_{\vp \undvp} = Q_{v \vp}^{1/2} Q_{\vp \undvp}^3 Q_{\undvp G}^2$ and $N \geq Q_{v\vp}^{3/2}$, where $N = Z^2 Q_{v\vp}^{1/2} Q_{\vp \undvp}^2 Q_{\undvp G}^2$.  In logarithmic form these conditions become
\ali{
2 L_Z - L_{\vp \undvp} &\geq 0, & 
2 L_Z + 2 L_{\undvp G} + 2 L_{\vp \undvp} - L_{v\vp} &\geq 0. \label{eq:admissCondits}
}
We write $\mrg{L} = [L_Z, L_{\undvp G}, L_{\vp \undvp}, L_{v\vp}]^t$ to denote the column vector of logarithms of these parameters, which evolves according to the equation $\mrg{L}_{(k+1)} = \mrg{T} \mrg{L}_{(k)} + L_{\hc} e_2$ that in matrix notation becomes
\ali{
\vect{ L_Z \\ L_{\undvp G} \\ L_{\vp \undvp} \\ L_{v\vp} }_{(k+1)} 
&= \mat{cccc}{1 & 0 & 0 & 0 \\
							1 & 1/3 & 0 & 0 \\
							0 & 2/3 & 0 & 0 \\
							0 & 0 & 1 & 0
}\vect{ L_Z \\ L_{\undvp G} \\ L_{\vp \undvp} \\ L_{v\vp} }_{(k)} + \vect{0 \\ L_{\hc} \\ 0 \\ 0}. \label{eq:affineTrans}
}
Observe that $\mrg{T}$ has eigenvalues $(1,1/3,0,0)$, and that the $1$-eigenspace is the span of the vector $\zeta = [1,3/2,1,1]^t \in \ker(T - 1)$.  The crucial point in checking the condition \eqref{eq:admissCondits} is that $\zeta$ (and hence any positive multiple of $\zeta$) satisfies \eqref{eq:admissCondits} with strict inequality.  The positive linear span of $\zeta$ is on the other hand not invariant under the affine transformation \eqref{eq:affineTrans}, so we cannot perform an induction with parameters in the span of $\zeta$.  However, as the following Proposition shows, if we start with parameters sufficiently close to the positive span of $\zeta$ and $Z$ is sufficiently large, then  \eqref{eq:admissCondits} will remain satisfied.
\begin{prop}[Admissible truncated sector] \label{prop:admissSect} There exist $r_0 \in (0,1/2)$ and $\undl{Z} \geq 1$ such that the truncated sector defined by
\ali{
\wtld{C} &:= \{ \mrg{L} \in \R^4 : \mrg{L} = L_Z(\zeta + \varep), \quad \varep_1 = 0, \quad \max_i |\varep_i| \leq r_0, \quad L_Z \geq L_{\undl{Z}} \} \label{eq:truncSect}
}
is contained in the set of solutions to \eqref{eq:admissCondits} and is mapped to itself by the affine transformation \eqref{eq:affineTrans}.
\end{prop}
\begin{proof}  Recall that both inequalities in \eqref{eq:admissCondits} hold with strict inequality for $\zeta = [1,3/2,1,1]^t$.  We may therefore choose $r_0 \in (0,1/2)$ such that for all $\varep \in \R^4$ with $\max_i |\varep_i| \leq r_0$ the vector $\zeta + \varep$ also satisfies both inequalities in \eqref{eq:admissCondits}.  With this choice of $r_0$ we have that every vector in $\wtld{C}$ also satisfies the conditions in \eqref{eq:admissCondits}, since the value of $L_Z \geq L_{\undl{Z}} > 0$ is positive and the conditions are linear.

Now suppose $\mrg{L} = L_Z(\zeta + \varep)$ belongs to the truncated sector $\wtld{C}$ defined in \eqref{eq:truncSect}.  Applying the affine transformation in \eqref{eq:affineTrans} and recalling that $\mrg{T} \zeta = \zeta$, we have 
\ali{
\mrg{T} \mrg{L} + L_{\hc} e_2 = L_Z(\zeta + \ost{\varep}), \quad \ost{\varep} := \mrg{T} \varep + L_Z^{-1} L_{\hc} e_2.
}
Computing with the matrix $\mrg{T}$ in \eqref{eq:affineTrans} and using $\varep_1 = 0$, we have that $\ost{\varep}_1 = 0$ and we calculate $\ost{\varep}_2 = \varep_2/3 + L_Z^{-1} L_{\hc}$.  We thus obtain $|\ost{\varep}_2| \leq r_0/3 + L_Z^{-1} L_{\hc}$.  For $\undl{Z}$ chosen sufficiently large, we have $L_Z^{-1}L_{\hc} \leq L_{\undl{Z}}^{-1} L_{\hc} \leq 2r_0/3$ and hence $|\ost{\varep}_2| \leq r_0$.  We also have that $|\ost{\varep}_3| = 2|\varep_2|/3 \leq 2r_0/3 \leq r_0$ and $|\ost{\varep}_4| = |\varep_3| \leq r_0$.  These inequalities together imply that $\mrg{T} \mrg{L} + L_{\hc} e_2$ belongs to $\wtld{C}$ as desired.
\end{proof}

At this point, it is also useful to make the following observation on the limiting behavior of $\mrg{L}_{(k)}$.
\begin{prop} \label{prop:changingQuotients} If $\mrg{L}_{(k)}$ evolves by the rule \eqref{eq:affineTrans}, then $\lim_{k \to \infty} \mrg{L}_{(k)}$ exists and $\lim_{k \to \infty} \de_{(k)} \mrg{L}_{(k)} = 0$, where $\de_{(k)} \mrg{L}_{(k)} := \mrg{L}_{(k+1)} - \mrg{L}_{(k)}$.
\end{prop}
\begin{proof} Let $\mrg{L}_{(k)}$ obey \eqref{eq:affineTrans}.  Then the first component $L_Z = e_1^t \mrg{L}_{(k)}$ remains constant in $k$, which implies that $L_{\hc} = \eta L_Z$, where $\eta > 0$ is constant in $k$.  With this observation, we can re-write the parameter evolution as the iteration of a fixed matrix, rather than an affine map.  That is, for all $k$ we obtain that $\mrg{L}_{(k)} = T_\eta^k \mrg{L}_{(0)}$, $T_\eta = (\mrg{T} + \eta e_2 e_1^t)$.  Note that $T_\eta$ is lower triangular with a diagonal consisting of the eigenvalues $(1,1/3,0,0)$.  As a consequence, $\lim_{k \to \infty} T_\eta^k$ exists and is equal to the $T_\eta$-invariant projection onto its $1$-eigenspace.  It follows that $\mrg{L}_{(\infty)} := \lim_{k \to \infty} \mrg{L}_{(k)} =  \lim_{k \to \infty} T_\eta^k \mrg{L}_{(0)}$ exists, and as a consequence $\lim_{k \to \infty} \de_{(k)} \mrg{L}_{(k)} = \mrg{L}_{(\infty)} - \mrg{L}_{(\infty)} = 0$.
\end{proof}

We will also use the following estimates on the decay rates of the time scale and the error.
\begin{prop} There exists $\undl{Z}_1 > 0$ such that for all  $Z \geq \undl{Z}_1$ if $L_{(k)}$ evolves according to \eqref{eq:matrixParamEvol} and $\mrg{L}_{(0)}$ belongs to the set $\wtld{C}$ in \eqref{eq:truncSect},  we have for all $k \geq 0$ that
\ali{
Z^{-3/4} e_{\vp,(k)}^{1/2} \leq e_{\vp,(k+1)}^{1/2} &\leq Z^{-1/4} e_{\vp,(k)}^{1/2} \label{eq:shrinkingAmplitudeLaw} \\
Z \Xi_{(k)} e_{v,(k)}^{1/2} \leq \Xi_{(k+1)} e_{v,(k+1)}^{1/2} &\leq Z^9 \Xi_{(k)} e_{v,(k)}^{1/2}  \label{eq:shrinkingTimescaleLaw}
}
\end{prop}
\begin{proof}  Note that $e_{\vp,(k)} = Q_{\vp\undvp} Q_{\undvp G} e_{G,(k)}$.  Using \eqref{eq:evolRule},\eqref{eq:affineTrans} and \eqref{eq:truncSect} we have
\ALI{
\de_{(k)} \log e_{\vp,(k)} &= \de_{(k)} L_G  \, + \de_{(k)} L_{\vp \undvp}  \qquad \qquad \, \, +  \de_{(k)} L_{\undvp G} \\
&= (-L_Z) + ((2/3) L_{\undvp G} - L_{\vp\undvp}) + (L_Z + L_{\hc} - (2/3)L_{\undvp G}) \\
&= -L_{\vp \undvp} + L_{\hc} = -L_Z(1 + \varep_{2(k)}) + L_{\hc},
}
where $|\varep_{2(k)}| \leq r_0 < 1/2$.  For sufficently large $Z \geq \undl{Z}_1$ we have $-(3/2) L_Z \leq \de_{(k)} \log e_{\vp,(k)} \leq - (1/2) L_Z$, yielding \eqref{eq:shrinkingAmplitudeLaw}.  To obtain \eqref{eq:shrinkingTimescaleLaw}, we start by noting that
\ALI{
-\de_{(k)} \log (\Xi_{(k)} e_{v,(k)}^{1/2})   &= - (1/2) (\de_{(k)}L_G + \de_{(k)}L_{\undvp G} + \de_{(k)} L_{\vp\undvp} +\de_{(k)} L_{v\vp}) -\de_{(k)} L_\Xi   \\
&= (-1/2) [0, 1, 1,1,1,2] \de_{(k)} L_{(k)}, 
}
where $L_{(k)}$ is the column vector of parameters in \eqref{eq:evolRule}.  Using that $L_{(k)}$ obeys the evolution equation $\de_{(k)} L_{(k)} = (T - 1) L_{(k)} + L_{\hc}(e_3 + e_6)$, where $T$ is the matrix in \eqref{eq:matrixParamEvol}, 
we obtain
\ALI{
-\de_{(k)} \log (\Xi_{(k)} e_{v,(k)}^{1/2}) &= (-1/2)\left([4,0,4,4,0,0]L_{(k)} + 3 L_{\hc}\right) \\
\de_{(k)} \log (\Xi_{(k)} e_{v,(k)}^{1/2}) &= 2 L_Z + 2 L_{\undvp G} + 2 L_{\vp \undvp} + (3/2) L_{\hc} = 2 L_Z(1 + 3/2 + \varep_{2(k)} + 1 + \varep_{3(k)}) + 2 L_{\hc}.
}
Using that $\max_i |\varep_{i(k)}| \leq 1/2$, we obtain $L_Z \leq \de_{(k)} \log (\Xi_{(k)} e_{v,(k)}^{1/2}) \leq 9 L_Z$ for $L \geq \undl{Z}_1$ sufficiently large.  This bound implies the desired estimate \eqref{eq:shrinkingTimescaleLaw}.
\end{proof}


\subsection{The Approximation Theorem} \label{sec:approximThm}
We now state our Approximation Theorem, which immediately implies Theorems~\ref{thm:enDispNunq} and \ref{thm:enEqNunq}.  

\begin{thm}[Approximation Theorem] \label{thm:approxThm} Let $\a < 1/15$.  Then there exists a constant $\ovl{C}_\a > 0$ such that the following holds.  Suppose that $\Xi_0 \geq 1$ and $E_0 > 0$ are positive, that $\tilde{I}$ is an interval having length $|\tilde{I}| \geq 8(\Xi_0 E_0^{1/2})^{-1}$, and that $\tilde{I}_{(0)}, \tilde{I}_{(1)}$ are nonempty subintervals of $\tilde{I}$ such that  $\sup \tilde{I}_{(0)} + (\Xi_0 E_0^{1/2})^{-1} \leq \sup \tilde{I}_{(1)}$.  Let $(v,p,R,\kk,\vp,\mu)_{0}$ be a dissipative Euler-Reynolds flow on $\tilde{I} \times \T^3$ with compound frequency energy levels to order $2$ bounded by $(\Xi, e_v, e_\vp, e_R, e_G)_{0} = (\Xi_0, E_0, E_0, E_0, E_0)$ satisfying $\suppt (R,\kk,\vp)_{0} \subseteq \tilde{I}_{(0)}$.

Then there is a family $(v_\b, p_\b)$ of weak solutions to the incompressible Euler equations on $\tilde{I} \times \T^3$ of class $v \in C_{t,x}^\a$ that is parameterized by $\b \in 2^\N$ and has the following properties:

\noindent $1.$ For each $\b$, $(v_\b, p_\b)$ satisfies the local energy inequality \eqref{eq:locEnInq} with a common dissipation measure $\mu_\b := -[ \pr_t(|v_\b|^2/2) + \nb_j( (|v_\b|^2/2 + p_\b)v_\b^j) ] = \mu_{\infty} \geq 0$ that is independent of $\b$.

\noindent $2.$ On $\tilde{I}_{(0)} \times \T^3$, the weak solutions $(v_\a, p_\a)$ are all identically equal to each other.

\noindent $3.$ On $\tilde{I}_{(1)} \times \T^3$,  
the dissipation measure $\mu_\infty = \mu_{(0)}$ coincides with the initial dissipation measure.

\noindent $4.$ For all $\b \in 2^\N$, the support of $(v_\b, p_\b)$ is contained in $\{ t \leq \sup \tilde{I}_{(1)} + (\Xi E_0^{1/2})^{-1} \} \cap \tilde{I}$.

\noindent $5.$ The map $\b \mapsto v_\b$ is a homeomorphism of $2^\N$ onto its image in $C_{t,x}^\a$.

\noindent $6.$ The image $(v_\b)_{\b \in 2^{\N}}$ has positive Hausdorff dimension as a subspace of $C_tL_x^2(\tilde{I} \times\T^3)$.

\noindent $7.$ We have $\sup_{\b \in 2^\N} \co{v_\b - v_{0}} \leq \ovl{C}_\a E_0^{1/2}$. 
\end{thm}
\noindent We refer to the proof of Theorem~\ref{thm:approxThm} below for a review of the topology of $2^\N$.

Theorems~\ref{thm:enDispNunq} and \ref{thm:enEqNunq} follow immediately from Theorem~\ref{thm:approxThm} by taking the initial dissipative Euler-Reynolds flow to be identically $0$, the $\Xi_0 \geq 1, E_0 > 0$ to be arbitrary positive numbers, setting $\tilde{I} = \R$, and taking $\tilde{I}_{(0)} = (-\infty, 1)$, $\tilde{I}_{(1)} = (-\infty,2)$ in the case of Theorem~\ref{thm:enDispNunq} or taking $\tilde{I}_{(0)} = (-\infty, 1)$, $\tilde{I}_{(1)} = (-\infty,\infty)$ in the case of Theorem~\ref{thm:enEqNunq}.

The first step in the proof of the Approximation Theorem will be to observe that the initial dissipative Euler-Reynolds flow can be made $(\bar{\de},M)$-well-prepared for a suitable choice of energy increment function and a (larger) choice of frequency energy levels belonging to the admissible region \eqref{eq:truncSect} without making changes to the velocity field.  The lemma involves a parameter $Z$ that will be chosen to be a large constant (depending on $\a$) at the end of the argument.
\begin{lem} \label{lem:prepareInitLevels} There exists $\undl{Z}_1 > 0$ such if $(v,p,R,\kk,\vp,\mu)_{0}$ are as in Theorem~\ref{thm:approxThm} and if $Z \geq \undl{Z}_1$ there is a smooth function $\bar{e}_{Z} \colon \tilde{I} \to \R_{\geq 0}$ and a dissipative Euler-Reynolds flow of the form $(v,\tilde{p},\wtld{R},\tilde{\kk},\vp,\tilde{\mu})_{(0)}$ that is $(\bar{\de},Z^{2})$-well-prepared for stage 1 with respect to the compound frequency energy levels given by 
\ali{
(\Xi, e_v, e_\vp, e_R, e_G)_{(0)} &:= (\Xi_0,  Z^{7/2}E_0, Z^{5/2} E_0, Z E_0, E_0), \label{eq:bigInitFren}
}
the pair of intervals $(\tilde{I}_{[10]}, \tilde{I}_{[G0]})$, $\tilde{I}_{[10]} = \{ t \leq \sup \tilde{I}_{(1)} \}$, $\tilde{I}_{[G0]} = \{ t \leq \sup \tilde{I}_{(1)} + (\Xi_0 E_0^{1/2})^{-1} \}$, and the function $e_Z$.  One may arrange that $\tilde{\mu}_{(0)}(t,x) = \mu_{(0)}(t,x)$ for $t \leq \sup \tilde{I}_{(1)} + 4^{-1}(\Xi_0 E_0^{1/2})^{-1}$.
\end{lem}
During this proof we will ignore the subscripts in writing $(v,p,R,\kk,\vp,\mu) := (v,p,R,\kk, \vp,\mu)_{0}$.
\begin{proof}   We recall the notation $\calD[v,p] := \pr_t(|v|^2/2) + \nb_j((|v|^2/2 + p)v^j)$ for the resolved kinetic energy density.  We set $\tilde{p} = p + P$ with $P$ to be determined.  Recall also the notation $\de_{[1\ast]} = e_3 \otimes e_3 + (e_2\otimes e_2)/2$ and that $\de^{j\ell} = \de_{[1\ast]}^{j\ell} + \de_{[2\ast]}^{j\ell}$.  Starting with $\calD[v,p] = D_t \kk + \nb_j[v_\ell R^{j\ell}] + \nb_j \vp^j - \mu$ and using the Euler-Reynolds equations \eqref{eq:euReyn}, we obtain for any choice of non-increasing function $e_Z \colon \tilde{I} \to \R_{\geq 0}$ the following equalities
\ali{
\pr_t v^\ell + \nb_j(v^j v^\ell) + \nb^\ell \tilde{p} &= \nb_j(P \de^{j\ell} + R^{j\ell} ) \\
\calD[v,\tilde{p}]&= D_t(\kk - (1/2) \mbox{tr}(\de_{[1\ast]}) e_Z(t)) + \nb_j[v_\ell(P \de^{j\ell} + R^{j\ell})] + \nb_j \vp^j - \tilde{\mu}
}
by setting $\tilde{\mu} = \mu -(1/2) \mbox{tr}(\de_{[1\ast]})\pr_t e_Z(t) \geq 0$.

Define $\tilde{I}'_{(1)} := \{ t \leq \sup \tilde{I}_{(1)} + 2^{-1} (\Xi_0 E_0^{1/2})^{-1} \}$ and set $\bar{e}_Z^{1/2}(t) = Z^{1/4} E_0^{1/2} \eta_{\tau_0} \ast 1_{\tilde{I}'_{(1)}}(t)$, where $\eta_{\tau_Z}(t)$ is a standard, non-negative mollifier in $t$ with support in $|t| \leq (8 \Xi_0 E_0^{1/2})^{-1}$.  Define $\wtld{R}_{[1]}^{j\ell} = P \de_{[1\ast]}^{j\ell} + R_{[1]}^{j\ell}$ and choose $P = - e_Z(t)$, which implies that
\ali{
\tilde{\kk}_{[1]} := (1/2)\de_{j\ell} (P \de_{[1\ast]}^{j\ell} + R_{[1]}^{j\ell}) &= \kk_{[1]} - (1/2) e_Z(t) \mbox{tr}(\de_{[1\ast]}). \label{eq:mainCurrentDensityStart} 
}
We may now define the new dissipative Euler-Reynolds flow $(v, \tilde{p}, \wtld{R}, \tilde{\kk}, \tilde{\vp}, \tilde{\mu})$ with $\tilde{p}$ and $\tilde{\mu}$ defined as above, $\wtld{R}_{[1]} = - e_Z(t) \de_{[1]}^{j\ell} + R_{[1]}^{j\ell}$, $\wtld{R}_{[2]}^{j\ell} = - e_Z(t) \de_{[2\ast]}^{j\ell} + R_{[2]}^{j\ell}$, $\wtld{R}_{[G]}^{j\ell} = R_{[G]}^{j\ell}$, $\tilde{\kk}_{[1]}$ as in \eqref{eq:mainCurrentDensityStart}, 
$\tilde{\kk}_{[G]} = \kk_{[G]}$, and $\tilde{\vp}_{[1]} = \vp_{[1]}$, $\tilde{\vp}_{[2]} = \vp_{[2]}$.  We set $\undl{e}_\vp = Z^{1/2} E_0$ and observe that $\undl{e}_{\vp} \geq Z^{-2} e_{\vp,(0)}$ as defined in \eqref{eq:bigInitFren}.  The conditions of $(\bar{\de},Z^2)$-well-preparedness for $\wtld{R}_{[1\circ]}^{j\ell} = R_{[1]}^{j\ell}$ in Definition~\ref{defn:wellPrepared} with the above choices of tensor fields and parameters translate to the following:
\ALI{
\suppt (R_{[1]}, \vp_{[1]}) \subseteq \tilde{I}_{[10]},& \quad \suppt e_Z \subseteq \tilde{I}_{[G0]} \\
\co{D_t^s \bar{e}^{1/2}(t)} &\leq (\Xi_0 Z^{7/4} E_0^{1/2})^s (Z^{5/2} E_0)^{1/2}, \\
e_Z(t)^{-1} \co{R_{[1]}} + e_Z^{-3/2}(t) &\co{\vp_{[1]}} \leq \bar{\de}, \qquad \qquad \mbox{ if } t \leq \sup \tilde{I}_{[10]} + (\Xi_0 Z^{7/4} E_0^{1/2})^{-1} \\ 
\co{D_t^s [e_Z^{-1/2}(t)] } \leq 10 (\Xi_0 Z^{7/4}&E_0^{1/2})^s (Z^{1/2} E_0)^{-1/2} \quad \mbox{ if } t \leq \sup \tilde{I}_{[10]} + (\Xi_0 Z^{7/4} E_0^{1/2})^{-1} \\
(Z^{1/2} E_0)^{1/2} \co{\nb_{\va} D_t^r R_{[1]} } &+ \co{\nb_{\va} D_t^r \vp_{[1]} } \leq \Xi_0^{|\va|} (\Xi_0 Z^{7/4}E_0^{1/2})^r (Z^{1/2} E_0)^{3/2}, 
}
for all $0 \leq s \leq 2$, and all $0 \leq r \leq 1$, $0 \leq r + |\va| \leq 2$.  Note that the above conditions are clear from construction for $Z$ sufficiently large using the bounds assumed on $(v,p,R,\kk,\vp,\mu)$.  Similarly, the estimates \eqref{eq:vpBd}-\eqref{eq:quadDdtbdlvl} for the frequency energy levels \eqref{eq:bigInitFren} are also clear from construction for $Z$ sufficiently large using the bounds assumed on $(v,p,R,\kk,\vp,\mu)$.  With this estimate we obtain Lemma~\ref{lem:prepareInitLevels}.  
\end{proof}

We now proceed with the proof of Theorem~\ref{thm:approxThm}.  
\begin{proof}[Proof of Theorem~\ref{thm:approxThm}]  Let $\a < 1/15$ and let $\Xi_0, E_0, (\tilde{I}, \tilde{I}_{(0)}, \tilde{I}_{(1)})$, $(v,p,R,\kk,\vp,\mu)_{0}$ be as in the assumptions of Theorem~\ref{thm:approxThm}.  Let $Z$ be a large constant that will be specified at the end of the proof.  Let $(\Xi, e_v, e_\vp, e_R, e_G)_{(0)}$ be the compound frequency energy levels specified in line \eqref{eq:bigInitFren} and define the dissipative Euler-Reynolds flow $(v,p,R,\kk,\vp,\mu)_{(0)}$ to be the $(v,\tilde{p}, \wtld{R}, \tilde{\kk}, \tilde{\vp}, \tilde{\mu})$ obtained in Lemma~\ref{lem:prepareInitLevels} with $Z$ as above.  Let $L_{(0)} = [L_Z, L_G, L_{\undvp G}, L_{\vp\undvp}, L_{v \vp}, L_\Xi]_{(0)}^t$ and $\mrg{L}_{(0)} = [L_Z, L_{\undvp G}, L_{\vp\undvp}, L_{v \vp}]^t_{(0)}$ be the vectors of parameter logarithms following the notation of Section~\ref{sec:admissConditions}.  

From \eqref{eq:bigInitFren}, we obtain that $\mrg{L}_{(0)} = L_Z[1,3/2,1,1] = L_Z \zeta$ lies in the $1$-eigenspace of the matrix $\mrg{T}$ in \eqref{eq:affineTrans}.  Since $\mrg{L}_{(0)}$ lies in the admissible truncated sector \eqref{eq:truncSect}, we may apply Lemma~\ref{lem:mainLem} repeatedly to obtain a sequence of dissipative Euler-Reynolds flows $(v,p,R,\kk,\vp,\mu)_{(k)}$ with compound frequency energy levels below $(\Xi, e_v, e_\vp, e_R, e_G)_{(k)}$ that evolve according to the equation \eqref{eq:matrixParamEvol}.  Moreover, Lemma~\ref{lem:mainLem} allows us to construct many possible such sequences of Euler-Reynolds flows by providing two choices of dissipative Euler-Reynolds flows (either $\ost{v}$ or $v_2$) at each stage of the iteration.  We take advantage of this freedom as follows.

Let $\N$ be the set of non-negative integers and let $2^{\N}$ be its power set; i.e., each $\b \in 2^{\N}$ is a subset of $\N$.  We endow $2^\N$ with its natural product topology, which is
compact, Hausdorff and metrizable (in fact $2^{\N}$ is homeomorphic to the Cantor set).  In this topology a sequence $\b^{(j)}$ in $2^{\N}$ converges to $\b \in 2^{\N}$ when $\limsup \b^{(j)} = \bigcup_n \bigcap_{j \geq n} \b^{(j)}$ and $\liminf \b^{(j)} = \bigcap_n \bigcup_{n \geq j} \b^{(j)}$ are both equal to $\b$ (in other words, the sequence $\b^{(j)}$ converges pointwise when viewed as maps $\b^{(j)} \colon \N \to \{0, 1\}$).  The symmetric difference operation $\b \De \bar{\b} = (\b \cap \bar{\b}^c)\cup(\bar{\b}\cap \b^c)$, which is addition mod $2$ if we view $\b$ and $\bar{\b}$ as maps into $\Z/2\Z$, makes $2^\N$ a locally compact abelian group with identity element $\emptyset$, and one has the following property characterizing convergence of sequences in the topology of $2^\N$: 
\ali{
\b^{(j)} \to \b \mbox{ in } 2^\N \mbox{ as $j \to \infty$  if and only if } \min (\b^{(j)} \De \b) \to \infty \mbox{ in $\N$ as $j \to \infty$.} \label{eq:convergenceCharacterization}
}  

To each $\b \in 2^{\N}$ we associate a sequence of dissipative Euler-Reynolds flows $(v,p,R,\kk,\vp,\mu)_{\b,(k)}$ together with a sequence of frequency energy levels $(\Xi, e_v, e_\vp, e_R,e_G)_{(k)}$, intervals $\tilde{I}_{[1k]}, \tilde{I}_{[Gk]}$ and non-negative functions $\bar{e}_{(k)}(t)$ as follows.  We first define the sequence of frequency energy levels, intervals, and non-negative functions.  We let $(\Xi, e_v, e_\vp,e_R,e_G)_{(0)}$ be defined as in \eqref{eq:bigInitFren} and take $\tilde{I}_{[10]}$, $\tilde{I}_{[G0]}$ and $\bar{e}_{(0)} = e_Z$ as given by the Lemma.  The sequence of frequency energy levels is then dictated by the evolution rules \eqref{eq:evolutionRule}-\eqref{eq:matrixParamEvol}, and coincides with the result of repeatedly applying Lemma~\ref{lem:mainLem} with $N_{(k)} = Z^2 (e_v^{1/2}/e_\vp^{1/2})_{(k)} (e_\vp/e_G)_{(k)}^{1/2}$ chosen as in Section~\ref{sec:mainLImpMThm}.
The sequence of frequency energy levels together with the initial $\bar{e}_{(0)}(t)$, $\tilde{I}_{[10]}$ and $\tilde{I}_{[G0]}$ determine the sequence of non-negative functions $\bar{e}_{(k)}(t)$ and intervals $\tilde{I}_{[1k]}, \tilde{I}_{[Gk]}$ obeying \eqref{eq:tildeIlp1}-\eqref{eq:suppContained} that arises from repeatedly applying Lemma~\ref{lem:mainLem} with the above choice of $N_{(k)}$.  

Letting $\tilde{I}_{(0)}$ be as in Theorem~\ref{thm:approxThm}, we set $t^* = \sup \tilde{I}_{(0)} + 2^{-1} (\Xi_0 E_0^{1/2})^{-1}$ and we define $\tilde{J}_{(k)}$  to be the open interval $(t^* - (\Xi_{(k)} e_{v,(k)}^{1/2})^{-1}, t^* + \Xi_{(k)} e_{v,(k)}^{1/2})$.  We note that $\tilde{J}_{(k)}$ is an open interval of length $|\tilde{J}_{(k)}| \geq (\Xi_{(k)} e_{v,(k)}^{1/2})^{-1}$ that is contained in $\tilde{I}_{[10]} \subseteq \tilde{I}_{[1k]}$ for all $k \geq 0$.

We now define the sequence of Euler-Reynolds flows $(v,p,R,\kk,\vp,\mu)_{\b,(k)}$ associated to $\b \in 2^{\N}$ as follows.  We let $(v,p,R,\kk,\vp,\mu)_{\b,(0)}$ be the dissipative Euler-Reynolds flow $(v,\tilde{p}, \wtld{R}, \tilde{\kk}, \tilde{\vp}, \tilde{\mu})_{(0)}$ in the conclusion of Lemma~\ref{lem:prepareInitLevels}.  We define $(v,p,R,\kk,\vp,\mu)_{\b,(k+1)}$ inductively by applying Lemma~\ref{lem:prepareInitLevels} to $(v,p,R,\kk,\vp,\mu)_{\b,(k)}$ with the choice of $N_{(k)} = Z^2 (e_v^{1/2}/e_\vp^{1/2})_{(k)} (e_\vp/e_G)_{(k)}^{1/2}$ and taking $\tilde{I}_{[1k]}, \tilde{I}_{[Gk]}$ and positive function $\bar{e}_{(k)}$ as above.  If $k + 1 \notin \b$ we choose $(v,p,R,\kk,\vp,\mu)_{\b,(k+1)} = (\ost{v}, \ost{p}, \ost{R}, \ost{\kk}, \ost{\vp}, \ost{\mu})_{\b,(k)}$, while if $k+1 \in \b$ we choose $(v,p,R,\kk,\vp,\mu)_{\b,(k+1)} = (v_2, p_2, R_2, \kk_2, \vp_2, \mu_2)_{\b,(k)}$, taking $\tilde{J}_{(k)} \subseteq \tilde{I}_{[1k]}$ as above.  In either case, the dissipation measure $\mu_{(k+1)}$ is the same and independent of $\b$.  Note that $N_{(k)}$ satisfies the admissibility conditions \eqref{eq:Nastdef} for all $k \geq 0$ by Proposition~\ref{prop:admissSect} since the initial frequency energy levels belong to the admissible truncated sector \eqref{eq:truncSect}.  This fact together with the $(\bar{\de},M)$-well-preparedness of $(v,p,R,\kk,\vp,\mu)_{\b,(k)}$ justifies the application of the Main Lemma.  

For each $\b \in 2^{\N}$, the sequence of velocity fields $v_{\b(k)}$ converges uniformly as $k \to \infty$ to a bounded, continuous vector field $v_\b$ on $\tilde{I} \times \T^3$, as follows from $\co{v_{\b(k+1)} - v_{\b(k)}} \leq C_L e_{\vp,(k)}^{1/2}$ and \eqref{eq:shrinkingAmplitudeLaw}.  Since $R_{\b(k)}$ converges uniformly to $0$ (again by \eqref{eq:shrinkingAmplitudeLaw}), we have that $(v_\b, p_\b)$ solve the incompressible Euler equations, where $p_\b \in \DD'(\tilde{I} \times \T^3)$ is equal to the weak limit $p_\b = \lim_{k \to \infty} \De^{-1}( \nb_\ell\nb_j (R_{\b(k)}^{j\ell} - v_{\b(k)}^j v_{\b(k)}^\ell)$.  Note that we may assume without loss of generality that the sequence of approximate pressures is equal to $p_{\b(k)} = \De^{-1}( \nb_\ell\nb_j (R_{\b(k)}^{j\ell} - v_{\b(k)}^j v_{\b(k)}^\ell)$ by adding a constant if necessary to replace each $p_{\b(k)}$ by its integral $0$ representative.  With this normalization, we have that the convergence $p_\b = \lim_{k \to \infty} p_{\b(k)}$ occurs strongly in $L^p$ for all $p < \infty$ by Calder\'{o}n-Zygmund estimates and the uniform convergence of $(R_{\b(k)}^{j\ell} - v_{\b(k)}^j v_{\b(k)}^\ell)$ (in fact the convergence holds also in H\"{o}lder spaces).  In particular $p_\b \in L^p$ for all $p < \infty$, making $\calD[v_\b, p_\b] := \pr_t(|v_\b|^2/2) + \nb_j[(|v_{\b}|^2/2 + p_{\b} ) v_{\b}^j ) ]$ a well-defined distribution for all $\b$.

We now verify that the limiting Euler flows $(v_\b, p_\b)$ must satisfy the local energy inequality \eqref{eq:locEnInq} for all $\b$ and all have the same dissipation measure $- \calD[v_\b, p_\b] = \mu_\infty = \lim_{k \to \infty} \mu_{(k)}$, which is equal to the weak limit $\mu_\infty = - \lim_{k \to \infty} [ \pr_t( |v_{\b(k)}|^2/2 ) + \nb_j( (|v_{\b(k)}|^2/2 + p_{\b(k)} ) v_{\b(k)}^j ) ]$.  The fact that $\calD[v_\b,p_\b]$ is equal to this weak limit follows from the strong $L^p$ convergence of the products including $p_{\b(k)} v_{\b(k)}^j$ as in the previous remark above.  The fact that the limit is non-negative as a distribution (and hence given by a measure), and is equal to $\mu_{\infty} = \lim_{k \to \infty} \mu_{(k)} \leq 0$ follows from \eqref{eq:dissipMeasdef} by observing that the functions on the right hand side of the relaxed local energy inequality \eqref{eq:relaxedLocEnIneq} all converge weakly to $0$ in $\DD'$ as $k \to \infty$.  Specifically, since the $v_{\b(k)}$ are uniformly bounded in $C^0$ and $(R_{\b(k)}, \kk_{\b(k)}, \vp_{\b(k)})$ converge to $0$ uniformly, the functions $D_t \kk_{\b(k)} := \pr_t\kk_{\b(k)} + \nb_j( v_{\b(k)}^j \kk_{\b(k)})$, $\nb_j[(v_{\b(k)})_\ell R_{\b(k)}^{j\ell}]$ and $\nb_j \vp_{\b(k)}^j$ in the relaxed energy inequality \eqref{eq:relaxedLocEnIneq},\eqref{eq:dissipMeasdef} all converge weakly to $0$ in $\DD'$ as $k \to \infty$.



The solutions constructed above all coincide on $\tilde{I}_{(0)} \times \T^3$.  To see this equality, note that (by induction on $k$ and the local dependence properties of Lemma~\ref{lem:mainLem}) the dissipative Euler-Reynolds flows $(v,p,R,\kk,\vp,\mu)_{\b,(k)}$ are all equal outside of the interval
\ALI{
\ovl{J}_{(k)} &:= \Big\{ t_0 + \bar{t} : t_0 \in \tilde{J}_{(0)}, |\bar{t}| \leq \sum_{j=1}^k (\Xi_{(j)} e_{v,(j)}^{1/2})^{-1} \Big\},
}
which does not intersect $\tilde{I}_{(0)}$ by \eqref{eq:shrinkingTimescaleLaw} when $Z$ is taken sufficiently large.

We claim that the map $\b \mapsto v_\b$ is injective when restricted to $2^{\N+2} = \{ \b \in 2^{\N} ~:~ \b \cap \{0,1\} = \emptyset \}$.  To see this fact, suppose $\b_1, \b_2 \in 2^{\N + 2}$ are distinct and let $k^* = \min \b_1 \De \b_2$ be the smallest integer that belongs to exactly one of $\b_1, \b_2$.  Then, letting $V_{\b(k)} = v_{\b(k+1)} - v_{\b(k)}$, we have by \eqref{eq:c0CorrectBd} and \eqref{eq:shrinkingAmplitudeLaw}
\ALI{
\| v_{\b_1} - v_{\b_2} \|_{C_t L_x^2} &\geq \| v_{\b_1(k^*)} - v_{\b_2(k^*)} \|_{C_t L_x^2} - \sum_{k > k^*} (\co{V_{\b_1(k)}} + \co{V_{\b_2(k)}}) \\
&\geq \| v_{\b_1(k^*)} - v_{\b_2(k^*)} \|_{C_t L_x^2} - \sum_{k \geq k^* + 1} 2 Z^{-(k-k_*)/4} C_L e_{\vp,(k^*)}^{1/2} \\
\| v_{\b_1} - v_{\b_2} \|_{C_t L_x^2} &\geq \|v_{\b_1(k^*)} - v_{\b_2(k^*)} \|_{C_t L_x^2} - 4C_L Z^{-1/4} e_{\vp,(k^*)}^{1/2}
}
From \eqref{eq:lowerBdNew} with $M = \hc$ we have $\|v_{\b_1(k^*)} - v_{\b_2(k^*)} \|_{C_t L_x^2} \geq \big( (C_L \hc)^{-1} - N_{(k^*)}^{-1}\big)^{1/2}e_{\vp,(k^*)}^{1/2}$.  Using that $N_{(k^*)} \geq Z$ and taking $Z$ sufficiently large we obtain
\ali{
\| v_{\b_1} - v_{\b_2} \|_{C_t L_x^2} &\geq 2^{-1} (C_L \hc)^{-1/2} e_{\vp,(k^*)}^{1/2}, \qquad k^* = \min(\b_1 \De \b_2). \label{eq:injectBd}
}
It follows that $v_{\b_1} \neq v_{\b_2}$ for all $\b_1 \neq \b_2$ in $2^{\N + 2}$; i.e., the map $\b \mapsto v_\b$ is injective.

We claim furthermore that, if $\a < 1/15$ and $Z \geq \undl{Z}_\a$ is sufficiently large, then $\b \mapsto v_\b$ is a continuous map from $2^{\N}$ into $C_{t,x}^\a$.  We first note that for all $k$ such that $e_{v,(k)} \leq 1$ we have the bound
\ali{
\co{\nb v_{\b(k)}} + \co{\pr_t v_{\b(k)} } &\leq \co{\nb v_{\b(k)}} + \co{D_t v_{\b(k)} } + \co{v_{\b(k)}}\co{\nb v_{\b(k)}} \notag \\
&\leq C \co{\nb v_{\b(k)}} + \co{\nb p_{\b(k)}} + \co{\nb R_{\b(k)}} \notag \\
&\leq C (\Xi e_{v,(k)}^{1/2} + \Xi e_{v,(k)} ) \leq C \Xi_{(k)} e_{v,(k)}^{1/2},
}
where $C$ depends on $\sup_{\b, k} \co{v_{\b(k)}}$.  Using \eqref{eq:shrinkingTimescaleLaw}, we bound $V_{\b(k)} = v_{\b(k+1)}-v_{\b(k)}$ for large $k$ by
\ALI{
\co{V_{\b(k)}} &\leq C_L e_{\vp,(k)}^{1/2} \\
\co{\nb V_{\b(k)}} + \co{\pr_t V_{\b(k)}} &\leq C \Xi_{(k+1)} e_{v,(k+1)}^{1/2} \leq C Z^9 \Xi_{(k)} e_{v,(k)}^{1/2} .
}
Interpolating, we obtain that $\| V_{\b(k)} \|_{C_{t,x}^\a} \leq C_Z (\Xi e_v^{1/2}/e_\vp^{1/2})^\a e_{\vp(k)}^{1/2}$ for some constant $C_Z$ depending on $Z$ and the $C$ above.  We let $H_{\a(k)}:= (\Xi e_v^{1/2}/e_\vp^{1/2})^\a e_{\vp(k)}^{1/2}$, which we estimate as follows using the notation $L_{(k)} = [L_Z, L_G, L_{\undvp G}, L_{\vp \undvp}, L_{v \vp}, L_\Xi]_{(k)}^t$ and $\mrg{L}_{(k)} = [L_Z, L_{\undvp G}, L_{\vp \undvp}, L_{v \vp}]_{(k)}^t$, of Section~\ref{sec:admissConditions}
\ALI{
\log H_{\a(k)} &= (1/2) [ L_G +  L_{\undvp G} + L_{\vp \undvp} + (1+\a) L_{v\vp} ] + \a L_\Xi \\
\de_{(k)} \log H_{\a(k)} &= (1/2) \de_{(k)} L_G + \a \de_{(k)} L_\Xi + O(|\de_{(k)}\mrg{L}_{(k)}|) \\
&= -(1/15 - \a) \de_{(k)} L_\Xi + (1/2) \de_{(k)} L_G + (1/15) \de_{(k)} L_\Xi +  O(|\de_{(k)}\mrg{L}_{(k)}|) \\
\de_{(k)}\log H_{\a(k)} &= -(1/15 - \a) \de_{(k)} L_\Xi + (1/30)[0,15,11,5,1,2]\de_{(k)}L_{(k)} + O(|\de_{(k)} \mrg{L}_{(k)}|) \\
&= -(1/15 - \a) \de_{(k)} L_\Xi + (13/30)L_{\hc} + o(1), \qquad \mbox{ as } k \to \infty.
}
In the last line we used equation \eqref{eq:evolRule} and the fact from Proposition~\ref{prop:changingQuotients} that $|\de_{(k)} \mrg{L}_{(k)}| \to 0$ as $k \to \infty$.  Since $- \de_{(k)} L_\Xi \leq -L_Z$ and $\a < 1/15$, we have for all $Z \geq \undl{Z}_\a$ sufficiently large (depending on $\hc$ and $\a$) that $\de_{(k)} \log H_{\a(k)} \leq - \log 10 + o(1)$ as $k \to \infty$.  In particular, there exists $k_0 \in \N$ such that for all $k \geq k_0$ we have $H_{\a(k)} \leq 9^{-(k - k_0)} H_{\a(k_0)}$.  This estimate implies that every vector field $v_\b$ is of class $C_{t,x}^\a$, and also that the map $\b \mapsto v_\b$ from $2^\N$ to $C_{t,x}^\a$ is continuous, since it implies the estimates
\ali{
\| v_\b \|_{C_{t,x}^\a} \leq \| v_0 \|_{C_{t,x}^\a} &+ \sum_{k \geq 0} \| V_{\b(k)} \|_{C_{t,x}^\a} \leq \| v_0 \|_{C_{t,x}^\a} + \sum_{k \geq 0} C_Z H_{\a(k)} < \infty \notag \\
\| v_{\b_1} - v_{\b_2} \|_{C_{t,x}^\a} &\leq \sum_{k \geq \min(\b_1 \De \b_2)} \| V_{\b(k)} \|_{C_{t,x}^\a} \leq \sum_{k \geq \min(\b_1 \De \b_2)} C_Z H_{\a(k)}. \label{eq:CtxalphCtinuity}
}
Since \eqref{eq:CtxalphCtinuity} goes to $0$ as $\min(\b_1 \De \b_2)$ tends to $\infty$, we obtain continuity of the map $\b \mapsto v_\b$ by \eqref{eq:convergenceCharacterization}.  

Since $2^{\N+2}$ is a compact topological space (homeomorphic to a Cantor set), $C_{t,x}^\a$ is Hausdorff, and we have shown that $\b \mapsto v_\b$ is injective on $2^{\N+2}$, it follows that the map $\b \mapsto v_\b$ is a homeomorphism onto its image in $C_{t,x}^\a$ when restricted to $2^{\N + 2}$.  Identifying $2^{\N + 2}$ with $2^\N$ defines the map $v_\b$ of Theorem~\ref{thm:approxThm}.

Now fix a particular value of $Z \geq \undl{Z}_\a$ that is sufficiently large to guarantee all of the preceding estimates.  We claim that the image of $(v_\b)_{\b \in 2^{\N + 2}} $ restricted to $2^{\N + 2}$ has positive Hausdorff dimension as a subspace of $C_t L_x^2$.  To establish this claim, we show that there exist positive numbers $\de = \de_Z > 0$, $\ep_0 > 0$ and $a_0 > 0$ such that whenever there is a covering $\{ v_\b : \b \in 2^{\N + 2} \} \subseteq \bigcup_{i \in I} B_{r_i}(v_{\b_i})$ of the image by open balls of radius $r_i \leq \ep_0$ in $C_t L_x^2$ centered at points in the image, we have a lower bound of
\ali{
0 < a_0 &\leq \sum_{i \in I} r_i^{\de_Z}. \label{eq:dimensionBound}
}
We may assume without loss of generality that the collection of balls $I$ in the covering is countable.  The estimate \eqref{eq:dimensionBound} then shows that the Hausdorff dimension of the space $\{ v_\b : \b \in 2^{\N+2} \}$ (viewed as a subspace of $C_tL_x^2$) is at least $\de_Z$.

   We recall the notation $[k] := \{ 0, 1, \ldots, k \}$ and we introduce the ``basic cylinders'' defined by
\ALI{
2^{\N + k} := \{ \b \in 2^\N : \b \cap [k - 1] = \emptyset \}, \quad \bar{\b} \De 2^{\N + k} := \{ \b \in 2^\N ~:~ (\b \De \bar{\b}) \cap [k-1] = \emptyset \}, \quad k \geq 1, \bar{\b} \in 2^\N.
}
For any $k \geq 1$, the set $2^\N$ is the union of $2^k$ disjoint basic cylinders $2^\N = \bigcup_{\bar{\b} \subseteq [k-1]} \bar{\b} \De 2^{\N + k}$.  Moreover, every basic cylinder has the form $\bar{\b} \De 2^{\N + k} = (\bar{\b} \cap [k-1]) \De 2^{\N + k}$, where $\bar{\b} \cap [k-1]$ is finite.	 Every basic cylinder is therefore both closed (by \eqref{eq:convergenceCharacterization}) and open (as its complement is closed), and its Haar measure is equal to $\mu(\bar{\b} \De 2^{\N + k} ) = \mu((\bar{\b} \cap [k-1]) \De 2^{\N+k}) = \mu(2^{\N + k}) = 2^{-k}$ when we endow $2^\N$ with the natural Haar measure $\mu$ of unit total mass.  The collection of all basic cylinders $\bar{\b} \De 2^{\N + k}$, where $\bar{\b}$ is finite and $k \geq 1$, can be visualized as forming a countable basis for the topology of $2^\N$.  

Suppose now that we have a covering $\{ v_\b ~:~ \b \in 2^{\N + 2} \} \subseteq \bigcup_{i \in I} B_{r_i}(v_{\b_i}) \subseteq C_tL_x^2$ with $\max_i r_i \leq \ep_0$, where $\ep_0 := 4^{-1} (C_L \hc)^{-1/2} e_{\vp,(0)}^{1/2}$.  For each $i \in I$ define $k^*_i \in \N$ to be the largest integer that satisfies %
$r_i \leq 4^{-1} (C_L \hc)^{-1/2} e_{\vp,(k^*_i)}^{1/2}$.
  By the lower bound \eqref{eq:injectBd}, every $v_\b \in B_{r_i}(v_{\b_i})$ satisfies $\min(\b \De \b_i) \geq k^*_i$, or equivalently belongs to the basic cylinder $\b \in \b_i \De 2^{\N + k^*_i}$.  We therefore obtain the inequalities
\ali{
1_{2^{\N + 2}}(\b) &\leq \sum_{i \in I} 1_{B_{r_i}(v_{\b_i})}(\b) \leq \sum_{i \in I} 1_{\b_i \De 2^{\N + k^*_i}}(\b),  \notag \\
2^{-2} = \int_{2^\N} 1_{2^{\N + 2}} \, \tx{d}\mu &\leq \sum_{i \in I} \mu(\b_i \De 2^{\N + k^*_i}) = \sum_{i \in I} 2^{-k^*_i}. \label{eq:lowBoundkstar}
}
On the other hand, by the choice of $k^*_i$ and \eqref{eq:shrinkingAmplitudeLaw} we have also a lower bound of
\ali{
r_i &\geq c e_{\vp,(k^*_i + 1)}^{1/2} \geq c Z^{-3\left(k_i^* + 1\right)/4} e_{\vp,(0)}^{1/2} \qquad \mbox{ for all } i \in I, \label{eq:upBoundkstar}
}
where $c := 4^{-1} (C_L \hc)^{-1/2} > 0$.  Combining \eqref{eq:lowBoundkstar} and \eqref{eq:upBoundkstar}, we obtain the desired bound \eqref{eq:dimensionBound} with $\de_Z = \fr{4 \log(2)}{3\log(Z)}$ for some $a_0 > 0$ that depends explicitly on $C_L, \hc, e_{\vp,(0)}$ and $Z$.  

We remark that the calculation of line~\ref{eq:lowBoundkstar} can be viewed as a purely combinatorial statement taking place on a finite quotient of $2^\N$ when the collection of balls $I$ is finite, but relies on the use of measure theory to handle the case of a countably infinite $I$.  


The last bound we require is that $\sup_{\b \in 2^{\N+2}} \co{v_\b - v_0} \leq \ovl{C}_\a E_0^{1/2}$ with $\ovl{C}_\a$ depending on $\a$.  This bound follows from the choice of $e_{\vp,(0)}^{1/2} \leq C_\a E_0^{1/2}$ in \eqref{eq:bigInitFren}, the bound $\co{V_{\b(k)}} \leq C_L e_{\vp,(k)}^{1/2}$ in \eqref{eq:c0CorrectBd} for $V_{\b(k)} = v_{\b(k+1)} - v_{\b(k)}$, and inequality \eqref{eq:shrinkingAmplitudeLaw} for $e_{\vp,(k)}$.  With this estimate we have concluded the proof of Theorem~\ref{thm:approxThm}.
\end{proof}

\subsection{Density of Wild Initial Data for Conservative Solutions} \label{sec:constructDissipSolns}
As the final application of Theorem~\ref{thm:approxThm} that we consider here, we state the following theorem, which proves the density of ``wild'' initial data of class $C^\a$ for $\a < 1/15$.  The theorem is in the spirit of \cite{danSze}, which extended a similar theorem of \cite{szeWiedYoung}, where an analogous $L^2$ density result  is obtained for solutions with kinetic energy below that of the initial data and having H\"{o}lder regularity $\a < 1/5$.   Our approach to Theorem~\ref{thm:wildDenseThm} turns out to be considerably simpler than the approach to the analogous result in \cite{danSze} (in particular we avoid the use of time-dependent estimates and the notion of an adapted subsolution), but at the expense of having a nonuniform time interval of existence for the corresponding solutions.    
  
\begin{thm}[Density of wild initial data] \label{thm:wildDenseThm} For any $\a < 1/15$, there is a set of vector fields $F_\a$ on $\T^3$ such that the $C^0$ closure of $F_\a$ consists of all continuous, divergence free vector fields $\bar{v}_0 \colon \T^3 \to \R^3$, and such that for every $v_0 \in F_\a$ there exists an open interval $I$ containing $0$ and uncountably many vector fields $(v_{\b})_{\b \in 2^\N}$ of class $C_{t,x}^\a(I \times \T^3)$ that satisfy the incompressible Euler equations on $I \times \T^3$ with initial data $v_\b(0,x) = v_0(x)$ for all $\b \in 2^\N$ and satisfy the local energy equality \eqref{eq:locEnEq} on $I \times \T^3$.
\end{thm}
\begin{proof}  Let $\bar{v}_0 : \T^3 \to \R^3$ be a smooth, divergence free vector field, and let $\tilde{I}$ be a bounded open interval containing $0$ such that a unique, smooth solution to the Euler equations $(\bar{v}, \bar{p})$ having initial data $\bar{v}(0,x) = \bar{v}_0$ exists on the interval $2 \tilde{I} \times \T^3$.  Let $E_0 > 0$ be any positive number and choose $\Xi_0$ large enough such that the dissipative Euler-Reynolds flow defined by $(\bar{v}, \bar{p}, R, \kk, \vp, \mu)_0$ with zero error terms $(R, \kk, \vp, \mu)_0 = (0,0,0,0)$ has compound frequency energy levels bounded by $(\Xi_0, E_0, E_0, E_0, E_0)$ to order $2$ in $C^0$.  (Such a choice is possible since all the error terms are $0$ and since $(\bar{v}, \bar{p})$ and their partial derivatives are uniformly bounded on $\tilde{I} \times \T^3$.) 

Applying Theorem~\ref{thm:approxThm} with $\tilde{I}_{(0)} = \{t \leq (1/2) \sup \tilde{I} \}$ and $\tilde{I}_{(1)} = \tilde{I}$ (taking $\Xi_0$ larger if necessary), we obtain an uncountable family of weak solutions $(v_\b, p_\b)_{\b \in 2^\N}$ in the class $(v_\b)\in C_{t,x}^\a$ that all share the same initial data (being equal on $\tilde{I}_{(0)}$) and share a uniform bound of $\sup_{\b \in 2^\N} \| v_\b - \bar{v} \|_{C^0_{t,x}} \leq C_\a E_0^{1/2}$.  These solutions also satisfy the local energy equality \eqref{eq:locEnEq} on $\tilde{I}$ by our choice of $\tilde{I}_{(1)} = \tilde{I}$ and $\mu_0 = 0$.  

We now take $F_\a$ to be the union of all initial data arising from the above construction over all choices of smooth divergence free $\bar{v}_0$ on $\T^3$ and $E_0 > 0$.  It is clear that the $C^0$ closure of $F_\a$ contains all smooth, divergence free $\bar{v}_0$, since $E_0 > 0$ is arbitrary and the common initial data of the family $v_\b$ constructed above converges uniformly to $\bar{v}_0$ as $E_0$ tends to $0$.  By approximation, every continuous, divergence free $\bar{v}_0$ on $\T^3$ belongs to the $C^0$ closure of $F_\a$.  The other properties stated for $F_\a$ are immediate from the construction.  
\end{proof}

\bibliographystyle{alpha}
\bibliography{eulerOnRn}

\end{document}